\documentclass[11pt,english]{article} 
\pdfoutput=1
\usepackage{geometry}
\usepackage[T1]{fontenc}
\usepackage[utf8]{inputenc}
\usepackage[english]{babel}
\usepackage{latexsym}
\usepackage{mathrsfs}
\usepackage{enumerate}
\usepackage{amsmath}
\usepackage{amssymb}
\usepackage{mathtools}
\usepackage{braket}
\usepackage{amsthm}
\usepackage{graphicx} 
\usepackage[dvips]{psfrag}
\usepackage{tikz}
\usepackage{tikz-cd}
\usepackage{frontespizio}
\usepackage{shellesc}
\usepackage{quoting}
\usepackage{url}
\usepackage{tabularx}
\usepackage{booktabs}
\usepackage{spectralsequences}

\usepackage[backend=biber, style=numeric, bibstyle=numeric,doi=false,isbn=false,url=false,eprint=false]{biblatex}
\usepackage[autostyle,italian=guillemets]{csquotes}

\theoremstyle{plain} 
\newtheorem{thm}{Theorem}[section] 
\newtheorem{corollario}[thm]{Corollary} 
\newtheorem{lem}[thm]{Lemma} 
\newtheorem{prop}[thm]{Proposition}

\theoremstyle{definition} 
\newtheorem*{defn}{Definition}
\newtheorem*{es}{Example}

\newtheorem{oss}[thm]{Remark} 

\theoremstyle{remark}

\newcommand{\M}{\mathcal{M}}
\newcommand{\numberset}{\mathbb}
\newcommand{\N}{\numberset{N}}
\newcommand{\F}{\numberset{F}}
\newcommand{\R}{\numberset{R}}
\newcommand{\C}{\numberset{C}}
\newcommand{\Z}{\numberset{Z}}

\DeclarePairedDelimiter{\abs}{\lvert}{\rvert}

\newsavebox{\pullback}
\sbox\pullback{%
\begin{tikzpicture}%
\draw (0,0) -- (1ex,0ex);%
\draw (1ex,0ex) -- (1ex,1ex);%
\end{tikzpicture}}

\newsavebox{\pushout}
\sbox\pushout{%
\begin{tikzpicture}%
\draw (0,0) -- (0ex,1ex);%
\draw (0ex,1ex) -- (1ex,1ex);%
\end{tikzpicture}}

\newsavebox{\hpushout}
\sbox\hpushout{%
\begin{tikzpicture}%
\draw (0ex,2ex) -- (0ex,3ex);%
\draw (0ex,3ex) -- (1ex,3ex);%
\draw (-2ex,5ex) node [anchor=north west][inner sep=0.75pt]  {$h$};

\end{tikzpicture}}

\title{On the topology of $\M_{0,n+1}/\Sigma_n$}
\date{}
\author{Tommaso Rossi\footnote{ tommaso.rossi118@gmail.com. This work was funded by the PhD program of the University of Roma Tor Vergata and by the MIUR Excellence Department Project awarded to the Department of Mathematics,
University of Roma Tor Vergata, CUP E83C18000100006.}\\ 
\footnotesize Dipartimento di Matematica, Università di Roma Tor Vergata}

\begin{document}

\maketitle
\begin{abstract}
This paper contains some results about the topology of $\M_{0,n+1}/\Sigma_n$, where $\M_{0,n+1}$ is the moduli space of genus zero Riemann surfaces with marked points. We show that $\M_{0,n+1}/\Sigma_n$ is not a topological manifold for $n\geq 4$, and it is simply connected for any $n\in\N$. We also present some homology computations: for example we show that $\M_{0,p+1}/\Sigma_p$ has no $p$ torsion, where $p$ is a prime. Lastly we compute $H_*(\M_{0,n+1}/\Sigma_n;\Z)$ for small values of $n$, proving that $\M_{0,n+1}/\Sigma_n$ is contractible for $n\leq 5$ while $\M_{0,7}/\Sigma_6$ is not. 
\end{abstract}

\tableofcontents

 \section{Introduction}
Let $\M_{0,n+1}$ be the moduli space of genus zero Riemann surfaces with $n+1$ marked points. In this paper we study the topology of the quotients $\mathcal{M}_{0,n+1}/\Sigma_n$. Rationally $\mathcal{M}_{0,n+1}/\Sigma_n$ has the homology of the point (as one can deduce from \cite{Getzler}), so its homology is composed by torsion classes only. In this paper we get some results on the torsion of $H_*(\mathcal{M}_{0,n+1}/\Sigma_n;\Z)$, together with other topological informations about the fundamental group and the orbifold structure. The main results are the following:
\begin{itemize}
    \item $\mathcal{M}_{0,n+1}/\Sigma_n$ is not a topological manifold for $n\geq 4$ (Theorem \ref{thm: quozienti stretti non sono varietà topologiche}).
    \item $\overline{\mathcal{M}}_{0,n+1}/\Sigma_n$ and $\mathcal{M}_{0,n+1}/\Sigma_n$ are simply connected (Theorems \ref{thm: compattificati sono semplicemente connessi} and \ref{thm: quozienti stretti sono semplicemente connessi}).
    \item $\mathcal{M}_{0,p+1}/\Sigma_p$ and $\mathcal{M}_{0,p+2}/\Sigma_{p+1}$ have no $p$-torsion (Theorem \ref{thm: quoziente stretto di p punti non ha omologia modulo p}), where $p$ is a prime number.
    \item $\mathcal{M}_{0,n+1}/\Sigma_n$ is contractible for $n\leq 5$ ( Paragraph \ref{sec:examples}).
\end{itemize}
Here is the outline of the paper:
\begin{description}
    \item[Section \ref{sec:topology of strict quotients}] presents an embedding of $\M_{0,n+1}/\Sigma_n$ into the weighted projective space $\mathbb{P}(n,n-1,\dots,2)$ as an open dense subset. This is crucial to prove that $\mathcal{M}_{0,n+1}/\Sigma_n$ is not a topological manifold if $n\geq 4$ (Theorem \ref{thm: quozienti stretti non sono varietà topologiche}).
    \item[Section \ref{sec: a combinatorial model using cacti}] contains a cambinatorial model for $\M_{0,n+1}/\Sigma_n$ based on cacti. This is useful to do computations when $n$ is small. 
    \item[Section \ref{sec:fundamental group and rational homology}] deals with the computation of the fundamental group of $\mathcal{M}_{0,n+1}/\Sigma_n$ and $\overline{\mathcal{M}}_{0,n+1}/\Sigma_n$.
    \item[Section \ref{sec:torsion of strict quotients}] contains some results about $H_*(\M_{0,n+1}/\Sigma_n;\F_p)$, where $p$ is a prime. We give an upper bound for the order of a class in $H_*(\M_{0,n+1}/\Sigma_n;\Z)$ (Theorem \ref{thm: bound sulla torsione dei quozienti}). Then we construct a long exact sequence (Proposition \ref{prop: Mayer-Vietoris}) that expresses $H^*(X/S^1;\F_p)$ in terms of $H^*_{S^1}(X;\F_p)$, $H^*_{S^1}(X^{\Z/p};\F_p)$ and $H^*(X^{\Z/p}/S^1;\F_p)$, where $X$ is any suitably nice $S^1$-space. This sequence is an easy consequence of classical facts about transformation groups. Since $\M_{0,n+1}/\Sigma_n$ turns out to be homotopy equivalent to $C_n(\C)/S^1$, where $C_n(\C)$ is the unordered configuration space of point in the plane, we will use the long exact sequence of Proposition \ref{prop: Mayer-Vietoris} to get some information about $H_*(\M_{0,n+1}/\Sigma_n;\F_p)\cong H_*(C_n(\C)/S^1;\F_p)$. In particular we compute $H_*(\M_{0,n+1}/\Sigma_n;\F_p)$ when $n\neq 0,1$ mod $p$, and we prove that there is no $p$ torsion in $\M_{0,p+1}/\Sigma_p$ and $\M_{0,p+2}/\Sigma_{p+1}$ . We also compute $H_*(\M_{0,n+1}/\Sigma_n;\Z)$ for small values of $n$, and this proves that $\M_{0,n+1}/\Sigma_n$ is contractible for any $n\leq 5$, while it has non trivial homology when $n=6$. 
    \item [Appendix \ref{app: omologia spazi di configurazione etichettati }] is just a recollection on the homology of $C_n(\C^*)$, which is a consequence of the work of F. Cohen \cite{Cohen}. These results are used in Section \ref{sec:torsion of strict quotients}.
\end{description}
\paragraph{Notations and conventions:} if $X$ is a topological space we will indicate by $F_n(X)$ (resp. $C_n(X)$) the ordered (resp. unordered) configuration space.
\paragraph{Acknowledgments:} This project is part of the author's PhD thesis, funded by the PhD program of the University of Roma Tor Vergata and by the MIUR Excellence Project MatMod@TOV awarded to the
Department of Mathematics, University of Roma Tor Vergata,  CUP E83C18000100006. I can not omit the invaluable help that my advisor Paolo Salvatore gave to me during the PhD. I am grateful to him for the patience and kindness he showed every time we discussed about this project, and for sharing with me many good ideas about it. Without his guidance and support this paper would have not be possible. I am also grateful to the anonymous referee for many suggestions that improved the readability of this paper.

\section{Orbifold structure of $\mathcal{M}_{0,n+1}/\Sigma_n$}\label{sec:topology of strict quotients}
In this Section we show that $\M_{0,n+1}/\Sigma_n$ is not a topological manifold for $n\geq 4$. For completeness we also discuss the case $n=3$. 
	\begin{prop}
		$\M_{0,4}/\Sigma_3$ is homeomorphic to $S^2-\{pt\}$.
	\end{prop}
	\begin{proof}
	Consider the Deligne-Mumford compactification $\overline{\M}_{0,4}$. It is a compact Riemann surface, which is well known to be homeomorphic to $S^2$. Since $\Sigma_3$ acts faithfully and by biholomorphisms we get that the quotient $\overline{\M}_{0,4}/\Sigma_3$ is again a Riemann surface. By the Riemann-Hurwitz formula $\overline{\M}_{0,4}/\Sigma_3$ has genus zero, so it is homeomorphic to $S^2$. Now observe that $\M_{0,4}$ is $\overline{\M}_{0,4}$ minus three points (corresponding to the three stable curves). These points are all identified in $\overline{\M}_{0,4}/\Sigma_3$, so $\M_{0,4}/\Sigma_3$ is $\overline{\M}_{0,4}/\Sigma_3$ minus a point, as claimed.
	\end{proof}
	We now focus on $\M_{0,n+1}/\Sigma_n$ when $n\geq 4$.
\begin{oss}
	$\M_{0,n+1}$ is homeomorphic to $F_n(\C)/\C\rtimes \C^*$: a point in $\M_{0,n+1}$ is just a configuration of points $(p_0,p_1,\dots,p_n)$ in the Riemann sphere up to biholomorphisms. Up to rotations we can suppose that $p_0$ is the point at the infinity $\infty\in \C\cup\{\infty\}$. Deleting this point and using the stereographic projection we obtain a configuration of $n$ points in the complex plane $\numberset{C}$ up to translation, dilatations and rotations, as claimed. From now on we will identify freely $\M_{0,n+1}$ with $F_n(\C)/\C\rtimes \C^*$. 
\end{oss} 
\begin{prop}\label{prop:quotient as unordered configuration}
There is a homeomorphism 
\begin{align*}
    \phi:\frac{\mathcal{M}_{0,n+1}}{\Sigma_n}&\to\frac{C_n(\C)}{\C\rtimes\C^*}\\
  [[z_1,\dots,z_n]]&\mapsto[\{z_1,\dots,z_n\}]
\end{align*}
  
where $[[z_1.\dots,z_n]]$ is the class associated to $[(z_1,\dots,z_n)]\in \mathcal{M}_{0,n+1}\cong F_n(\C)/\C\rtimes\C^*$.
\end{prop}
\begin{proof}
$\Sigma_n$ acts on $F_n(\C)$ by permuting the coordinates, and this action commutes with that of $\C\rtimes \C^*$ by translations, rotations and dilations. Therefore the quotients $(F_n(\C)/\C\rtimes \C^*)/\Sigma_n=\M_{0,n+1}/\Sigma_n$ and $(F_n(\C)/\Sigma_n)/\C\rtimes \C^*=C_n(\C)/\Sigma_n$ are homeomorphic.

\end{proof}
\begin{oss}\label{oss: trucco dei prodotti simmetrici}
The unordered configuration space is naturally a subspace of the symmetric power $SP^n(\C)$ which is homeomorphic to $\C^n$. The homeomorphism maps an unordered $n$-uple $\{z_1,\dots,z_n\}$ to the coefficients $(a_0,\dots,a_{n-1})$ of the unique monic polinomial $a_0+a_1z+\dots+a_{n-1}z^{n-1}+z^n=(z-z_1)\cdots(z-z_n)$ which have $\{z_1,\dots,z_n\}$ as roots. 
\end{oss}

\begin{defn}
If $\mathbf{z}\coloneqq(z_1,\dots,z_n)$ is a $n$-uple of complex numbers, its \textbf{barycenter} is 
\[
B(\mathbf{z})\coloneqq \frac{z_1+\dots+z_n}{n}
\]
\end{defn}
Before we get into the topology of $\M_{0,n+1}/\Sigma_n$ we recall the definitions of weighted projective space and lens complex, just to fix the notation.
\begin{defn}
    Let $(b_0,\dots,b_n)$ be a $(n+1)$-tuple of positive integers. The \textbf{weighted projective space} (of weights $(b_0,\dots,b_n)$) is defined as 
    \[
    \mathbb{P}(b_0,\dots,b_n)\coloneqq \C^{n+1}-\{0\}/\sim
    \]
    where $(z_0,\dots,z_n)\sim (t^{b_0}z_0,\dots,t^{b_n}z_n)$ for any $t\in\C^*$.
\end{defn}
\begin{defn}
    Let $(b_0,\dots,b_n)$ be a $(n+1)$-tuple of positive integers and $\zeta\coloneqq e^{2\pi i/b_n}$. The \textbf{lens complex} is defined as  
    \[
    L(b_n;b_0,\dots,b_{n-1})\coloneqq S^{2n-1}/\sim
    \]
    where $(z_0,\dots,z_{n-1})\sim (\zeta^{b_0}z_0,\dots,\zeta^{b_{n-1}}z_{n-1})$.
\end{defn}
\begin{prop}\label{prop:omeo con il proiettivo pesato}
There is an homeomorphism 
\begin{align*}
  \psi:\frac{SP^n(\C)-\Delta}{\C\rtimes \C^*}&\to \mathbb{P}(n,n-1,\dots,2)  \\
  [\{z_1,\dots,z_n\}]&\mapsto [a_0:\dots:a_{n-2}]
\end{align*}
where $(a_0,\dots,a_{n-2})$ are the coefficients of  $(z-z_1+B(\mathbf{z}))\cdots(z-z_n+B(\mathbf{z}))$ and $\Delta\coloneqq\{\{z,\dots,z\}\in SP^n(\C)\mid z\in\C\}$ is the diagonal of $SP^n(\C)$.

\end{prop}
\begin{oss}\label{oss: coefficiente 0} 
The coefficient $a_{n-1}$ of $(z-z_1+B(\mathbf{z}))\cdots(z-z_n+B(\mathbf{z}))$ is $0$, indeed
 \begin{align*}
     a_{n-1}&=-z_1+B(\mathbf{z})+\dots-z_n+B(\mathbf{z})=-(z_1+\dots+z_n)+nB(\mathbf{z})=0
 \end{align*}
\end{oss}
\begin{proof}
Consider the map
\begin{align*}
    f:SP^n(\C)&\to SP^n(\C)\\
    \{z_1,\dots,z_n\}&\mapsto \{z_1-B(\mathbf{z}),\dots,z_n-B(\mathbf{z})\}
\end{align*}
whose image consists of the unordered $n$-uples whose barycenter is the origin. By Remark \ref{oss: trucco dei prodotti simmetrici} there is an homeomorphism $g:SP^n(\C)\to\C^n$ which sends a set of $n$ complex numbers $\{z_1,\dots,z_n\}$ to the coefficients $(a_0,\dots,a_{n-1})$ of the monic polynomial $(z-z_1)\cdots(z-z_n)$. The composite $g\circ f$ maps $\{z_1,\dots,z_n\}$ to the coefficients of the monic polynomial $(z-z_1+B(\mathbf{z}))\cdots(z-z_n+B(\mathbf{z}))$. By Remark \ref{oss: coefficiente 0} the image of $g\circ f$ is $\{(a_0,\dots,a_{n-2},0)\mid a_i\in\C\}\cong\C^{n-1}$ therefore we get a map
\[
SP^n(\C)-\Delta \xrightarrow{g\circ f} \C^{n-1}-\{0\}\to \mathbb{P}(n,n-1,\dots,2)
\]
where the last map is the natural projection $\C^{n-1}-\{0\}\to\mathbb{P}(n,n-1,\dots,2)$. It is easy to show that this map is constant on the $\C\rtimes\C^*$ orbits, so it induces a bijective map 
\[
  \psi:\frac{SP^n(\C)-\Delta}{\C\rtimes \C^*}\to \mathbb{P}(n,n-1,\dots,2)  
\]
which is an homeomorphism because the source is compact and the target is Hausdorff.
\end{proof}
\begin{corollario}\label{cor:omeo con il proiettivo pesato}
 The composition $\psi\circ \phi:\M_{0,n+1}/\Sigma_n\to \mathbb{P}(n,n-1,\dots,2)$ is an embedding.
\end{corollario}
\begin{proof}
Just combine Proposition \ref{prop:quotient as unordered configuration} and Proposition \ref{prop:omeo con il proiettivo pesato}.
\end{proof}

\begin{thm}\label{thm: quozienti stretti non sono varietà topologiche}
$\mathcal{M}_{0,n+1}/\Sigma_n$ is not a topological manifold for any $n\geq 4$. 
\end{thm}
\begin{proof}
Let us consider the point $p\in\mathcal{M}_{0,n+1}/\Sigma_n$ consisting of the roots of the polynomial $z^n+1$. We show that $p$ does not have an Euclidean neighborhood. By the embedding of Corollary \ref{cor:omeo con il proiettivo pesato}, this corresponds to the point $[1:0:\dots:0]$ of $\mathbb{P}(n,n-1,\dots,2)$. Since $\mathcal{M}_{0,n+1}/\Sigma_n$ is an open subset of $\mathbb{P}(n,n-1,\dots,2)$ our claim is equivalent to show that $p$ does not have an Euclidean neighborhood contained in $\mathbb{P}(n,n-1,\dots,2)$. Consider the open set 
\[
    U_0\coloneqq\{[1:a_1:\dots:a_{n-2}]\in\mathbb{P}(n,n-1\dots,2)\}
\]
If $[1:a_1:\dots:a_{n-2}]=[1:b_1:\dots:b_{n-2}]$ then exists $\lambda\in\C^*$ such that 
\[
\begin{cases}
\lambda^n=1\\
\lambda^{n-k}a_k=b_k \text{ for } k=1,\dots,n-2
\end{cases}
\]
Therefore $U_0\cong \C^{n-2}/\sim$
where $(a_1,\dots,a_{n-2})\sim(b_1,\dots,b_{n-2})$ if and only if there is a $n$-th root of unity $\lambda$ such that $\lambda^{n-k}a_k=b_k$ for all $k=1,\dots,n-2$. Therefore $U_0$ is homeomorphic to a cone on the lens complex $L(n;n-1,\dots,2)$. The point $p=[1,0,\dots,0]$ is precisely the vertex of this cone, therefore (by excision) we have 
\begin{align*}
    H_k(\mathbb{P}(n,n-1,\dots,2),\mathbb{P}(n,n-1,\dots,2)-\{p\})&=H_k(U_0,U_0-\{p\})\\
    &=H_{k-1}(U_0-\{p\}) \\
    &=H_{k-1}(L(n;n-1,\dots,2))
\end{align*}
The lens complex $L(n;n-1,\dots,2)$ is not a homology sphere if $n\geq 4$ (Kawasaki, \cite{Kawasaki}), therefore $p$ can not have a Euclidean neighborhood. 
\end{proof}

	\section{A combinatorial model using cacti}\label{sec: a combinatorial model using cacti}
The main goal of this section is to provide a combinatorial model for $\M_{0,n+1}/\Sigma_n$. The key observation is that $\M_{0,n+1}/\Sigma_n$ is homotopy equivalent to $C_n(\C)/S^1$, the quotient of the unordered configuration space by the circle action. Starting from the homotopy equivalence between the ordered configuration space $F_n(\C)$ and the space of cacti (see the work by McClure-Smith \cite{McClure-Smith1}) we will build a CW-complex which is homotopy equivalent to $C_n(\C)/S^1$.
 \subsection{Cacti}
In this paragraph we recall the definition and properties of the space of cacti. A 	combinatorial construction of this space was introduced by McClure and Smith \cite{McClure-Smith2}, while geometric constructions are due to Voronov \cite{Voronov} and Kaufmann \cite{Kaufmann}. Salvatore compared the two approaches in \cite{Salvatore2}. 
	\begin{defn}[The space of cacti]
		Let $\mathcal{C}_n$ be the set of partitions $x$ of $S^1$ into $n$ closed $1$-manifolds $I_1(x),\dots,I_n(x)$ such that:
		\begin{itemize}
			\item They have equal measure.
			\item They have pairwise disjoint interiors.
			\item  There is no cyclically ordered sequence of points $(z_1,z_2,z_3,z_4)$ in $S^1$ such that $z_1,z_3\in I_j(x)$, $z_2,z_4\in I_k(x)$ and $j\neq k$.
		\end{itemize}
		We can equip this set with a topology by defining a metric on it: for any $x,y\in \mathcal{C}_n$ we set
		\[
		d(x,y)=\sum_{j=1}^n \mu(I_j(x)-\mathring{I}_j(y))
		\]
		where $\mu$ denotes the measure. We will call $\mathcal{C}_n$ (with this topology) the \textbf{space of based cacti with $n$-lobes}. See Figure \ref{fig:esempio di cactus} for an example. 
	\end{defn}
	\begin{oss}
		$\mathcal{C}_n$ is called the space of cacti for the following reason: given $x\in\mathcal{C}_n$, let us define a relation $\sim$ on $S^1$: two points $z_1,z_2\in S^1$ are  equivalent if there is an index $j\in\{0,\dots,n\}$ such that $z_1$ and $z_2$ are the boundary points of the same connected component of $S^1-\mathring{I}_j$. The quotient space $c(x)\coloneqq S^1/\sim$ by this relation is a pointed space (the basepoint is just the image of $1\in S^1$) called the \textbf{cactus} associated to $x$: topologically it is a configuration of $n$-circles in the plane, called lobes, whose dual graph is a tree. The dual graph is a graph with two kind of vertices: a white vertex for any lobe and a black vertex for any intersection point between two lobes. An edges connects a white vertex $w$ to a black vertex $b$ if $b$ represent the intersection point of the lobe corresponding to $w$ with some other lobe. See Figure \ref{fig:esempi di cactus} for some examples. Note that the lobes are the image of the $1$-manifolds $I_1(x),\dots,I_n(x)$ under the quotient map $S^1\to S^1/\sim$. In what follows we will freely identify a partition $x\in\mathcal{C}_n$ and its associated cactus $c(x)$.
	\end{oss}
	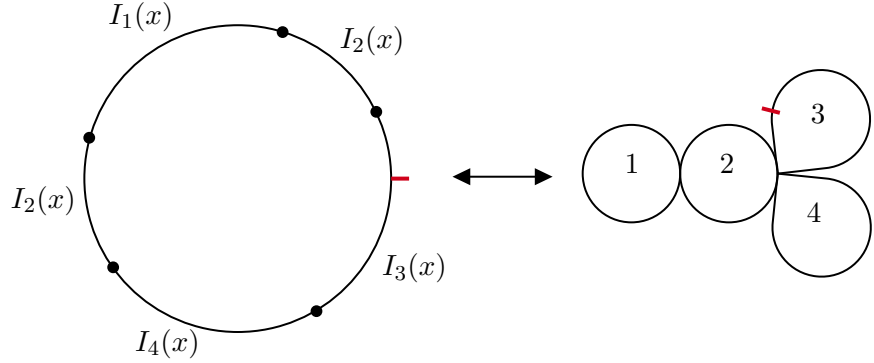
\begin{figure}
		\centering

		\tikzset{every picture/.style={line width=0.75pt}} 
		
		\begin{tikzpicture}[x=0.75pt,y=0.75pt,yscale=-1,xscale=1]
			
			\draw   (54.32,193.95) .. controls (54.32,151.45) and (88.78,117) .. (131.27,117) .. controls (173.77,117) and (208.23,151.45) .. (208.23,193.95) .. controls (208.23,236.45) and (173.77,270.9) .. (131.27,270.9) .. controls (88.78,270.9) and (54.32,236.45) .. (54.32,193.95) -- cycle ;
			\draw  [fill={rgb, 255:red, 0; green, 0; blue, 0 }  ,fill opacity=1 ] (66.32,238.45) .. controls (66.32,237.1) and (67.42,236) .. (68.77,236) .. controls (70.13,236) and (71.23,237.1) .. (71.23,238.45) .. controls (71.23,239.81) and (70.13,240.9) .. (68.77,240.9) .. controls (67.42,240.9) and (66.32,239.81) .. (66.32,238.45) -- cycle ;
			\draw  [fill={rgb, 255:red, 0; green, 0; blue, 0 }  ,fill opacity=1 ] (151.32,120.45) .. controls (151.32,119.1) and (152.42,118) .. (153.77,118) .. controls (155.13,118) and (156.23,119.1) .. (156.23,120.45) .. controls (156.23,121.81) and (155.13,122.9) .. (153.77,122.9) .. controls (152.42,122.9) and (151.32,121.81) .. (151.32,120.45) -- cycle ;
			\draw  [fill={rgb, 255:red, 0; green, 0; blue, 0 }  ,fill opacity=1 ] (198.32,160.45) .. controls (198.32,159.1) and (199.42,158) .. (200.77,158) .. controls (202.13,158) and (203.23,159.1) .. (203.23,160.45) .. controls (203.23,161.81) and (202.13,162.9) .. (200.77,162.9) .. controls (199.42,162.9) and (198.32,161.81) .. (198.32,160.45) -- cycle ;
			\draw  [fill={rgb, 255:red, 0; green, 0; blue, 0 }  ,fill opacity=1 ] (54.32,173.45) .. controls (54.32,172.1) and (55.42,171) .. (56.77,171) .. controls (58.13,171) and (59.23,172.1) .. (59.23,173.45) .. controls (59.23,174.81) and (58.13,175.9) .. (56.77,175.9) .. controls (55.42,175.9) and (54.32,174.81) .. (54.32,173.45) -- cycle ;
			\draw  [fill={rgb, 255:red, 0; green, 0; blue, 0 }  ,fill opacity=1 ] (168.32,260.45) .. controls (168.32,259.1) and (169.42,258) .. (170.77,258) .. controls (172.13,258) and (173.23,259.1) .. (173.23,260.45) .. controls (173.23,261.81) and (172.13,262.9) .. (170.77,262.9) .. controls (169.42,262.9) and (168.32,261.81) .. (168.32,260.45) -- cycle ;
			\draw [color={rgb, 255:red, 208; green, 2; blue, 27 }  ,draw opacity=1 ][line width=1.5]    (208.23,193.95) -- (217.23,193.95) ;
			\draw   (304.32,191.45) .. controls (304.32,177.95) and (315.27,167) .. (328.77,167) .. controls (342.28,167) and (353.23,177.95) .. (353.23,191.45) .. controls (353.23,204.96) and (342.28,215.9) .. (328.77,215.9) .. controls (315.27,215.9) and (304.32,204.96) .. (304.32,191.45) -- cycle ;
			\draw   (353.23,191.45) .. controls (353.23,177.95) and (364.17,167) .. (377.68,167) .. controls (391.18,167) and (402.13,177.95) .. (402.13,191.45) .. controls (402.13,204.96) and (391.18,215.9) .. (377.68,215.9) .. controls (364.17,215.9) and (353.23,204.96) .. (353.23,191.45) -- cycle ;
			\draw   (426.67,193.92) .. controls (440.22,195.29) and (450.1,207.38) .. (448.74,220.93) .. controls (447.37,234.49) and (435.28,244.37) .. (421.73,243) .. controls (408.17,241.64) and (398.29,229.54) .. (399.66,215.99) .. controls (401.31,199.63) and (402.13,191.45) .. (402.13,191.45) .. controls (402.13,191.45) and (410.31,192.28) .. (426.67,193.92) -- cycle ;
			\draw   (399.32,166.95) .. controls (397.77,153.41) and (407.49,141.19) .. (421.02,139.64) .. controls (434.56,138.09) and (446.78,147.81) .. (448.33,161.34) .. controls (449.88,174.87) and (440.16,187.1) .. (426.63,188.65) .. controls (410.3,190.52) and (402.13,191.45) .. (402.13,191.45) .. controls (402.13,191.45) and (401.19,183.28) .. (399.32,166.95) -- cycle ;
			\draw [color={rgb, 255:red, 208; green, 2; blue, 27 }  ,draw opacity=1 ][line width=1.5]    (394.23,158.68) -- (403.23,160.95) ;
			\draw    (242,193) -- (286.23,193) ;
			\draw [shift={(289.23,193)}, rotate = 180] [fill={rgb, 255:red, 0; green, 0; blue, 0 }  ][line width=0.08]  [draw opacity=0] (8.93,-4.29) -- (0,0) -- (8.93,4.29) -- cycle    ;
			\draw [shift={(239,193)}, rotate = 0] [fill={rgb, 255:red, 0; green, 0; blue, 0 }  ][line width=0.08]  [draw opacity=0] (8.93,-4.29) -- (0,0) -- (8.93,4.29) -- cycle    ;
			
			\draw (65,104) node [anchor=north west][inner sep=0.75pt]   [align=left] {$\displaystyle I_{1}( x)$};
			\draw (181,116) node [anchor=north west][inner sep=0.75pt]   [align=left] {$\displaystyle I_{2}( x)$};
			\draw (16,196) node [anchor=north west][inner sep=0.75pt]   [align=left] {$\displaystyle I_{2}( x)$};
			\draw (202.23,229.95) node [anchor=north west][inner sep=0.75pt]   [align=left] {$\displaystyle I_{3}( x)$};
			\draw (78.23,267.95) node [anchor=north west][inner sep=0.75pt]   [align=left] {$\displaystyle I_{4}( x)$};
			\draw (323.9,180.2) node [anchor=north west][inner sep=0.75pt]   [align=left] {$\displaystyle 1$};
			\draw (371.63,180.2) node [anchor=north west][inner sep=0.75pt]   [align=left] {$\displaystyle 2$};
			\draw (417.35,155.04) node [anchor=north west][inner sep=0.75pt]   [align=left] {$\displaystyle 3$};
			\draw (415.35,205.04) node [anchor=north west][inner sep=0.75pt]   [align=left] {$\displaystyle 4$};

		\end{tikzpicture}
		
		\caption{On the left there is an element $x\in\mathcal{C}_4$, on the right its associated cactus $c(x)$. The basepoint of the circle $S^1$ is depicted in red and corresponds to a basepoint on the cactus $c(x)$ (which we depict as a red spine).}
		\label{fig:esempio di cactus}
	\end{figure}
		\paragraph{Cell decomposition:} to any point $x\in \mathcal{C}_n$ we can associate a sequence $(X_1,\dots,X_l)$ of numbers in $\{1,\dots,n\}$ by the following procedure: start from the point $1\in S^1$ and  move along the circle clockwise. The sequence $(X_1,\dots,X_l)$ is obtained by writing one after the other the indices of all the $1$-manifolds you encounter and has the following properties:
	\begin{itemize}
		\item All values between $1$ and $n$ appear.
		\item $X _i\neq X_{i+1}$ for every $i=1,\dots,l-1$.
		\item There is no subsequence of the form $(i,j,i,j)$ with $i\neq j$.
	\end{itemize}
	This sequence is just a combinatorial way to encode the shape of the cactus $c(x)$. By an abuse of notation we will also denote by $(X_1,\dots,X_l)$ the subspace of $\mathcal{C}_n$ consisting of all partitions whose associated sequence is $(X_1,\dots,X_l)$. This subspace turns out to be homeomorphic to a product of simplices. More precisely, if $m_i$ is the cardinality of $\{j\in\{1,\dots,l\}\mid X_j=i\}$, $i=1,\dots,n$, then 
	\[
	(X_1,\dots,X_l)\cong \prod_{i=1}^n\Delta^{m_i-1}
	\]
	For an example look at the cactus of Figure \ref{fig:esempio di cactus}: it belongs to the cell $(3,4,2,1,2,3)\cong\Delta^0\times \Delta^1\times\Delta^1\times\Delta^0$.
	Intuitively all the cacti $c(x)$ associated to partitions $x\in(X_1,\dots,X_l)$ have the same shape. So we will represent pictorially a cell by drawing a cactus, meaning that the cell contains all the partitions $x\in\mathcal{C}_n$ whose associated cactus $c(x)$ has that shape. From this point of view, the parameters
	of a cell $(X_1,\dots,X_l)$ can be thought as the lengths of the arcs between two consecutive \emph{marked points}, where a marked point is an intersection point of lobes or the base point. The boundary of a cell is obtained by collapsing some of these arcs. See Figure \ref{fig:esempi di cactus e loro bordi} for some examples. This gives a regular CW-decomposition of   $\mathcal{C}_n$.
	\begin{figure}
		\centering

		\tikzset{every picture/.style={line width=0.75pt}} 
		
		\begin{tikzpicture}[x=0.75pt,y=0.75pt,yscale=-1,xscale=1]
			
			\draw   (321.32,82.28) .. controls (321.32,80.19) and (323.01,78.51) .. (325.1,78.51) .. controls (327.18,78.51) and (328.87,80.19) .. (328.87,82.28) .. controls (328.87,84.36) and (327.18,86.05) .. (325.1,86.05) .. controls (323.01,86.05) and (321.32,84.36) .. (321.32,82.28) -- cycle ;
			\draw   (400.41,82.28) .. controls (400.41,80.19) and (402.1,78.51) .. (404.19,78.51) .. controls (406.27,78.51) and (407.96,80.19) .. (407.96,82.28) .. controls (407.96,84.36) and (406.27,86.05) .. (404.19,86.05) .. controls (402.1,86.05) and (400.41,84.36) .. (400.41,82.28) -- cycle ;
			\draw   (479.51,82.28) .. controls (479.51,80.19) and (481.19,78.51) .. (483.28,78.51) .. controls (485.36,78.51) and (487.05,80.19) .. (487.05,82.28) .. controls (487.05,84.36) and (485.36,86.05) .. (483.28,86.05) .. controls (481.19,86.05) and (479.51,84.36) .. (479.51,82.28) -- cycle ;
			\draw  [fill={rgb, 255:red, 0; green, 0; blue, 0 }  ,fill opacity=1 ] (360.87,82.28) .. controls (360.87,80.19) and (362.56,78.51) .. (364.64,78.51) .. controls (366.73,78.51) and (368.41,80.19) .. (368.41,82.28) .. controls (368.41,84.36) and (366.73,86.05) .. (364.64,86.05) .. controls (362.56,86.05) and (360.87,84.36) .. (360.87,82.28) -- cycle ;
			\draw  [fill={rgb, 255:red, 0; green, 0; blue, 0 }  ,fill opacity=1 ] (439.96,82.28) .. controls (439.96,80.19) and (441.65,78.51) .. (443.73,78.51) .. controls (445.82,78.51) and (447.51,80.19) .. (447.51,82.28) .. controls (447.51,84.36) and (445.82,86.05) .. (443.73,86.05) .. controls (441.65,86.05) and (439.96,84.36) .. (439.96,82.28) -- cycle ;
			\draw    (328.87,82.28) -- (360.87,82.28) ;
			\draw    (368.41,82.28) -- (400.41,82.28) ;
			\draw    (407.96,82.28) -- (439.96,82.28) ;
			\draw    (447.51,82.28) -- (479.51,82.28) ;
			\draw   (169.76,82.28) .. controls (169.76,73.12) and (177.19,65.69) .. (186.35,65.69) .. controls (195.51,65.69) and (202.93,73.12) .. (202.93,82.28) .. controls (202.93,91.44) and (195.51,98.86) .. (186.35,98.86) .. controls (177.19,98.86) and (169.76,91.44) .. (169.76,82.28) -- cycle ;
			\draw   (202.93,82.28) .. controls (202.93,73.12) and (210.35,65.69) .. (219.51,65.69) .. controls (228.67,65.69) and (236.1,73.12) .. (236.1,82.28) .. controls (236.1,91.44) and (228.67,98.86) .. (219.51,98.86) .. controls (210.35,98.86) and (202.93,91.44) .. (202.93,82.28) -- cycle ;
			\draw   (136.6,82.28) .. controls (136.6,73.12) and (144.02,65.69) .. (153.18,65.69) .. controls (162.34,65.69) and (169.76,73.12) .. (169.76,82.28) .. controls (169.76,91.44) and (162.34,98.86) .. (153.18,98.86) .. controls (144.02,98.86) and (136.6,91.44) .. (136.6,82.28) -- cycle ;
			\draw    (275,82) -- (307.17,82) ;
			\draw [shift={(309.17,82)}, rotate = 180] [color={rgb, 255:red, 0; green, 0; blue, 0 }  ][line width=0.75]    (10.93,-3.29) .. controls (6.95,-1.4) and (3.31,-0.3) .. (0,0) .. controls (3.31,0.3) and (6.95,1.4) .. (10.93,3.29)   ;
			\draw    (275,82) -- (248.17,82) ;
			\draw [shift={(246.17,82)}, rotate = 360] [color={rgb, 255:red, 0; green, 0; blue, 0 }  ][line width=0.75]    (10.93,-3.29) .. controls (6.95,-1.4) and (3.31,-0.3) .. (0,0) .. controls (3.31,0.3) and (6.95,1.4) .. (10.93,3.29)   ;
			\draw   (171.32,149.23) .. controls (164.23,150.66) and (157.28,145.89) .. (155.81,138.58) .. controls (154.33,131.28) and (158.89,124.19) .. (165.98,122.76) .. controls (173.07,121.33) and (180.02,126.09) .. (181.49,133.4) .. controls (183.27,142.22) and (184.16,146.63) .. (184.16,146.63) .. controls (184.16,146.63) and (179.88,147.5) .. (171.32,149.23) -- cycle ;
			\draw   (187.72,134.03) .. controls (187.72,134.03) and (187.72,134.03) .. (187.72,134.03) .. controls (189.68,127.06) and (197.09,123.06) .. (204.27,125.08) .. controls (211.44,127.11) and (215.67,134.39) .. (213.7,141.36) .. controls (211.74,148.32) and (204.33,152.32) .. (197.15,150.3) .. controls (188.49,147.86) and (184.16,146.63) .. (184.16,146.63) .. controls (184.16,146.63) and (185.35,142.43) .. (187.72,134.03) -- cycle ;
			\draw   (193.57,155.74) .. controls (193.57,155.74) and (193.57,155.74) .. (193.57,155.74) .. controls (193.57,155.74) and (193.57,155.74) .. (193.57,155.74) .. controls (198.77,160.78) and (198.79,169.2) .. (193.6,174.56) .. controls (188.42,179.91) and (180,180.18) .. (174.8,175.15) .. controls (169.6,170.12) and (169.59,161.69) .. (174.77,156.34) .. controls (181.04,149.87) and (184.17,146.64) .. (184.16,146.63) .. controls (184.17,146.64) and (187.3,149.67) .. (193.57,155.74) -- cycle ;
			\draw   (346.32,144.28) .. controls (346.32,142.19) and (348.01,140.51) .. (350.1,140.51) .. controls (352.18,140.51) and (353.87,142.19) .. (353.87,144.28) .. controls (353.87,146.36) and (352.18,148.05) .. (350.1,148.05) .. controls (348.01,148.05) and (346.32,146.36) .. (346.32,144.28) -- cycle ;
			\draw   (414.94,174.74) .. controls (413.42,173.32) and (413.35,170.93) .. (414.77,169.41) .. controls (416.2,167.89) and (418.59,167.81) .. (420.11,169.24) .. controls (421.62,170.67) and (421.7,173.06) .. (420.27,174.57) .. controls (418.85,176.09) and (416.46,176.17) .. (414.94,174.74) -- cycle ;
			\draw   (420.23,118.96) .. controls (418.77,120.45) and (416.38,120.48) .. (414.89,119.03) .. controls (413.4,117.57) and (413.37,115.19) .. (414.82,113.69) .. controls (416.27,112.2) and (418.66,112.17) .. (420.15,113.62) .. controls (421.65,115.08) and (421.68,117.46) .. (420.23,118.96) -- cycle ;
			\draw  [fill={rgb, 255:red, 0; green, 0; blue, 0 }  ,fill opacity=1 ] (385.87,144.28) .. controls (385.87,142.19) and (387.56,140.51) .. (389.64,140.51) .. controls (391.73,140.51) and (393.41,142.19) .. (393.41,144.28) .. controls (393.41,146.36) and (391.73,148.05) .. (389.64,148.05) .. controls (387.56,148.05) and (385.87,146.36) .. (385.87,144.28) -- cycle ;
			\draw    (353.87,144.28) -- (385.87,144.28) ;
			\draw    (414.77,169.41) -- (389.64,144.28) ;
			\draw    (389.64,144.28) -- (414.89,119.03) ;
			\draw    (275,143) -- (307.17,143) ;
			\draw [shift={(309.17,143)}, rotate = 180] [color={rgb, 255:red, 0; green, 0; blue, 0 }  ][line width=0.75]    (10.93,-3.29) .. controls (6.95,-1.4) and (3.31,-0.3) .. (0,0) .. controls (3.31,0.3) and (6.95,1.4) .. (10.93,3.29)   ;
			\draw    (275,143) -- (248.17,143) ;
			\draw [shift={(246.17,143)}, rotate = 360] [color={rgb, 255:red, 0; green, 0; blue, 0 }  ][line width=0.75]    (10.93,-3.29) .. controls (6.95,-1.4) and (3.31,-0.3) .. (0,0) .. controls (3.31,0.3) and (6.95,1.4) .. (10.93,3.29)   ;
			\draw (320,61) node [anchor=north west][inner sep=0.75pt]   [align=left] {$\displaystyle 1$};
			\draw (399,61) node [anchor=north west][inner sep=0.75pt]   [align=left] {$\displaystyle 2$};
			\draw (478.52,60.99) node [anchor=north west][inner sep=0.75pt]   [align=left] {$\displaystyle 3$};
			\draw (149,75) node [anchor=north west][inner sep=0.75pt]   [align=left] {$\displaystyle 1$};
			\draw (181,75) node [anchor=north west][inner sep=0.75pt]   [align=left] {$\displaystyle 2$};
			\draw (214,75) node [anchor=north west][inner sep=0.75pt]   [align=left] {$\displaystyle 3$};
			\draw (179,154.46) node [anchor=north west][inner sep=0.75pt]   [align=left] {$\displaystyle 3$};
			\draw (163,129.46) node [anchor=north west][inner sep=0.75pt]   [align=left] {$\displaystyle 1$};
			\draw (195,129.46) node [anchor=north west][inner sep=0.75pt]   [align=left] {$\displaystyle 2$};
			\draw (335,135) node [anchor=north west][inner sep=0.75pt]   [align=left] {$\displaystyle 1$};
			\draw (421,97) node [anchor=north west][inner sep=0.75pt]   [align=left] {$\displaystyle 2$};
			\draw (419.52,174.99) node [anchor=north west][inner sep=0.75pt]   [align=left] {$\displaystyle 3$};

		\end{tikzpicture}
		
		\caption{On the left there are some cacti (without basepoint), on the right the corresponding dual graphs.}
		\label{fig:esempi di cactus}
	\end{figure}
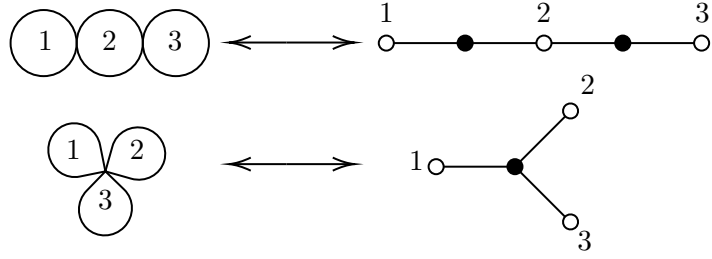
	
	\begin{figure}
		\centering

		\tikzset{every picture/.style={line width=0.75pt}} 
		
		\begin{tikzpicture}[x=0.75pt,y=0.75pt,yscale=-1,xscale=1]
			
			\draw   (64.76,278.28) .. controls (64.76,269.12) and (72.19,261.69) .. (81.35,261.69) .. controls (90.51,261.69) and (97.93,269.12) .. (97.93,278.28) .. controls (97.93,287.44) and (90.51,294.86) .. (81.35,294.86) .. controls (72.19,294.86) and (64.76,287.44) .. (64.76,278.28) -- cycle ;
			\draw   (97.93,278.28) .. controls (97.93,269.12) and (105.35,261.69) .. (114.51,261.69) .. controls (123.67,261.69) and (131.1,269.12) .. (131.1,278.28) .. controls (131.1,287.44) and (123.67,294.86) .. (114.51,294.86) .. controls (105.35,294.86) and (97.93,287.44) .. (97.93,278.28) -- cycle ;
			\draw   (31.6,278.28) .. controls (31.6,269.12) and (39.02,261.69) .. (48.18,261.69) .. controls (57.34,261.69) and (64.76,269.12) .. (64.76,278.28) .. controls (64.76,287.44) and (57.34,294.86) .. (48.18,294.86) .. controls (39.02,294.86) and (31.6,287.44) .. (31.6,278.28) -- cycle ;
			\draw   (235.32,330.23) .. controls (228.23,331.66) and (221.28,326.89) .. (219.81,319.58) .. controls (218.33,312.28) and (222.89,305.19) .. (229.98,303.76) .. controls (237.07,302.33) and (244.02,307.09) .. (245.49,314.4) .. controls (247.27,323.22) and (248.16,327.63) .. (248.16,327.63) .. controls (248.16,327.63) and (243.88,328.5) .. (235.32,330.23) -- cycle ;
			\draw   (251.72,315.03) .. controls (253.68,308.06) and (261.09,304.06) .. (268.27,306.08) .. controls (275.44,308.11) and (279.67,315.39) .. (277.7,322.36) .. controls (275.74,329.32) and (268.33,333.32) .. (261.15,331.3) .. controls (252.49,328.86) and (248.16,327.63) .. (248.16,327.63) .. controls (248.16,327.63) and (249.35,323.43) .. (251.72,315.03) -- cycle ;
			\draw   (257.57,336.74) .. controls (257.57,336.74) and (257.57,336.74) .. (257.57,336.74) .. controls (262.77,341.78) and (262.79,350.2) .. (257.6,355.56) .. controls (252.42,360.91) and (244,361.18) .. (238.8,356.15) .. controls (233.6,351.12) and (233.59,342.69) .. (238.77,337.34) .. controls (245.04,330.87) and (248.17,327.64) .. (248.16,327.63) .. controls (248.17,327.64) and (251.3,330.67) .. (257.57,336.74) -- cycle ;
			\draw  [color={rgb, 255:red, 0; green, 0; blue, 0 }  ,draw opacity=1 ] (120.8,92.51) .. controls (120.8,85.33) and (126.62,79.51) .. (133.8,79.51) .. controls (140.98,79.51) and (146.8,85.33) .. (146.8,92.51) .. controls (146.8,99.69) and (140.98,105.51) .. (133.8,105.51) .. controls (126.62,105.51) and (120.8,99.69) .. (120.8,92.51) -- cycle ;
			\draw  [color={rgb, 255:red, 0; green, 0; blue, 0 }  ,draw opacity=1 ] (94.8,92.51) .. controls (94.8,85.33) and (100.62,79.51) .. (107.8,79.51) .. controls (114.98,79.51) and (120.8,85.33) .. (120.8,92.51) .. controls (120.8,99.69) and (114.98,105.51) .. (107.8,105.51) .. controls (100.62,105.51) and (94.8,99.69) .. (94.8,92.51) -- cycle ;
			\draw  [color={rgb, 255:red, 0; green, 0; blue, 0 }  ,draw opacity=1 ] (374.8,92.51) .. controls (374.8,85.33) and (380.62,79.51) .. (387.8,79.51) .. controls (394.98,79.51) and (400.8,85.33) .. (400.8,92.51) .. controls (400.8,99.69) and (394.98,105.51) .. (387.8,105.51) .. controls (380.62,105.51) and (374.8,99.69) .. (374.8,92.51) -- cycle ;
			\draw  [color={rgb, 255:red, 0; green, 0; blue, 0 }  ,draw opacity=1 ] (348.8,92.51) .. controls (348.8,85.33) and (354.62,79.51) .. (361.8,79.51) .. controls (368.98,79.51) and (374.8,85.33) .. (374.8,92.51) .. controls (374.8,99.69) and (368.98,105.51) .. (361.8,105.51) .. controls (354.62,105.51) and (348.8,99.69) .. (348.8,92.51) -- cycle ;
			\draw  [color={rgb, 255:red, 0; green, 0; blue, 0 }  ,draw opacity=1 ] (250.8,163.51) .. controls (250.8,156.33) and (256.62,150.51) .. (263.8,150.51) .. controls (270.98,150.51) and (276.8,156.33) .. (276.8,163.51) .. controls (276.8,170.69) and (270.98,176.51) .. (263.8,176.51) .. controls (256.62,176.51) and (250.8,170.69) .. (250.8,163.51) -- cycle ;
			\draw  [color={rgb, 255:red, 0; green, 0; blue, 0 }  ,draw opacity=1 ] (224.8,163.51) .. controls (224.8,156.33) and (230.62,150.51) .. (237.8,150.51) .. controls (244.98,150.51) and (250.8,156.33) .. (250.8,163.51) .. controls (250.8,170.69) and (244.98,176.51) .. (237.8,176.51) .. controls (230.62,176.51) and (224.8,170.69) .. (224.8,163.51) -- cycle ;
			\draw  [fill={rgb, 255:red, 0; green, 0; blue, 0 }  ,fill opacity=1 ] (156.9,92.8) .. controls (156.9,91.25) and (158.15,90) .. (159.7,90) .. controls (161.25,90) and (162.5,91.25) .. (162.5,92.8) .. controls (162.5,94.35) and (161.25,95.6) .. (159.7,95.6) .. controls (158.15,95.6) and (156.9,94.35) .. (156.9,92.8) -- cycle ;
			\draw  [fill={rgb, 255:red, 0; green, 0; blue, 0 }  ,fill opacity=1 ] (335.5,92.8) .. controls (335.5,91.25) and (336.75,90) .. (338.3,90) .. controls (339.85,90) and (341.1,91.25) .. (341.1,92.8) .. controls (341.1,94.35) and (339.85,95.6) .. (338.3,95.6) .. controls (336.75,95.6) and (335.5,94.35) .. (335.5,92.8) -- cycle ;
			\draw   (159.7,92.8) .. controls (159.7,64.3) and (199.68,41.2) .. (249,41.2) .. controls (298.32,41.2) and (338.3,64.3) .. (338.3,92.8) .. controls (338.3,121.3) and (298.32,144.4) .. (249,144.4) .. controls (199.68,144.4) and (159.7,121.3) .. (159.7,92.8) -- cycle ;
			\draw  [color={rgb, 255:red, 0; green, 0; blue, 0 }  ,draw opacity=1 ] (250.8,23.51) .. controls (250.8,16.33) and (256.62,10.51) .. (263.8,10.51) .. controls (270.98,10.51) and (276.8,16.33) .. (276.8,23.51) .. controls (276.8,30.69) and (270.98,36.51) .. (263.8,36.51) .. controls (256.62,36.51) and (250.8,30.69) .. (250.8,23.51) -- cycle ;
			\draw  [color={rgb, 255:red, 0; green, 0; blue, 0 }  ,draw opacity=1 ] (224.8,23.51) .. controls (224.8,16.33) and (230.62,10.51) .. (237.8,10.51) .. controls (244.98,10.51) and (250.8,16.33) .. (250.8,23.51) .. controls (250.8,30.69) and (244.98,36.51) .. (237.8,36.51) .. controls (230.62,36.51) and (224.8,30.69) .. (224.8,23.51) -- cycle ;
			\draw  [color={rgb, 255:red, 208; green, 2; blue, 27 }  ,draw opacity=1 ][fill={rgb, 255:red, 208; green, 2; blue, 27 }  ,fill opacity=1 ] (119.3,87.01) .. controls (119.3,86.18) and (119.97,85.51) .. (120.8,85.51) .. controls (121.63,85.51) and (122.3,86.18) .. (122.3,87.01) .. controls (122.3,87.84) and (121.63,88.51) .. (120.8,88.51) .. controls (119.97,88.51) and (119.3,87.84) .. (119.3,87.01) -- cycle ;
			\draw  [color={rgb, 255:red, 208; green, 2; blue, 27 }  ,draw opacity=1 ][fill={rgb, 255:red, 208; green, 2; blue, 27 }  ,fill opacity=1 ] (223.3,163.51) .. controls (223.3,162.68) and (223.97,162.01) .. (224.8,162.01) .. controls (225.63,162.01) and (226.3,162.68) .. (226.3,163.51) .. controls (226.3,164.34) and (225.63,165.01) .. (224.8,165.01) .. controls (223.97,165.01) and (223.3,164.34) .. (223.3,163.51) -- cycle ;
			\draw  [color={rgb, 255:red, 208; green, 2; blue, 27 }  ,draw opacity=1 ][fill={rgb, 255:red, 208; green, 2; blue, 27 }  ,fill opacity=1 ] (373.3,97.01) .. controls (373.3,96.18) and (373.97,95.51) .. (374.8,95.51) .. controls (375.63,95.51) and (376.3,96.18) .. (376.3,97.01) .. controls (376.3,97.84) and (375.63,98.51) .. (374.8,98.51) .. controls (373.97,98.51) and (373.3,97.84) .. (373.3,97.01) -- cycle ;
			\draw  [color={rgb, 255:red, 208; green, 2; blue, 27 }  ,draw opacity=1 ][fill={rgb, 255:red, 208; green, 2; blue, 27 }  ,fill opacity=1 ] (275.3,23.51) .. controls (275.3,22.68) and (275.97,22.01) .. (276.8,22.01) .. controls (277.63,22.01) and (278.3,22.68) .. (278.3,23.51) .. controls (278.3,24.34) and (277.63,25.01) .. (276.8,25.01) .. controls (275.97,25.01) and (275.3,24.34) .. (275.3,23.51) -- cycle ;
			\draw  [color={rgb, 255:red, 208; green, 2; blue, 27 }  ,draw opacity=1 ][fill={rgb, 255:red, 208; green, 2; blue, 27 }  ,fill opacity=1 ] (79.85,261.69) .. controls (79.85,260.87) and (80.52,260.19) .. (81.35,260.19) .. controls (82.17,260.19) and (82.85,260.87) .. (82.85,261.69) .. controls (82.85,262.52) and (82.17,263.19) .. (81.35,263.19) .. controls (80.52,263.19) and (79.85,262.52) .. (79.85,261.69) -- cycle ;
			\draw   (237.76,224.28) .. controls (237.76,215.12) and (245.19,207.69) .. (254.35,207.69) .. controls (263.51,207.69) and (270.93,215.12) .. (270.93,224.28) .. controls (270.93,233.44) and (263.51,240.86) .. (254.35,240.86) .. controls (245.19,240.86) and (237.76,233.44) .. (237.76,224.28) -- cycle ;
			\draw   (270.93,224.28) .. controls (270.93,215.12) and (278.35,207.69) .. (287.51,207.69) .. controls (296.67,207.69) and (304.1,215.12) .. (304.1,224.28) .. controls (304.1,233.44) and (296.67,240.86) .. (287.51,240.86) .. controls (278.35,240.86) and (270.93,233.44) .. (270.93,224.28) -- cycle ;
			\draw   (204.6,224.28) .. controls (204.6,215.12) and (212.02,207.69) .. (221.18,207.69) .. controls (230.34,207.69) and (237.76,215.12) .. (237.76,224.28) .. controls (237.76,233.44) and (230.34,240.86) .. (221.18,240.86) .. controls (212.02,240.86) and (204.6,233.44) .. (204.6,224.28) -- cycle ;
			\draw  [color={rgb, 255:red, 208; green, 2; blue, 27 }  ,draw opacity=1 ][fill={rgb, 255:red, 208; green, 2; blue, 27 }  ,fill opacity=1 ] (236.26,218.78) .. controls (236.26,217.95) and (236.93,217.28) .. (237.76,217.28) .. controls (238.59,217.28) and (239.26,217.95) .. (239.26,218.78) .. controls (239.26,219.61) and (238.59,220.28) .. (237.76,220.28) .. controls (236.93,220.28) and (236.26,219.61) .. (236.26,218.78) -- cycle ;
			\draw   (237.76,274.28) .. controls (237.76,265.12) and (245.19,257.69) .. (254.35,257.69) .. controls (263.51,257.69) and (270.93,265.12) .. (270.93,274.28) .. controls (270.93,283.44) and (263.51,290.86) .. (254.35,290.86) .. controls (245.19,290.86) and (237.76,283.44) .. (237.76,274.28) -- cycle ;
			\draw   (270.93,274.28) .. controls (270.93,265.12) and (278.35,257.69) .. (287.51,257.69) .. controls (296.67,257.69) and (304.1,265.12) .. (304.1,274.28) .. controls (304.1,283.44) and (296.67,290.86) .. (287.51,290.86) .. controls (278.35,290.86) and (270.93,283.44) .. (270.93,274.28) -- cycle ;
			\draw   (204.6,274.28) .. controls (204.6,265.12) and (212.02,257.69) .. (221.18,257.69) .. controls (230.34,257.69) and (237.76,265.12) .. (237.76,274.28) .. controls (237.76,283.44) and (230.34,290.86) .. (221.18,290.86) .. controls (212.02,290.86) and (204.6,283.44) .. (204.6,274.28) -- cycle ;
			\draw  [color={rgb, 255:red, 208; green, 2; blue, 27 }  ,draw opacity=1 ][fill={rgb, 255:red, 208; green, 2; blue, 27 }  ,fill opacity=1 ] (269.43,269.78) .. controls (269.43,268.95) and (270.1,268.28) .. (270.93,268.28) .. controls (271.76,268.28) and (272.43,268.95) .. (272.43,269.78) .. controls (272.43,270.61) and (271.76,271.28) .. (270.93,271.28) .. controls (270.1,271.28) and (269.43,270.61) .. (269.43,269.78) -- cycle ;
			\draw  [color={rgb, 255:red, 208; green, 2; blue, 27 }  ,draw opacity=1 ][fill={rgb, 255:red, 208; green, 2; blue, 27 }  ,fill opacity=1 ] (275.43,312.78) .. controls (275.43,311.95) and (276.1,311.28) .. (276.93,311.28) .. controls (277.76,311.28) and (278.43,311.95) .. (278.43,312.78) .. controls (278.43,313.61) and (277.76,314.28) .. (276.93,314.28) .. controls (276.1,314.28) and (275.43,313.61) .. (275.43,312.78) -- cycle ;
			\draw    (139.1,278.28) -- (186.17,278.28) ;
			\draw [shift={(188.17,278.28)}, rotate = 180] [color={rgb, 255:red, 0; green, 0; blue, 0 }  ][line width=0.75]    (10.93,-3.29) .. controls (6.95,-1.4) and (3.31,-0.3) .. (0,0) .. controls (3.31,0.3) and (6.95,1.4) .. (10.93,3.29)   ;
			
			\draw (44,269) node [anchor=north west][inner sep=0.75pt]   [align=left] {$\displaystyle 1$};
			\draw (76,269) node [anchor=north west][inner sep=0.75pt]   [align=left] {$\displaystyle 2$};
			\draw (109,269) node [anchor=north west][inner sep=0.75pt]   [align=left] {$\displaystyle 3$};
			\draw (243,335.46) node [anchor=north west][inner sep=0.75pt]   [align=left] {$\displaystyle 3$};
			\draw (227,310.46) node [anchor=north west][inner sep=0.75pt]   [align=left] {$\displaystyle 1$};
			\draw (259,310.46) node [anchor=north west][inner sep=0.75pt]   [align=left] {$\displaystyle 2$};
			\draw (103,83) node [anchor=north west][inner sep=0.75pt]   [align=left] {$\displaystyle 1$};
			\draw (129,84) node [anchor=north west][inner sep=0.75pt]   [align=left] {$\displaystyle 2$};
			\draw (356,84) node [anchor=north west][inner sep=0.75pt]   [align=left] {$\displaystyle 1$};
			\draw (383,84) node [anchor=north west][inner sep=0.75pt]   [align=left] {$\displaystyle 2$};
			\draw (233,154) node [anchor=north west][inner sep=0.75pt]   [align=left] {$\displaystyle 1$};
			\draw (259,155) node [anchor=north west][inner sep=0.75pt]   [align=left] {$\displaystyle 2$};
			\draw (233,14) node [anchor=north west][inner sep=0.75pt]   [align=left] {$\displaystyle 1$};
			\draw (259,15) node [anchor=north west][inner sep=0.75pt]   [align=left] {$\displaystyle 2$};
			\draw (30,86.8) node [anchor=north west][inner sep=0.75pt]   [align=left] {$\displaystyle \mathcal{C}_{2} =$\\};
			\draw (217,215) node [anchor=north west][inner sep=0.75pt]   [align=left] {$\displaystyle 1$};
			\draw (249,215) node [anchor=north west][inner sep=0.75pt]   [align=left] {$\displaystyle 2$};
			\draw (282,215) node [anchor=north west][inner sep=0.75pt]   [align=left] {$\displaystyle 3$};
			\draw (217,265) node [anchor=north west][inner sep=0.75pt]   [align=left] {$\displaystyle 1$};
			\draw (249,265) node [anchor=north west][inner sep=0.75pt]   [align=left] {$\displaystyle 2$};
			\draw (282,265) node [anchor=north west][inner sep=0.75pt]   [align=left] {$\displaystyle 3$};
			\draw (156,258) node [anchor=north west][inner sep=0.75pt]   [align=left] {$\displaystyle \partial $};

		\end{tikzpicture}
		
		\caption{On top where is a full description of $\mathcal{C}_2\cong S^1$: there are two zero cells $(2,1)$ (on the left) and $(1,2)$ (on the right). The $1$-cells are $(2,1,2)$ (on the top) and $(1,2,1)$ (on the bottom). Below we see the cell $(2,3,2,1,2)\cong \Delta^0\times\Delta^2\times\Delta^0$ of $\mathcal{C}_3$ and the codimension one cells in its boundary.}
		\label{fig:esempi di cactus e loro bordi}
	\end{figure}
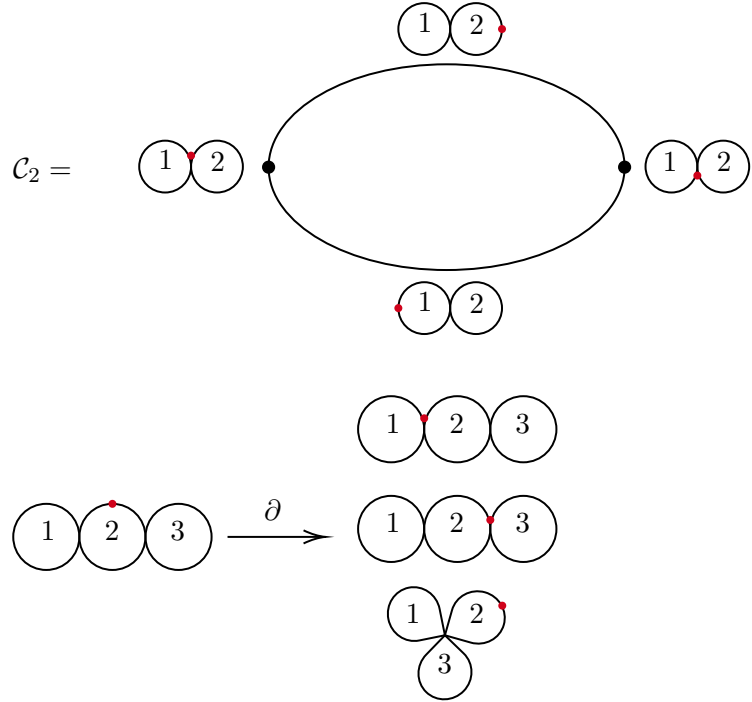
	\begin{oss}
		There are two relevant groups acting on $\mathcal{C}_n$: $S^1$ acts by rotating the basepoint of a cactus, $\Sigma_n$ acts by relabelling the lobes.
	\end{oss}
	The important thing about cacti is that they are a very small cellular model for the configuration space $F_n(\C)$:
	
	\begin{thm}[ \cite{Salvatore}]\label{thm:cactus sono deformation retract dello spazio di configurazioni}
		The space of cacti $\mathcal{C}_n$ is $(S^1\times\Sigma_n)$-equivariantly homotopy equivalent to $F_n(\C)$.
	\end{thm}
 \begin{corollario}\label{cor: spazio dei moduli omotopo a cactus senza pt base}
     The moduli space  $\M_{0,n+1}$ is homotopy equivalent to $\mathcal{C}_n/S^1$. 
 \end{corollario}
 \begin{proof}
      Recall that $\M_{0,n+1}$ is homeomorphic to $F_n(\C)/\C\rtimes \C^*$. By Theorem \ref{thm:cactus sono deformation retract dello spazio di configurazioni} $\mathcal{C}_n/S^1$ is homotopy equivalent to $F_n(\C)/S^1$. Therefore we get that 
     \[
     \M_{0,n+1}\cong \frac{F_n(\C)}{\C\rtimes \C^*}\simeq \frac{F_n(\C)}{S^1}\simeq \frac{\mathcal{C}_n}{S^1} 
     \]
     where the first homotopy equivalence holds because the group $\C\rtimes \C^*$ is homotopy equivalent to $S^1$.
 \end{proof}
\begin{oss}
    The quotient $\mathcal{C}_n/S^1$ is still a regular CW-complex, and its cells are described by the same combinatorics of $\mathcal{C}_n$ except that we do not have a basepoint on our cacti (see Figure \ref{fig:esempio di cactus senza punto base} for an example). We call $\mathcal{C}_n/S^1$ the \textbf{space of unbased cacti}. 
\end{oss}	
		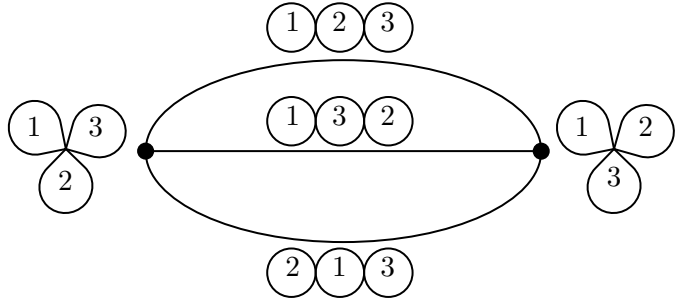
\begin{figure}
			\centering

			\tikzset{every picture/.style={line width=0.75pt}} 
			
			\begin{tikzpicture}[x=0.75pt,y=0.75pt,yscale=-1,xscale=1]
				
				\draw  [fill={rgb, 255:red, 0; green, 0; blue, 0 }  ,fill opacity=1 ] (295.87,96.73) .. controls (295.87,94.65) and (297.56,92.96) .. (299.64,92.96) .. controls (301.73,92.96) and (303.41,94.65) .. (303.41,96.73) .. controls (303.41,98.82) and (301.73,100.51) .. (299.64,100.51) .. controls (297.56,100.51) and (295.87,98.82) .. (295.87,96.73) -- cycle ;
				\draw   (48.32,98.23) .. controls (48.32,98.23) and (48.32,98.23) .. (48.32,98.23) .. controls (48.32,98.23) and (48.32,98.23) .. (48.32,98.23) .. controls (41.23,99.66) and (34.28,94.89) .. (32.81,87.58) .. controls (31.33,80.28) and (35.89,73.19) .. (42.98,71.76) .. controls (50.07,70.33) and (57.02,75.09) .. (58.49,82.4) .. controls (60.27,91.22) and (61.16,95.63) .. (61.16,95.63) .. controls (61.16,95.63) and (56.88,96.5) .. (48.32,98.23) -- cycle ;
				\draw   (64.72,83.03) .. controls (64.72,83.03) and (64.72,83.03) .. (64.72,83.03) .. controls (66.68,76.06) and (74.09,72.06) .. (81.27,74.08) .. controls (88.44,76.11) and (92.67,83.39) .. (90.7,90.36) .. controls (88.74,97.32) and (81.33,101.32) .. (74.15,99.3) .. controls (65.49,96.86) and (61.16,95.63) .. (61.16,95.63) .. controls (61.16,95.63) and (62.35,91.43) .. (64.72,83.03) -- cycle ;
				\draw   (70.57,104.74) .. controls (70.57,104.74) and (70.57,104.74) .. (70.57,104.74) .. controls (75.77,109.78) and (75.79,118.2) .. (70.6,123.56) .. controls (65.42,128.91) and (57,129.18) .. (51.8,124.15) .. controls (46.6,119.12) and (46.59,110.69) .. (51.77,105.34) .. controls (51.77,105.34) and (51.77,105.34) .. (51.77,105.34) .. controls (58.04,98.87) and (61.17,95.64) .. (61.16,95.63) .. controls (61.17,95.64) and (64.3,98.67) .. (70.57,104.74) -- cycle ;
				\draw  [fill={rgb, 255:red, 0; green, 0; blue, 0 }  ,fill opacity=1 ] (97.39,96.7) .. controls (97.39,94.62) and (99.08,92.93) .. (101.17,92.93) .. controls (103.25,92.93) and (104.94,94.62) .. (104.94,96.7) .. controls (104.94,98.78) and (103.25,100.47) .. (101.17,100.47) .. controls (99.08,100.47) and (97.39,98.78) .. (97.39,96.7) -- cycle ;
				\draw   (323.32,98.23) .. controls (323.32,98.23) and (323.32,98.23) .. (323.32,98.23) .. controls (323.32,98.23) and (323.32,98.23) .. (323.32,98.23) .. controls (316.23,99.66) and (309.28,94.89) .. (307.81,87.58) .. controls (306.33,80.28) and (310.89,73.19) .. (317.98,71.76) .. controls (325.07,70.33) and (332.02,75.09) .. (333.49,82.4) .. controls (335.27,91.22) and (336.16,95.63) .. (336.16,95.63) .. controls (336.16,95.63) and (331.88,96.5) .. (323.32,98.23) -- cycle ;
				\draw   (339.72,83.03) .. controls (339.72,83.03) and (339.72,83.03) .. (339.72,83.03) .. controls (341.68,76.06) and (349.09,72.06) .. (356.27,74.08) .. controls (363.44,76.11) and (367.67,83.39) .. (365.7,90.36) .. controls (363.74,97.32) and (356.33,101.32) .. (349.15,99.3) .. controls (340.49,96.86) and (336.16,95.63) .. (336.16,95.63) .. controls (336.16,95.63) and (337.35,91.43) .. (339.72,83.03) -- cycle ;
				\draw   (345.57,104.74) .. controls (345.57,104.74) and (345.57,104.74) .. (345.57,104.74) .. controls (350.77,109.78) and (350.79,118.2) .. (345.6,123.56) .. controls (340.42,128.91) and (332,129.18) .. (326.8,124.15) .. controls (321.6,119.12) and (321.59,110.69) .. (326.77,105.34) .. controls (333.04,98.87) and (336.17,95.64) .. (336.16,95.63) .. controls (336.17,95.64) and (339.3,98.67) .. (345.57,104.74) -- cycle ;
				\draw   (101.17,96.7) .. controls (101.17,71.5) and (145.6,51.07) .. (200.4,51.07) .. controls (255.21,51.07) and (299.64,71.5) .. (299.64,96.7) .. controls (299.64,121.9) and (255.21,142.33) .. (200.4,142.33) .. controls (145.6,142.33) and (101.17,121.9) .. (101.17,96.7) -- cycle ;
				\draw   (186.48,34.71) .. controls (186.48,27.99) and (191.92,22.55) .. (198.63,22.55) .. controls (205.34,22.55) and (210.79,27.99) .. (210.79,34.71) .. controls (210.79,41.42) and (205.34,46.86) .. (198.63,46.86) .. controls (191.92,46.86) and (186.48,41.42) .. (186.48,34.71) -- cycle ;
				\draw   (210.79,34.71) .. controls (210.79,27.99) and (216.23,22.55) .. (222.94,22.55) .. controls (229.65,22.55) and (235.1,27.99) .. (235.1,34.71) .. controls (235.1,41.42) and (229.65,46.86) .. (222.94,46.86) .. controls (216.23,46.86) and (210.79,41.42) .. (210.79,34.71) -- cycle ;
				\draw   (162.17,34.71) .. controls (162.17,27.99) and (167.61,22.55) .. (174.32,22.55) .. controls (181.03,22.55) and (186.48,27.99) .. (186.48,34.71) .. controls (186.48,41.42) and (181.03,46.86) .. (174.32,46.86) .. controls (167.61,46.86) and (162.17,41.42) .. (162.17,34.71) -- cycle ;
				\draw    (101.17,96.7) -- (299.64,96.7) ;
				\draw   (186.48,81.71) .. controls (186.48,74.99) and (191.92,69.55) .. (198.63,69.55) .. controls (205.34,69.55) and (210.79,74.99) .. (210.79,81.71) .. controls (210.79,88.42) and (205.34,93.86) .. (198.63,93.86) .. controls (191.92,93.86) and (186.48,88.42) .. (186.48,81.71) -- cycle ;
				\draw   (210.79,81.71) .. controls (210.79,74.99) and (216.23,69.55) .. (222.94,69.55) .. controls (229.65,69.55) and (235.1,74.99) .. (235.1,81.71) .. controls (235.1,88.42) and (229.65,93.86) .. (222.94,93.86) .. controls (216.23,93.86) and (210.79,88.42) .. (210.79,81.71) -- cycle ;
				\draw   (162.17,81.71) .. controls (162.17,74.99) and (167.61,69.55) .. (174.32,69.55) .. controls (181.03,69.55) and (186.48,74.99) .. (186.48,81.71) .. controls (186.48,88.42) and (181.03,93.86) .. (174.32,93.86) .. controls (167.61,93.86) and (162.17,88.42) .. (162.17,81.71) -- cycle ;
				\draw   (186.48,157.71) .. controls (186.48,150.99) and (191.92,145.55) .. (198.63,145.55) .. controls (205.34,145.55) and (210.79,150.99) .. (210.79,157.71) .. controls (210.79,164.42) and (205.34,169.86) .. (198.63,169.86) .. controls (191.92,169.86) and (186.48,164.42) .. (186.48,157.71) -- cycle ;
				\draw   (210.79,157.71) .. controls (210.79,150.99) and (216.23,145.55) .. (222.94,145.55) .. controls (229.65,145.55) and (235.1,150.99) .. (235.1,157.71) .. controls (235.1,164.42) and (229.65,169.86) .. (222.94,169.86) .. controls (216.23,169.86) and (210.79,164.42) .. (210.79,157.71) -- cycle ;
				\draw   (162.17,157.71) .. controls (162.17,150.99) and (167.61,145.55) .. (174.32,145.55) .. controls (181.03,145.55) and (186.48,150.99) .. (186.48,157.71) .. controls (186.48,164.42) and (181.03,169.86) .. (174.32,169.86) .. controls (167.61,169.86) and (162.17,164.42) .. (162.17,157.71) -- cycle ;
				
				\draw (71,78.46) node [anchor=north west][inner sep=0.75pt]   [align=left] {$\displaystyle 3$};
				\draw (40,78.46) node [anchor=north west][inner sep=0.75pt]   [align=left] {$\displaystyle 1$};
				\draw (55.77,105.34) node [anchor=north west][inner sep=0.75pt]   [align=left] {$\displaystyle 2$};
				\draw (331,103.46) node [anchor=north west][inner sep=0.75pt]   [align=left] {$\displaystyle 3$};
				\draw (315,78.46) node [anchor=north west][inner sep=0.75pt]   [align=left] {$\displaystyle 1$};
				\draw (347,78.46) node [anchor=north west][inner sep=0.75pt]   [align=left] {$\displaystyle 2$};
				\draw (169.79,25.5) node [anchor=north west][inner sep=0.75pt]   [align=left] {$\displaystyle 1$};
				\draw (193.24,25.5) node [anchor=north west][inner sep=0.75pt]   [align=left] {$\displaystyle 2$};
				\draw (217.43,25.5) node [anchor=north west][inner sep=0.75pt]   [align=left] {$\displaystyle 3$};
				\draw (169.79,72.5) node [anchor=north west][inner sep=0.75pt]   [align=left] {$\displaystyle 1$};
				\draw (193.24,72.5) node [anchor=north west][inner sep=0.75pt]   [align=left] {$\displaystyle 3$};
				\draw (217.43,72.5) node [anchor=north west][inner sep=0.75pt]   [align=left] {$\displaystyle 2$};
				\draw (169.79,148.5) node [anchor=north west][inner sep=0.75pt]   [align=left] {$\displaystyle 2$};
				\draw (193.24,148.5) node [anchor=north west][inner sep=0.75pt]   [align=left] {$\displaystyle 1$};
				\draw (217.43,148.5) node [anchor=north west][inner sep=0.75pt]   [align=left] {$\displaystyle 3$};

			\end{tikzpicture}

			\caption{This picture shows the CW-complex $\mathcal{C}_3/S^1\simeq \M_{0,4}$. There are two zero cells and three edges. }
			\label{fig:esempio di cactus senza punto base}
		\end{figure}

 \subsection{A combinatorial model for $\mathcal{M}_{0,n+1}/\Sigma_n$}
 In this paragraph  we describe a small combinatorial model for $\mathcal{M}_{0,n+1}/\Sigma_n$. We will use this model in Paragraph \ref{sec:examples} to compute $H_*(\M_{0,n+1}/\Sigma_n;\Z)$ for small values of $n$.
	\begin{prop}\label{prop: equivalenza omotopica tra spazio dei moduli e quoziente dai cactus}
		$\M_{0,n+1}/\Sigma_n$ is homotopy equivalent to $\mathcal{C}_n/(S^1\times \Sigma_n)$. 
	\end{prop}
	\begin{proof}
		By Proposition \ref{thm:cactus sono deformation retract dello spazio di configurazioni} $F_n(\C)$ is $(S^1\times \Sigma_n)$-equivariantly homotopy equivalent to $\mathcal{C}_n$, so we have a homotopy equivalence between the quotients
		\[
		F_n(\C)/(S^1\times \Sigma_n)\simeq \mathcal{C}_n/(S^1\times \Sigma_n)
		\]
		The left hand space is homotopy equivalent to $\M_{0,n+1}/\Sigma_n$, and this proves the statement.
	\end{proof}
 \begin{oss}\label{oss: quoziente stretto è omotopo al quoziente delle configurazioni non ordinate per il cerchio}
     As we know from Theorem \ref{thm:cactus sono deformation retract dello spazio di configurazioni} the space of cacti $\mathcal{C}_n$ is $(S^1\times \Sigma_n)$-equivariantly homotopy equivalent to the ordered configuration space $F_n(\C)$. Therefore we have a homotopy equivalence between $\mathcal{C}_n/(S^1\times \Sigma_n)$  and $C_n(\C)/S^1$. Combining this with Proposition \ref{prop: equivalenza omotopica tra spazio dei moduli e quoziente dai cactus} we conclude that $\M_{0,n+1}/\Sigma_n$ is homotopy equivalent to $C_n(\C)/S^1$. We will freely switch between $\M_{0,n+1}/\Sigma_n$ and  $C_n(\C)/S^1$ depending on the situation.
 \end{oss}
	Now let us describe in details the space $\mathcal{C}_n/(S^1\times \Sigma_n)$: recall that the space of cacti $\mathcal{C}_n$ is a regular CW-complex whose cells are described by cacti (with numbered lobes and a basepoint). The boundary of a cell is described by collapsing arcs. The $(S^1\times \Sigma_n)$-action is encoded as follows:
	\begin{itemize}
		\item $S^1$ acts by rotating the basepoint. 
		\item  $\Sigma_n$ acts by relabelling the lobes.
	\end{itemize}
	To better understand the quotient $\mathcal{C}_n/(S^1\times \Sigma_n)$ it is useful to first consider the space of unbased cactus $\mathcal{C}_n/S^1\simeq \M_{0,n+1}$ and then quotient by the $\Sigma_n$-action: $\mathcal{C}_n/S^1$ is a CW-complex where a cell is given by an unbased cactus, and the boundary of each cell is obtained by collapsing arcs as in $\mathcal{C}_k$. Now we quotient by the $\Sigma_n$-action: this corresponds to relabelling the lobes. This time we have some cells with non-trivial stabilizer and we need to take this into account (some examples of these cells are depicted in Figure \ref{fig:alcune celle e i rispettivi stabilizzatori}). 
	\begin{figure}
		\centering

		\tikzset{every picture/.style={line width=0.75pt}} 
		
		\begin{tikzpicture}[x=0.75pt,y=0.75pt,yscale=-1,xscale=1]
			
			\draw   (100.16,73.86) .. controls (100.16,73.86) and (100.16,73.86) .. (100.16,73.86) .. controls (100.16,73.86) and (100.16,73.86) .. (100.16,73.86) .. controls (107.4,73.8) and (113.31,79.8) .. (113.37,87.25) .. controls (113.43,94.71) and (107.62,100.8) .. (100.38,100.86) .. controls (93.15,100.92) and (87.23,94.93) .. (87.17,87.47) .. controls (87.1,78.47) and (86.53,73.43) .. (85.45,72.33) .. controls (86.53,73.43) and (91.43,73.94) .. (100.16,73.86) -- cycle ;
			\draw   (70.75,70.62) .. controls (70.75,70.62) and (70.75,70.62) .. (70.75,70.62) .. controls (70.75,70.62) and (70.75,70.62) .. (70.75,70.62) .. controls (63.52,70.59) and (57.68,64.52) .. (57.71,57.06) .. controls (57.74,49.61) and (63.63,43.59) .. (70.86,43.62) .. controls (78.1,43.65) and (83.94,49.72) .. (83.91,57.17) .. controls (83.87,66.18) and (84.39,71.23) .. (85.45,72.33) .. controls (84.39,71.23) and (79.49,70.66) .. (70.75,70.62) -- cycle ;
			\draw   (87.18,57.64) .. controls (87.18,57.64) and (87.18,57.64) .. (87.18,57.64) .. controls (87.18,57.64) and (87.18,57.64) .. (87.18,57.64) .. controls (87.21,50.4) and (93.29,44.57) .. (100.74,44.61) .. controls (108.2,44.65) and (114.21,50.54) .. (114.17,57.78) .. controls (114.14,65.01) and (108.06,70.85) .. (100.61,70.81) .. controls (91.61,70.76) and (86.55,71.27) .. (85.45,72.33) .. controls (86.55,71.27) and (87.13,66.37) .. (87.18,57.64) -- cycle ;
			\draw   (83.7,87.02) .. controls (83.7,87.02) and (83.7,87.02) .. (83.7,87.02) .. controls (83.7,87.02) and (83.7,87.02) .. (83.7,87.02) .. controls (83.65,94.26) and (77.57,100.08) .. (70.11,100.04) .. controls (62.66,99.99) and (56.65,94.08) .. (56.7,86.85) .. controls (56.75,79.61) and (62.83,73.79) .. (70.28,73.84) .. controls (79.28,73.9) and (84.34,73.39) .. (85.45,72.33) .. controls (84.34,73.39) and (83.76,78.29) .. (83.7,87.02) -- cycle ;
			\draw   (59.48,203.71) .. controls (59.48,196.99) and (64.92,191.55) .. (71.63,191.55) .. controls (78.34,191.55) and (83.79,196.99) .. (83.79,203.71) .. controls (83.79,210.42) and (78.34,215.86) .. (71.63,215.86) .. controls (64.92,215.86) and (59.48,210.42) .. (59.48,203.71) -- cycle ;
			\draw   (83.79,203.71) .. controls (83.79,196.99) and (89.23,191.55) .. (95.94,191.55) .. controls (102.65,191.55) and (108.1,196.99) .. (108.1,203.71) .. controls (108.1,210.42) and (102.65,215.86) .. (95.94,215.86) .. controls (89.23,215.86) and (83.79,210.42) .. (83.79,203.71) -- cycle ;
			\draw   (35.17,203.71) .. controls (35.17,196.99) and (40.61,191.55) .. (47.32,191.55) .. controls (54.03,191.55) and (59.48,196.99) .. (59.48,203.71) .. controls (59.48,210.42) and (54.03,215.86) .. (47.32,215.86) .. controls (40.61,215.86) and (35.17,210.42) .. (35.17,203.71) -- cycle ;
			\draw   (108.1,203.71) .. controls (108.1,196.99) and (113.54,191.55) .. (120.25,191.55) .. controls (126.96,191.55) and (132.41,196.99) .. (132.41,203.71) .. controls (132.41,210.42) and (126.96,215.86) .. (120.25,215.86) .. controls (113.54,215.86) and (108.1,210.42) .. (108.1,203.71) -- cycle ;
			\draw   (49.79,146.71) .. controls (49.79,139.99) and (55.23,134.55) .. (61.94,134.55) .. controls (68.65,134.55) and (74.1,139.99) .. (74.1,146.71) .. controls (74.1,153.42) and (68.65,158.86) .. (61.94,158.86) .. controls (55.23,158.86) and (49.79,153.42) .. (49.79,146.71) -- cycle ;
			\draw   (74.1,146.71) .. controls (74.1,139.99) and (79.54,134.55) .. (86.25,134.55) .. controls (92.96,134.55) and (98.41,139.99) .. (98.41,146.71) .. controls (98.41,153.42) and (92.96,158.86) .. (86.25,158.86) .. controls (79.54,158.86) and (74.1,153.42) .. (74.1,146.71) -- cycle ;
			\draw   (113.12,148.24) .. controls (120.36,148.18) and (126.27,154.17) .. (126.33,161.63) .. controls (126.39,169.09) and (120.58,175.18) .. (113.34,175.24) .. controls (106.11,175.3) and (100.19,169.3) .. (100.13,161.84) .. controls (100.06,152.85) and (99.48,147.8) .. (98.41,146.71) .. controls (99.48,147.8) and (104.39,148.31) .. (113.12,148.24) -- cycle ;
			\draw   (100.14,132.01) .. controls (100.14,132.01) and (100.14,132.01) .. (100.14,132.01) .. controls (100.17,124.78) and (106.25,118.94) .. (113.7,118.98) .. controls (121.16,119.02) and (127.17,124.92) .. (127.13,132.15) .. controls (127.1,139.39) and (121.02,145.22) .. (113.57,145.18) .. controls (104.56,145.14) and (99.51,145.64) .. (98.41,146.71) .. controls (99.51,145.64) and (100.08,140.74) .. (100.14,132.01) -- cycle ;
			\draw   (80.48,266.71) .. controls (80.48,259.99) and (85.92,254.55) .. (92.63,254.55) .. controls (99.34,254.55) and (104.79,259.99) .. (104.79,266.71) .. controls (104.79,273.42) and (99.34,278.86) .. (92.63,278.86) .. controls (85.92,278.86) and (80.48,273.42) .. (80.48,266.71) -- cycle ;
			\draw   (80.48,291.02) .. controls (80.48,284.3) and (85.92,278.86) .. (92.63,278.86) .. controls (99.34,278.86) and (104.79,284.3) .. (104.79,291.02) .. controls (104.79,297.73) and (99.34,303.17) .. (92.63,303.17) .. controls (85.92,303.17) and (80.48,297.73) .. (80.48,291.02) -- cycle ;
			\draw   (60.17,252.71) .. controls (60.17,245.99) and (65.61,240.55) .. (72.32,240.55) .. controls (79.03,240.55) and (84.48,245.99) .. (84.48,252.71) .. controls (84.48,259.42) and (79.03,264.86) .. (72.32,264.86) .. controls (65.61,264.86) and (60.17,259.42) .. (60.17,252.71) -- cycle ;
			\draw   (101.17,252.71) .. controls (101.17,245.99) and (106.61,240.55) .. (113.32,240.55) .. controls (120.03,240.55) and (125.48,245.99) .. (125.48,252.71) .. controls (125.48,259.42) and (120.03,264.86) .. (113.32,264.86) .. controls (106.61,264.86) and (101.17,259.42) .. (101.17,252.71) -- cycle ;
			\draw    (143,41) -- (143,303.33) ;
			\draw    (227.17,114) -- (32.17,114) ;
			\draw    (227.17,182) -- (32.17,182) ;
			\draw    (227.17,232) -- (32.17,232) ;
			
			\draw (95,50.69) node [anchor=north west][inner sep=0.75pt]   [align=left] {$\displaystyle 3$};
			\draw (66,78.69) node [anchor=north west][inner sep=0.75pt]   [align=left] {$\displaystyle 1$};
			\draw (67,49.69) node [anchor=north west][inner sep=0.75pt]   [align=left] {$\displaystyle 2$};
			\draw (108.79,244.5) node [anchor=north west][inner sep=0.75pt]   [align=left] {$\displaystyle 2$};
			\draw (67.24,244.5) node [anchor=north west][inner sep=0.75pt]   [align=left] {$\displaystyle 1$};
			\draw (87.43,283.5) node [anchor=north west][inner sep=0.75pt]   [align=left] {$\displaystyle 3$};
			\draw (66.79,194.5) node [anchor=north west][inner sep=0.75pt]   [align=left] {$\displaystyle 2$};
			\draw (42.24,194.5) node [anchor=north west][inner sep=0.75pt]   [align=left] {$\displaystyle 1$};
			\draw (90.43,194.5) node [anchor=north west][inner sep=0.75pt]   [align=left] {$\displaystyle 3$};
			\draw (81.79,138.5) node [anchor=north west][inner sep=0.75pt]   [align=left] {$\displaystyle 2$};
			\draw (57.24,138.5) node [anchor=north west][inner sep=0.75pt]   [align=left] {$\displaystyle 1$};
			\draw (108.43,124.5) node [anchor=north west][inner sep=0.75pt]   [align=left] {$\displaystyle 3$};
			\draw (114.43,194.5) node [anchor=north west][inner sep=0.75pt]   [align=left] {$\displaystyle 4$};
			\draw (107.43,152.5) node [anchor=north west][inner sep=0.75pt]   [align=left] {$\displaystyle 4$};
			\draw (93.43,77.5) node [anchor=north west][inner sep=0.75pt]   [align=left] {$\displaystyle 4$};
			\draw (87.43,258.5) node [anchor=north west][inner sep=0.75pt]   [align=left] {$\displaystyle 4$};
			\draw (145,65) node [anchor=north west][inner sep=0.75pt]   [align=left] {$\displaystyle < ( 1234)  >$};
			\draw (158,140) node [anchor=north west][inner sep=0.75pt]   [align=left] {$\displaystyle < 1 >$};
			\draw (145,199) node [anchor=north west][inner sep=0.75pt]   [align=left] {$\displaystyle < ( 14)( 23)  >$};
			\draw (145,260) node [anchor=north west][inner sep=0.75pt]   [align=left] {$\displaystyle < ( 123)  >$};

		\end{tikzpicture}
		
		\caption{On the left there are some cells of $\mathcal{C}_4/S^1$. On the right is represented the stabilizer of the cell respect to the $\Sigma_4$ action by relabelling the lobes. In the first row we see a $0$-dimensional cell, whose stabilizer is the cyclic group of order four generated by $(1234)\in\Sigma_4$. In the second row there is a $1$-cell, which has trivial stabilizer. The last two cells are two dimensional with stabilizer respectively a cyclic group of order two and three.}
		\label{fig:alcune celle e i rispettivi stabilizzatori}
	\end{figure}
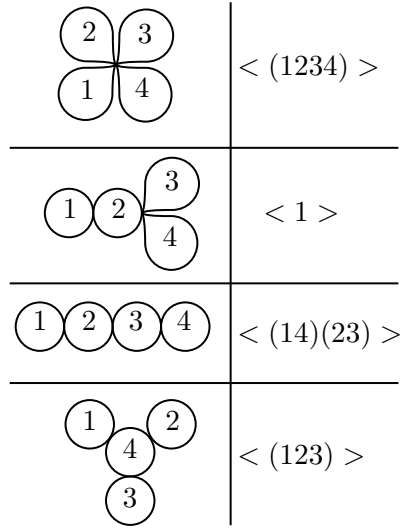
	To sum up, $\mathcal{C}_n/(S^1\times \Sigma_n)$ is built up by two types of cells, those with trivial stabilizer ("not symmetric cells") and those with non-trivial stabilizer ("symmetric cells"):
	\begin{enumerate}
		\item \textbf{Not symmetric cells:} take a cactus $c$ (without basepoint and labelling of the lobes) and fixes an arbitrary labelling of the lobes. Being a not symmetric cell means that the only permutation of $\Sigma_n$ which fix $c$ as a labelled cactus is the identity. The cell $\sigma(c)$ associated to $c$ is given by: 
		\[
		\sigma(c)\coloneqq \prod_{i=1}^{n}\Delta^{n_i-1}
		\]
		where $n_i$ is the number of intersection points of the $i$-th lobe with the other lobes.  
		As for the space of cacti $\mathcal{C}_n$ the parameters of $\Delta^{n_i-1}$ represent the length of the arcs between two intersection points. The boundary of $\sigma(c)$ is given by sending to zero some parameter, i.e. collapsing some arc.
		\item \textbf{Symmetric cells:} let $c$ be a cactus (without basepoint and labelling of the lobes) and fix an arbitrary labelling of the lobes. Being a symmetric cell means that the isotropy group of $c$ as a labelled cactus is a non trivial subroup $G_c\leq \Sigma_n$.  This cactus gives us a cell
		\[
		\sigma(c)\coloneqq \frac{\prod_{i=1}^{n}\Delta^{n_i-1}}{G_c}
		\]
		In other words, $\sigma(c)$ is the quotient of the cell associated to $c$ as a labelled cactus (with an arbitrary fixed labelling) by the action of its isotropy group.
		The boundary of $\sigma(c)$ is given by obtained by sending to zero some parameter, i.e. collapsing some arc.
	\end{enumerate}
	We conclude this paragraph by discussing some explicit examples.
	\begin{es}[n=3] There are only two cacti (unlabelled, without basepoint) with three lobes, let us call them $c_0$ and $c_1$ (see Figure \ref{fig:esempio con tre lobi}). $c_0$ is fixed by a rotation of $2\pi/3$, so its stabilizer is $\Z/3$. The corresponding cell is a point. $c_1$ is fixed by a rotation of $\pi$, so its stabilizer is $\Z/2$. The corresponding cell $\sigma(c_1)$ is obtained from $\Delta^1=\{(t_0,t_1)\in[0,1]^2\mid t_0+t_1=1\}$ quotienting by the relation $(t_0,t_1)\sim(t_1,t_0)$. Therefore $\mathcal{C}_3/(S^1\times \Sigma_3)$ is contractible.     
	\end{es}  
	\begin{figure}
		\centering

		\tikzset{every picture/.style={line width=0.75pt}} 
		
		\begin{tikzpicture}[x=0.75pt,y=0.75pt,yscale=-1,xscale=1]
			
			\draw  [fill={rgb, 255:red, 0; green, 0; blue, 0 }  ,fill opacity=1 ] (356.87,96.73) .. controls (356.87,94.65) and (358.56,92.96) .. (360.64,92.96) .. controls (362.73,92.96) and (364.41,94.65) .. (364.41,96.73) .. controls (364.41,98.82) and (362.73,100.51) .. (360.64,100.51) .. controls (358.56,100.51) and (356.87,98.82) .. (356.87,96.73) -- cycle ;
			\draw   (109.32,98.23) .. controls (109.32,98.23) and (109.32,98.23) .. (109.32,98.23) .. controls (109.32,98.23) and (109.32,98.23) .. (109.32,98.23) .. controls (102.23,99.66) and (95.28,94.89) .. (93.81,87.58) .. controls (92.33,80.28) and (96.89,73.19) .. (103.98,71.76) .. controls (111.07,70.33) and (118.02,75.09) .. (119.49,82.4) .. controls (121.27,91.22) and (122.16,95.63) .. (122.16,95.63) .. controls (122.16,95.63) and (117.88,96.5) .. (109.32,98.23) -- cycle ;
			\draw   (125.72,83.03) .. controls (125.72,83.03) and (125.72,83.03) .. (125.72,83.03) .. controls (127.68,76.06) and (135.09,72.06) .. (142.27,74.08) .. controls (149.44,76.11) and (153.67,83.39) .. (151.7,90.36) .. controls (149.74,97.32) and (142.33,101.32) .. (135.15,99.3) .. controls (126.49,96.86) and (122.16,95.63) .. (122.16,95.63) .. controls (122.16,95.63) and (123.35,91.43) .. (125.72,83.03) -- cycle ;
			\draw   (131.57,104.74) .. controls (131.57,104.74) and (131.57,104.74) .. (131.57,104.74) .. controls (136.77,109.78) and (136.79,118.2) .. (131.6,123.56) .. controls (126.42,128.91) and (118,129.18) .. (112.8,124.15) .. controls (107.6,119.12) and (107.59,110.69) .. (112.77,105.34) .. controls (119.04,98.87) and (122.17,95.64) .. (122.16,95.63) .. controls (122.17,95.64) and (125.3,98.67) .. (131.57,104.74) -- cycle ;
			\draw  [fill={rgb, 255:red, 0; green, 0; blue, 0 }  ,fill opacity=1 ] (158.39,96.7) .. controls (158.39,94.62) and (160.08,92.93) .. (162.17,92.93) .. controls (164.25,92.93) and (165.94,94.62) .. (165.94,96.7) .. controls (165.94,98.78) and (164.25,100.47) .. (162.17,100.47) .. controls (160.08,100.47) and (158.39,98.78) .. (158.39,96.7) -- cycle ;
			\draw   (384.32,98.23) .. controls (384.32,98.23) and (384.32,98.23) .. (384.32,98.23) .. controls (384.32,98.23) and (384.32,98.23) .. (384.32,98.23) .. controls (377.23,99.66) and (370.28,94.89) .. (368.81,87.58) .. controls (367.33,80.28) and (371.89,73.19) .. (378.98,71.76) .. controls (386.07,70.33) and (393.02,75.09) .. (394.49,82.4) .. controls (396.27,91.22) and (397.16,95.63) .. (397.16,95.63) .. controls (397.16,95.63) and (392.88,96.5) .. (384.32,98.23) -- cycle ;
			\draw   (400.72,83.03) .. controls (400.72,83.03) and (400.72,83.03) .. (400.72,83.03) .. controls (402.68,76.06) and (410.09,72.06) .. (417.27,74.08) .. controls (424.44,76.11) and (428.67,83.39) .. (426.7,90.36) .. controls (424.74,97.32) and (417.33,101.32) .. (410.15,99.3) .. controls (401.49,96.86) and (397.16,95.63) .. (397.16,95.63) .. controls (397.16,95.63) and (398.35,91.43) .. (400.72,83.03) -- cycle ;
			\draw   (406.57,104.74) .. controls (406.57,104.74) and (406.57,104.74) .. (406.57,104.74) .. controls (411.77,109.78) and (411.79,118.2) .. (406.6,123.56) .. controls (401.42,128.91) and (393,129.18) .. (387.8,124.15) .. controls (382.6,119.12) and (382.59,110.69) .. (387.77,105.34) .. controls (394.04,98.87) and (397.17,95.64) .. (397.16,95.63) .. controls (397.17,95.64) and (400.3,98.67) .. (406.57,104.74) -- cycle ;
			\draw   (162.17,96.7) .. controls (162.17,71.5) and (206.6,51.07) .. (261.4,51.07) .. controls (316.21,51.07) and (360.64,71.5) .. (360.64,96.7) .. controls (360.64,121.9) and (316.21,142.33) .. (261.4,142.33) .. controls (206.6,142.33) and (162.17,121.9) .. (162.17,96.7) -- cycle ;
			\draw   (247.48,32.71) .. controls (247.48,25.99) and (252.92,20.55) .. (259.63,20.55) .. controls (266.34,20.55) and (271.79,25.99) .. (271.79,32.71) .. controls (271.79,39.42) and (266.34,44.86) .. (259.63,44.86) .. controls (252.92,44.86) and (247.48,39.42) .. (247.48,32.71) -- cycle ;
			\draw   (271.79,32.71) .. controls (271.79,25.99) and (277.23,20.55) .. (283.94,20.55) .. controls (290.65,20.55) and (296.1,25.99) .. (296.1,32.71) .. controls (296.1,39.42) and (290.65,44.86) .. (283.94,44.86) .. controls (277.23,44.86) and (271.79,39.42) .. (271.79,32.71) -- cycle ;
			\draw   (223.17,32.71) .. controls (223.17,25.99) and (228.61,20.55) .. (235.32,20.55) .. controls (242.03,20.55) and (247.48,25.99) .. (247.48,32.71) .. controls (247.48,39.42) and (242.03,44.86) .. (235.32,44.86) .. controls (228.61,44.86) and (223.17,39.42) .. (223.17,32.71) -- cycle ;
			\draw    (162.17,96.7) -- (360.64,96.7) ;
			\draw   (247.48,81.71) .. controls (247.48,74.99) and (252.92,69.55) .. (259.63,69.55) .. controls (266.34,69.55) and (271.79,74.99) .. (271.79,81.71) .. controls (271.79,88.42) and (266.34,93.86) .. (259.63,93.86) .. controls (252.92,93.86) and (247.48,88.42) .. (247.48,81.71) -- cycle ;
			\draw   (271.79,81.71) .. controls (271.79,74.99) and (277.23,69.55) .. (283.94,69.55) .. controls (290.65,69.55) and (296.1,74.99) .. (296.1,81.71) .. controls (296.1,88.42) and (290.65,93.86) .. (283.94,93.86) .. controls (277.23,93.86) and (271.79,88.42) .. (271.79,81.71) -- cycle ;
			\draw   (223.17,81.71) .. controls (223.17,74.99) and (228.61,69.55) .. (235.32,69.55) .. controls (242.03,69.55) and (247.48,74.99) .. (247.48,81.71) .. controls (247.48,88.42) and (242.03,93.86) .. (235.32,93.86) .. controls (228.61,93.86) and (223.17,88.42) .. (223.17,81.71) -- cycle ;
			\draw   (247.48,157.71) .. controls (247.48,150.99) and (252.92,145.55) .. (259.63,145.55) .. controls (266.34,145.55) and (271.79,150.99) .. (271.79,157.71) .. controls (271.79,164.42) and (266.34,169.86) .. (259.63,169.86) .. controls (252.92,169.86) and (247.48,164.42) .. (247.48,157.71) -- cycle ;
			\draw   (271.79,157.71) .. controls (271.79,150.99) and (277.23,145.55) .. (283.94,145.55) .. controls (290.65,145.55) and (296.1,150.99) .. (296.1,157.71) .. controls (296.1,164.42) and (290.65,169.86) .. (283.94,169.86) .. controls (277.23,169.86) and (271.79,164.42) .. (271.79,157.71) -- cycle ;
			\draw   (223.17,157.71) .. controls (223.17,150.99) and (228.61,145.55) .. (235.32,145.55) .. controls (242.03,145.55) and (247.48,150.99) .. (247.48,157.71) .. controls (247.48,164.42) and (242.03,169.86) .. (235.32,169.86) .. controls (228.61,169.86) and (223.17,164.42) .. (223.17,157.71) -- cycle ;
			\draw   (172.32,266.23) .. controls (165.23,267.66) and (158.28,262.89) .. (156.81,255.58) .. controls (155.33,248.28) and (159.89,241.19) .. (166.98,239.76) .. controls (174.07,238.33) and (181.02,243.09) .. (182.49,250.4) .. controls (184.27,259.22) and (185.16,263.63) .. (185.16,263.63) .. controls (185.16,263.63) and (180.88,264.5) .. (172.32,266.23) -- cycle ;
			\draw   (188.72,251.03) .. controls (188.72,251.03) and (188.72,251.03) .. (188.72,251.03) .. controls (190.68,244.06) and (198.09,240.06) .. (205.27,242.08) .. controls (212.44,244.11) and (216.67,251.39) .. (214.7,258.36) .. controls (212.74,265.32) and (205.33,269.32) .. (198.15,267.3) .. controls (189.49,264.86) and (185.16,263.63) .. (185.16,263.63) .. controls (185.16,263.63) and (186.35,259.43) .. (188.72,251.03) -- cycle ;
			\draw   (194.57,272.74) .. controls (194.57,272.74) and (194.57,272.74) .. (194.57,272.74) .. controls (199.77,277.78) and (199.79,286.2) .. (194.6,291.56) .. controls (189.42,296.91) and (181,297.18) .. (175.8,292.15) .. controls (170.6,287.12) and (170.59,278.69) .. (175.77,273.34) .. controls (182.04,266.87) and (185.17,263.64) .. (185.16,263.63) .. controls (185.17,263.64) and (188.3,266.67) .. (194.57,272.74) -- cycle ;
			\draw  [fill={rgb, 255:red, 0; green, 0; blue, 0 }  ,fill opacity=1 ] (221.39,264.7) .. controls (221.39,262.62) and (223.08,260.93) .. (225.17,260.93) .. controls (227.25,260.93) and (228.94,262.62) .. (228.94,264.7) .. controls (228.94,266.78) and (227.25,268.47) .. (225.17,268.47) .. controls (223.08,268.47) and (221.39,266.78) .. (221.39,264.7) -- cycle ;
			\draw   (305.48,247.71) .. controls (305.48,240.99) and (310.92,235.55) .. (317.63,235.55) .. controls (324.34,235.55) and (329.79,240.99) .. (329.79,247.71) .. controls (329.79,254.42) and (324.34,259.86) .. (317.63,259.86) .. controls (310.92,259.86) and (305.48,254.42) .. (305.48,247.71) -- cycle ;
			\draw   (329.79,247.71) .. controls (329.79,240.99) and (335.23,235.55) .. (341.94,235.55) .. controls (348.65,235.55) and (354.1,240.99) .. (354.1,247.71) .. controls (354.1,254.42) and (348.65,259.86) .. (341.94,259.86) .. controls (335.23,259.86) and (329.79,254.42) .. (329.79,247.71) -- cycle ;
			\draw   (281.17,247.71) .. controls (281.17,240.99) and (286.61,235.55) .. (293.32,235.55) .. controls (300.03,235.55) and (305.48,240.99) .. (305.48,247.71) .. controls (305.48,254.42) and (300.03,259.86) .. (293.32,259.86) .. controls (286.61,259.86) and (281.17,254.42) .. (281.17,247.71) -- cycle ;
			\draw [color={rgb, 255:red, 208; green, 2; blue, 27 }  ,draw opacity=1 ][line width=1.5]    (261.4,47.07) -- (261.4,56.33) ;
			\draw    (225.17,264.7) -- (423.64,264.7) ;
			\draw [color={rgb, 255:red, 208; green, 2; blue, 27 }  ,draw opacity=1 ][line width=1.5]    (423.64,260.7) -- (423.64,269.97) ;
			
			\draw (132,78.46) node [anchor=north west][inner sep=0.75pt]   [align=left] {$\displaystyle 3$};
			\draw (101,78.46) node [anchor=north west][inner sep=0.75pt]   [align=left] {$\displaystyle 1$};
			\draw (116.77,105.34) node [anchor=north west][inner sep=0.75pt]   [align=left] {$\displaystyle 2$};
			\draw (392,103.46) node [anchor=north west][inner sep=0.75pt]   [align=left] {$\displaystyle 3$};
			\draw (376,78.46) node [anchor=north west][inner sep=0.75pt]   [align=left] {$\displaystyle 1$};
			\draw (408,78.46) node [anchor=north west][inner sep=0.75pt]   [align=left] {$\displaystyle 2$};
			\draw (230.79,23.5) node [anchor=north west][inner sep=0.75pt]   [align=left] {$\displaystyle 1$};
			\draw (254.24,23.5) node [anchor=north west][inner sep=0.75pt]   [align=left] {$\displaystyle 2$};
			\draw (278.43,23.5) node [anchor=north west][inner sep=0.75pt]   [align=left] {$\displaystyle 3$};
			\draw (230.79,72.5) node [anchor=north west][inner sep=0.75pt]   [align=left] {$\displaystyle 1$};
			\draw (254.24,72.5) node [anchor=north west][inner sep=0.75pt]   [align=left] {$\displaystyle 3$};
			\draw (278.43,72.5) node [anchor=north west][inner sep=0.75pt]   [align=left] {$\displaystyle 2$};
			\draw (230.79,148.5) node [anchor=north west][inner sep=0.75pt]   [align=left] {$\displaystyle 2$};
			\draw (254.24,148.5) node [anchor=north west][inner sep=0.75pt]   [align=left] {$\displaystyle 1$};
			\draw (278.43,148.5) node [anchor=north west][inner sep=0.75pt]   [align=left] {$\displaystyle 3$};
			\draw (13,89) node [anchor=north west][inner sep=0.75pt]   [align=left] {$\displaystyle \mathcal{C}_{3} /S^{1} =$};
			\draw (14,259) node [anchor=north west][inner sep=0.75pt]   [align=left] {$\displaystyle \mathcal{C}_{3} /\left( S^{1} \times \Sigma _{3}\right) =$};

		\end{tikzpicture}
		
		\caption{On top of this picture we see the CW-complex $\mathcal{C}_3/S^1$; the red segment indicate the middle of the $1$-cell. $\mathcal{C}_3/(S^1\times \Sigma_3)$ is depicted below and it is obtained from $\mathcal{C}_3/S^1$ by quotienting the $\Sigma_3$-action: the two 
			$0$-cells of $\mathcal{C}_3/S^1$ are identified, and the same happens for the three $1$-cells. Since the one cells have as stabilizer $\Z/2$, we have an additional identification: we need to glue together the two halves of any $1$-cell.}
		\label{fig:esempio con tre lobi}
	\end{figure}
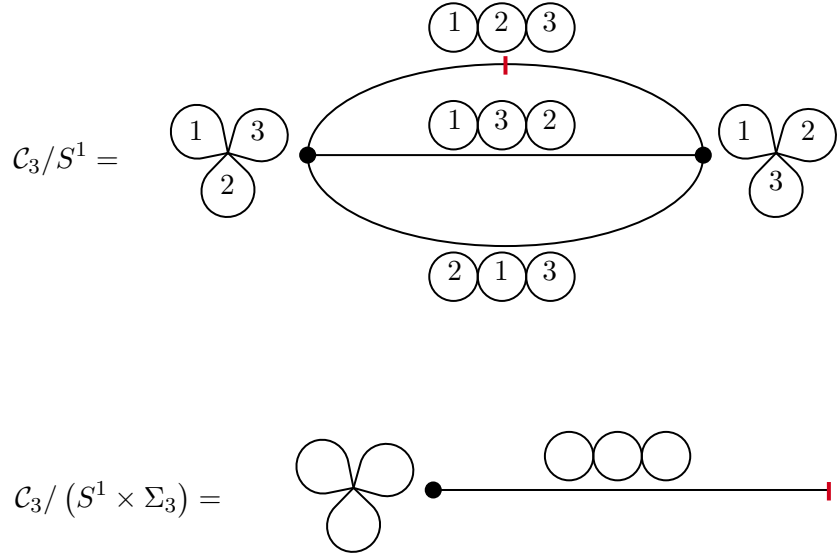
	\begin{es}[n=4] There are four cacti (unlabelled, without basepoint) with four lobes, let us call them $c_0$, $c_1$, $c_2$ and $c_3$ (see Figure \ref{fig:alcune celle e i rispettivi stabilizzatori} where such cacti are depicted with an arbitrary labelling of the lobes). $c_0$ is fixed by a rotation of $2\pi/4$, so its stabilizer is $\Z/4$. The corresponding cell is a point. $c_1$ is not fixed by any rotation, so its corresponding cell $\sigma(c_1)$ is a copy of $\Delta^1$.  $c_2$ is fixed by a rotation of $\pi$, so its stabilizer is $\Z/2$. The corresponding cell $\sigma(c_2)$ is obtained from $\Delta^1\times\Delta^1$ by imposing the relation
		\[
		(t_0,t_1)\times (s_0,s_1)\sim (s_1,s_0)\times (t_1,t_0)
		\]
		Finally, $c_3$ is fixed by a rotation of $2\pi/3$, so its stabilizer is $\Z/3$. The corresponding cell $\sigma(c_3)$ is obtained from $\Delta^2$ by quotienting the $\Z/3$-action. To be explicit, the generator of $\Z/3$ acts on $\Delta^2$ by permuting cyclically the coordinates $(t_0,t_1,t_2)$. 
	\end{es}

\section{Fundamental group}\label{sec:fundamental group and rational homology}

In this section we prove that $\mathcal{M}_{0,n+1}/\Sigma_n$ and $\overline{\M}_{0,n+1}/\Sigma_n$ are simply connected. First we recall some facts on the fundamental group of orbit spaces.
	\subsection{Fundamental group of orbit spaces}
	
	The goal of this paragraph is to explain the relation between $\pi_1(X)$ and $\pi_1(X/G)$, when $G$ is a topological group acting on a sufficiently nice topological space $X$. The following result can be found in the paper of B. Noohi \cite{Noohi}, which extends previous work by M.A. Armstrong \cite{Armstrong}. 
	
		\begin{thm}[\cite{Noohi}, p. 23]\label{thm:sequenza lunga per gruppo fondamentale orbit spaces}
		Let $G$ be a compact Lie group acting on a connected topological manifold $X$. We also assume that $X/G$ is semilocally simply connected. Fix a base point $x_0\in X$ and let $[x_0]$ be its image in $X/G$. Then we have an exact sequence
		\[
		\begin{tikzcd}
			&\pi_1(X,x_0)\arrow[r] &\pi_1(X/G,[x_0])\arrow[r]&\pi_0(G)/I\arrow[r] & 1
		\end{tikzcd}
		\]
		where $I\subset \pi_0(G)$ is the subgroup generated by the path components of $G$ containing an element $g$ which fix some point of $X$.
	\end{thm}
	
\begin{oss}
		This Theorem holds with weaker hypothesis on $G$ and $X$. Indeed it suffices that $X$ is connected, locally path connected and semilocally simply connected and that the $G$-action has the \emph{slice property}. For details see \cite{Palais} and  \cite{Noohi}.
	\end{oss}

\subsection{Computations}
We now apply Theorem \ref{thm:sequenza lunga per gruppo fondamentale orbit spaces} to compute the fundamental group of $\overline{\mathcal{M}}_{0,n+1}/\Sigma_n$ and $\mathcal{M}_{0,n+1}/\Sigma_n$.
	\begin{thm}\label{thm: compattificati sono semplicemente connessi}
		The quotient $\overline{\mathcal{M}}_{0,n+1}/\Sigma_n$ is simply connected.
	\end{thm}
	\begin{proof}
		
		Since $\overline{\mathcal{M}}_{0,n+1}$ is simply connected, Theorem \ref{thm:sequenza lunga per gruppo fondamentale orbit spaces} gives us an isomorphism between $\pi_1(\overline{\mathcal{M}}_{0,n+1}/\Sigma_n)$ and $\Sigma_n/I$. Let $(i,j)\in \Sigma_n$ be a transposition. It has at least one fixed point since it fixes a configuration whose clustering is of the form $((i,j),1\dots,\hat{i},\dots,\hat{j},\dots,n)$. Therefore the subgroup $I$ contains all the transpositions, so it must be $\Sigma_n$ itself.  This concludes the proof.
	\end{proof}
\begin{thm}\label{thm: quozienti stretti sono semplicemente connessi}
    The quotient $\mathcal{M}_{0,n+1}/\Sigma_n$ is simply connected.
\end{thm}
\begin{proof}
     The space $\mathcal{M}_{0,n+1}/\Sigma_n$ is homotopy equivalent to $C_n(\C)/S^1$, therefore it suffice to prove statement for this last space. Let us denote by $p:C_n(\C)\to C_n(\C)/S^1$ the quotient map. The proof follows several steps:
     \begin{enumerate}
         \item Since $S^1$ is connected we get that $\pi_0(G)/I=1$ in the exact sequence of Theorem \ref{thm:sequenza lunga per gruppo fondamentale orbit spaces}. This implies that $p_*: B_n\to \pi_1(C_n(\C)/S^1)$ is surjective. Therefore we have a short exact sequence
         \[
         \begin{tikzcd}
             & 1\arrow[r] &Ker(p_*) \arrow[r,hook]& B_n \arrow[r, "p_*"] &\pi_1(C_n(\C)/S^1) \arrow[r] & 1
         \end{tikzcd}
         \]
        
         \item $\pi_1(C_n(\C)/S^1)$ is abelian: the previous step tell us that $\pi_1(C_n(\C)/S^1)$ is a quotient of the braid group, in particular it is generated by elements $\sigma_1,\dots,\sigma_{n-1}$  and in addition to the relations of $B_n$ there are some extra relations coming from $Ker(p_*)$. Let us make explicit some of these relations: the braid $\Delta\coloneqq \sigma_1(\sigma_2\sigma_1)\dots(\sigma_{n-1}\sigma_{n-2}\dots\sigma_1)$ belongs to $Ker(p_*)$ since it is the given by a rotation of $\pi$. Therefore we have the relation 
         \[
         \sigma_i=\sigma_i\Delta=\Delta\sigma_{n-i}=\sigma_{n-i}
         \]
         Now suppose $n$ is odd. The previous relation enable us to prove that $\sigma_i$ and $\sigma_{i+1}$ commute in $\pi_1(C_n(\C)/S^1)$:
         \[
         \sigma_i\sigma_{i+1}= \sigma_{n-i}\sigma_{i+1}=\sigma_{i+1}\sigma_{n-i}=\sigma_{i+1}\sigma_i
         \]
         where the middle equality holds because $n$ is odd and therefore $n-i\neq i$ for each $i=1,\dots,n-1$. To get the statement in the case $n=2k$ we can do the same procedure to show that $\sigma_i$ and $\sigma_{i+1}$ commute for each $i\neq k$. So it remains to prove the equality $ \sigma_k\sigma_{k+1}=\sigma_{k+1}\sigma_{k}$: combining $\sigma_{k+1}\sigma_{k+2}\sigma_{k+1}=\sigma_{k+2}\sigma_{k+1}\sigma_{k+2}$ and  $ \sigma_{k+2}\sigma_{k+1}=\sigma_{k+1}\sigma_{k+2}$  we get that $ \sigma_{k+2}=\sigma_{k+1}$. Therefore $ \sigma_k\sigma_{k+1}=\sigma_k\sigma_{k+2}=\sigma_{k+2}\sigma_{k}=\sigma_{k+1}\sigma_{k}$.
         \item $\pi_1(C_n(\C)/S^1)$ is generated by $\sigma_1$: by the previous point $\pi_1(C_n(\C)/S^1)$ is abelian. Combining this fact with the braid relation $\sigma_{i}\sigma_{i+1}\sigma_{i}=\sigma_{i+1}\sigma_{i}\sigma_{i+1}$ we get $\sigma_i=\sigma_{i+1}$ for all $i=1,\dots,n-2$. Thus $\pi_1(C_n(\C)/S^1)$ is generated by $\sigma_1$.
         \item $\pi_1(C_n(\C)/S^1)$ is the trivial group: by the previous discussion we know that $\pi_1(C_n(\C)/S^1)$ is an abelian group, so it is isomorphic to $H_1(C_n(\C)/S^1;\Z)$ by Hurewicz. Moreover $\pi_1(C_n(\C)/S^1)$ it is generated by $\sigma_1$. So it suffice to show that $\sigma_1=0$. We proceed  as follows: consider the $n$-th roots of unity $\{\zeta_1,\dots,\zeta_n\}\in C_n(\C)$. Now take the loop 
         \begin{align*}
             \alpha:[0,1]&\to C_n(\C)\\
             t&\mapsto \{e^{2t\pi i/n}\zeta_1,\dots,e^{2t\pi i/n}\zeta_n\}
         \end{align*}
In plain words $\alpha$ rotates the $n$-agon $\{\zeta_1,\dots,\zeta_n\}$ counter-clockwise between $0$ and $2\pi/n$ degrees. It is easy to see that this loop represent the class $(n-1)\sigma_1\in H_1(C_n(\C);\Z)$ (see figure \ref{fig:laccio dell'n agono} for a picture). Therefore $p_*(\alpha)$ is $(n-1)$ times the generator of $H_1(C_n(\C)/S^1;\Z)$. But $p_*(\alpha)$ is a constant loop, so it is the zero class in homology. Therefore we get the equation $(n-1)\sigma_1=0$ in $H_1(C_n(\C)/S^1;\Z)$. Similarly, let $\{\zeta_1,\dots,\zeta_{n-1}\}$ be the set of $(n-1)$-th roots of unity. Consider the loop
  \begin{align*}
             \beta:[0,1]&\to C_n(\C)\\
             t&\mapsto \{e^{2t\pi i/(n-1)}\zeta_1,\dots,e^{2t\pi i/(n-1)}\zeta_{n-1},0\}
         \end{align*}
In plain words $\alpha$ rotates the configuration $\{\zeta_1,\dots,\zeta_{n-1},0\}$ counter-clockwise between $0$ and $2\pi/(n-1)$ degrees. It is easy to see that this loop represent the class $n\sigma_1\in H_1(C_n(\C);\Z)$ (see Figure \ref{fig:laccio dell'n agono}). Therefore $p_*(\alpha)$ is $n$ times the generator of $H_1(C_n(\C)/S^1;\Z)$. But $p_*(\alpha)$ is a constant loop, so it is the zero class in homology. Therefore we get the equation $n\sigma_1=0$ in $H_1(C_n(\C)/S^1;\Z)$. Finally we can conclude: we know that $\pi_1(C_n(\C)/S^1)\cong H_1(C_n(\C)/S^1;\Z)$ is a cyclic group with generator $\sigma_1$. However $n\sigma_1=0=(n-1)\sigma_1$, therefore $\sigma_1=0$.

\begin{figure}
   \centering
    \caption{On the left the loop $\alpha$, on the right $\beta$.}
    \label{fig:laccio dell'n agono}

\tikzset{every picture/.style={line width=0.75pt}} 

\begin{tikzpicture}[x=0.75pt,y=0.75pt,yscale=-1,xscale=1]

\draw  [dash pattern={on 4.5pt off 4.5pt}] (139.88,39) -- (204.36,85.92) -- (179.67,161.74) -- (99.92,161.68) -- (75.33,85.83) -- cycle ;
\draw  [dash pattern={on 4.5pt off 4.5pt}] (139.88,264) -- (204.36,310.92) -- (179.67,386.74) -- (99.92,386.68) -- (75.33,310.83) -- cycle ;
\draw    (99.92,161.68) .. controls (117.67,211.33) and (178.67,212) .. (179.67,383.41) ;
\draw [shift={(179.67,383.41)}, rotate = 269.67] [color={rgb, 255:red, 0; green, 0; blue, 0 }  ][line width=0.75]    (10.93,-3.29) .. controls (6.95,-1.4) and (3.31,-0.3) .. (0,0) .. controls (3.31,0.3) and (6.95,1.4) .. (10.93,3.29)   ;
\draw  [fill={rgb, 255:red, 0; green, 0; blue, 0 }  ,fill opacity=1 ] (136.55,39) .. controls (136.55,37.16) and (138.04,35.67) .. (139.88,35.67) .. controls (141.72,35.67) and (143.21,37.16) .. (143.21,39) .. controls (143.21,40.84) and (141.72,42.33) .. (139.88,42.33) .. controls (138.04,42.33) and (136.55,40.84) .. (136.55,39) -- cycle ;
\draw  [fill={rgb, 255:red, 0; green, 0; blue, 0 }  ,fill opacity=1 ] (201.03,85.92) .. controls (201.03,84.08) and (202.52,82.58) .. (204.36,82.58) .. controls (206.2,82.58) and (207.69,84.08) .. (207.69,85.92) .. controls (207.69,87.76) and (206.2,89.25) .. (204.36,89.25) .. controls (202.52,89.25) and (201.03,87.76) .. (201.03,85.92) -- cycle ;
\draw  [fill={rgb, 255:red, 0; green, 0; blue, 0 }  ,fill opacity=1 ] (176.33,161.74) .. controls (176.33,159.9) and (177.83,158.41) .. (179.67,158.41) .. controls (181.51,158.41) and (183,159.9) .. (183,161.74) .. controls (183,163.58) and (181.51,165.07) .. (179.67,165.07) .. controls (177.83,165.07) and (176.33,163.58) .. (176.33,161.74) -- cycle ;
\draw  [fill={rgb, 255:red, 0; green, 0; blue, 0 }  ,fill opacity=1 ] (96.59,161.68) .. controls (96.59,159.84) and (98.08,158.35) .. (99.92,158.35) .. controls (101.76,158.35) and (103.26,159.84) .. (103.26,161.68) .. controls (103.26,163.52) and (101.76,165.02) .. (99.92,165.02) .. controls (98.08,165.02) and (96.59,163.52) .. (96.59,161.68) -- cycle ;
\draw  [fill={rgb, 255:red, 0; green, 0; blue, 0 }  ,fill opacity=1 ] (72,85.83) .. controls (72,83.99) and (73.49,82.49) .. (75.33,82.49) .. controls (77.18,82.49) and (78.67,83.99) .. (78.67,85.83) .. controls (78.67,87.67) and (77.18,89.16) .. (75.33,89.16) .. controls (73.49,89.16) and (72,87.67) .. (72,85.83) -- cycle ;
\draw  [fill={rgb, 255:red, 0; green, 0; blue, 0 }  ,fill opacity=1 ] (136.55,264) .. controls (136.55,262.16) and (138.04,260.67) .. (139.88,260.67) .. controls (141.72,260.67) and (143.21,262.16) .. (143.21,264) .. controls (143.21,265.84) and (141.72,267.33) .. (139.88,267.33) .. controls (138.04,267.33) and (136.55,265.84) .. (136.55,264) -- cycle ;
\draw  [fill={rgb, 255:red, 0; green, 0; blue, 0 }  ,fill opacity=1 ] (201.03,310.92) .. controls (201.03,309.08) and (202.52,307.58) .. (204.36,307.58) .. controls (206.2,307.58) and (207.69,309.08) .. (207.69,310.92) .. controls (207.69,312.76) and (206.2,314.25) .. (204.36,314.25) .. controls (202.52,314.25) and (201.03,312.76) .. (201.03,310.92) -- cycle ;
\draw  [fill={rgb, 255:red, 0; green, 0; blue, 0 }  ,fill opacity=1 ] (176.33,386.74) .. controls (176.33,384.9) and (177.83,383.41) .. (179.67,383.41) .. controls (181.51,383.41) and (183,384.9) .. (183,386.74) .. controls (183,388.58) and (181.51,390.07) .. (179.67,390.07) .. controls (177.83,390.07) and (176.33,388.58) .. (176.33,386.74) -- cycle ;
\draw  [fill={rgb, 255:red, 0; green, 0; blue, 0 }  ,fill opacity=1 ] (96.59,386.68) .. controls (96.59,384.84) and (98.08,383.35) .. (99.92,383.35) .. controls (101.76,383.35) and (103.26,384.84) .. (103.26,386.68) .. controls (103.26,388.52) and (101.76,390.02) .. (99.92,390.02) .. controls (98.08,390.02) and (96.59,388.52) .. (96.59,386.68) -- cycle ;
\draw  [fill={rgb, 255:red, 0; green, 0; blue, 0 }  ,fill opacity=1 ] (72,310.83) .. controls (72,308.99) and (73.49,307.49) .. (75.33,307.49) .. controls (77.18,307.49) and (78.67,308.99) .. (78.67,310.83) .. controls (78.67,312.67) and (77.18,314.16) .. (75.33,314.16) .. controls (73.49,314.16) and (72,312.67) .. (72,310.83) -- cycle ;
\draw    (179.67,161.74) .. controls (191.61,211.09) and (205.53,228.82) .. (204.38,306.41) ;
\draw [shift={(204.36,307.58)}, rotate = 270.95] [color={rgb, 255:red, 0; green, 0; blue, 0 }  ][line width=0.75]    (10.93,-3.29) .. controls (6.95,-1.4) and (3.31,-0.3) .. (0,0) .. controls (3.31,0.3) and (6.95,1.4) .. (10.93,3.29)   ;
\draw    (75.33,85.83) .. controls (68.7,237.57) and (109.26,280.57) .. (100.07,381.82) ;
\draw [shift={(99.92,383.35)}, rotate = 275.44] [color={rgb, 255:red, 0; green, 0; blue, 0 }  ][line width=0.75]    (10.93,-3.29) .. controls (6.95,-1.4) and (3.31,-0.3) .. (0,0) .. controls (3.31,0.3) and (6.95,1.4) .. (10.93,3.29)   ;
\draw    (139.88,39) .. controls (142.67,131.89) and (140.67,153.89) .. (127.67,198.11) ;
\draw    (204.36,85.92) .. controls (205.67,142.89) and (202.67,158.89) .. (190.67,189.89) ;
\draw    (186.36,195.92) .. controls (176.67,214.89) and (164.67,210.89) .. (151.67,231.89) ;
\draw    (85.67,259.74) .. controls (74.83,280.57) and (77.58,270.86) .. (75.44,305.86) ;
\draw [shift={(75.33,307.49)}, rotate = 273.61] [color={rgb, 255:red, 0; green, 0; blue, 0 }  ][line width=0.75]    (10.93,-3.29) .. controls (6.95,-1.4) and (3.31,-0.3) .. (0,0) .. controls (3.31,0.3) and (6.95,1.4) .. (10.93,3.29)   ;
\draw    (124.67,204.89) .. controls (116.67,225.89) and (105.67,231.89) .. (92.67,252.89) ;
\draw    (147.67,235.89) .. controls (144.8,240.81) and (144.67,246.36) .. (140.5,258.85) ;
\draw [shift={(139.88,260.67)}, rotate = 289.16] [color={rgb, 255:red, 0; green, 0; blue, 0 }  ][line width=0.75]    (10.93,-3.29) .. controls (6.95,-1.4) and (3.31,-0.3) .. (0,0) .. controls (3.31,0.3) and (6.95,1.4) .. (10.93,3.29)   ;
\draw  [dash pattern={on 4.5pt off 4.5pt}] (330.16,87.65) -- (418.02,38.16) -- (467.5,126.02) -- (379.65,175.5) -- cycle ;
\draw  [dash pattern={on 4.5pt off 4.5pt}] (330.16,312.65) -- (418.02,263.16) -- (467.5,351.02) -- (379.65,400.5) -- cycle ;
\draw  [fill={rgb, 255:red, 0; green, 0; blue, 0 }  ,fill opacity=1 ] (414.69,38.16) .. controls (414.69,36.32) and (416.18,34.83) .. (418.02,34.83) .. controls (419.86,34.83) and (421.35,36.32) .. (421.35,38.16) .. controls (421.35,40) and (419.86,41.5) .. (418.02,41.5) .. controls (416.18,41.5) and (414.69,40) .. (414.69,38.16) -- cycle ;
\draw  [fill={rgb, 255:red, 0; green, 0; blue, 0 }  ,fill opacity=1 ] (326.83,87.65) .. controls (326.83,85.81) and (328.32,84.31) .. (330.16,84.31) .. controls (332,84.31) and (333.5,85.81) .. (333.5,87.65) .. controls (333.5,89.49) and (332,90.98) .. (330.16,90.98) .. controls (328.32,90.98) and (326.83,89.49) .. (326.83,87.65) -- cycle ;
\draw  [fill={rgb, 255:red, 0; green, 0; blue, 0 }  ,fill opacity=1 ] (464.17,126.02) .. controls (464.17,124.18) and (465.66,122.69) .. (467.5,122.69) .. controls (469.34,122.69) and (470.84,124.18) .. (470.84,126.02) .. controls (470.84,127.86) and (469.34,129.35) .. (467.5,129.35) .. controls (465.66,129.35) and (464.17,127.86) .. (464.17,126.02) -- cycle ;
\draw  [fill={rgb, 255:red, 0; green, 0; blue, 0 }  ,fill opacity=1 ] (376.31,175.5) .. controls (376.31,173.66) and (377.81,172.17) .. (379.65,172.17) .. controls (381.49,172.17) and (382.98,173.66) .. (382.98,175.5) .. controls (382.98,177.34) and (381.49,178.84) .. (379.65,178.84) .. controls (377.81,178.84) and (376.31,177.34) .. (376.31,175.5) -- cycle ;
\draw  [fill={rgb, 255:red, 0; green, 0; blue, 0 }  ,fill opacity=1 ] (395.5,106.83) .. controls (395.5,104.99) and (396.99,103.5) .. (398.83,103.5) .. controls (400.67,103.5) and (402.17,104.99) .. (402.17,106.83) .. controls (402.17,108.67) and (400.67,110.17) .. (398.83,110.17) .. controls (396.99,110.17) and (395.5,108.67) .. (395.5,106.83) -- cycle ;
\draw  [fill={rgb, 255:red, 0; green, 0; blue, 0 }  ,fill opacity=1 ] (376.31,400.5) .. controls (376.31,398.66) and (377.81,397.17) .. (379.65,397.17) .. controls (381.49,397.17) and (382.98,398.66) .. (382.98,400.5) .. controls (382.98,402.34) and (381.49,403.84) .. (379.65,403.84) .. controls (377.81,403.84) and (376.31,402.34) .. (376.31,400.5) -- cycle ;
\draw  [fill={rgb, 255:red, 0; green, 0; blue, 0 }  ,fill opacity=1 ] (464.17,351.02) .. controls (464.17,349.18) and (465.66,347.69) .. (467.5,347.69) .. controls (469.34,347.69) and (470.84,349.18) .. (470.84,351.02) .. controls (470.84,352.86) and (469.34,354.35) .. (467.5,354.35) .. controls (465.66,354.35) and (464.17,352.86) .. (464.17,351.02) -- cycle ;
\draw  [fill={rgb, 255:red, 0; green, 0; blue, 0 }  ,fill opacity=1 ] (326.83,312.65) .. controls (326.83,310.81) and (328.32,309.31) .. (330.16,309.31) .. controls (332,309.31) and (333.5,310.81) .. (333.5,312.65) .. controls (333.5,314.49) and (332,315.98) .. (330.16,315.98) .. controls (328.32,315.98) and (326.83,314.49) .. (326.83,312.65) -- cycle ;
\draw  [fill={rgb, 255:red, 0; green, 0; blue, 0 }  ,fill opacity=1 ] (414.69,263.16) .. controls (414.69,261.32) and (416.18,259.83) .. (418.02,259.83) .. controls (419.86,259.83) and (421.35,261.32) .. (421.35,263.16) .. controls (421.35,265) and (419.86,266.5) .. (418.02,266.5) .. controls (416.18,266.5) and (414.69,265) .. (414.69,263.16) -- cycle ;
\draw  [fill={rgb, 255:red, 0; green, 0; blue, 0 }  ,fill opacity=1 ] (395.5,331.83) .. controls (395.5,329.99) and (396.99,328.5) .. (398.83,328.5) .. controls (400.67,328.5) and (402.17,329.99) .. (402.17,331.83) .. controls (402.17,333.67) and (400.67,335.17) .. (398.83,335.17) .. controls (396.99,335.17) and (395.5,333.67) .. (395.5,331.83) -- cycle ;
\draw    (330.16,87.65) .. controls (323.5,240.15) and (379.67,208.67) .. (379.65,397.17) ;
\draw [shift={(379.65,397.17)}, rotate = 270.01] [color={rgb, 255:red, 0; green, 0; blue, 0 }  ][line width=0.75]    (10.93,-3.29) .. controls (6.95,-1.4) and (3.31,-0.3) .. (0,0) .. controls (3.31,0.3) and (6.95,1.4) .. (10.93,3.29)   ;
\draw    (379.65,175.5) .. controls (404.67,242.67) and (469.67,190.67) .. (467.5,347.69) ;
\draw [shift={(467.5,347.69)}, rotate = 270.79] [color={rgb, 255:red, 0; green, 0; blue, 0 }  ][line width=0.75]    (10.93,-3.29) .. controls (6.95,-1.4) and (3.31,-0.3) .. (0,0) .. controls (3.31,0.3) and (6.95,1.4) .. (10.93,3.29)   ;
\draw    (466.5,127.02) .. controls (467.81,183.99) and (442.67,190.67) .. (427.67,219.67) ;
\draw    (424.67,225.67) .. controls (421.77,230.64) and (418.88,240.91) .. (418.1,257.95) ;
\draw [shift={(418.02,259.83)}, rotate = 272.04] [color={rgb, 255:red, 0; green, 0; blue, 0 }  ][line width=0.75]    (10.93,-3.29) .. controls (6.95,-1.4) and (3.31,-0.3) .. (0,0) .. controls (3.31,0.3) and (6.95,1.4) .. (10.93,3.29)   ;
\draw    (398.83,106.83) .. controls (413.67,160.67) and (412.67,189.67) .. (404.67,202.67) ;
\draw    (399.67,209.67) .. controls (389.77,235.41) and (395.55,282.71) .. (398.74,327.15) ;
\draw [shift={(398.83,328.5)}, rotate = 265.96] [color={rgb, 255:red, 0; green, 0; blue, 0 }  ][line width=0.75]    (10.93,-3.29) .. controls (6.95,-1.4) and (3.31,-0.3) .. (0,0) .. controls (3.31,0.3) and (6.95,1.4) .. (10.93,3.29)   ;
\draw    (418.02,38.16) .. controls (432.85,92) and (421.67,194.67) .. (416.67,212.67) ;
\draw    (415.02,218.16) .. controls (409.67,233.67) and (402.67,237.67) .. (396.67,238.67) ;
\draw    (391.02,240.16) .. controls (375.67,244.67) and (370.67,250.67) .. (363.67,254.67) ;
\draw    (358.67,257.67) .. controls (344.95,268.45) and (335.07,276.35) .. (327.15,303.95) ;
\draw [shift={(326.67,305.67)}, rotate = 285.42] [color={rgb, 255:red, 0; green, 0; blue, 0 }  ][line width=0.75]    (10.93,-3.29) .. controls (6.95,-1.4) and (3.31,-0.3) .. (0,0) .. controls (3.31,0.3) and (6.95,1.4) .. (10.93,3.29)   ;

\draw (60,206) node [anchor=north west][inner sep=0.75pt]   [align=left] {$\displaystyle \alpha $};
\draw (315,202) node [anchor=north west][inner sep=0.75pt]   [align=left] {$\displaystyle \beta $};

\end{tikzpicture}

\end{figure}
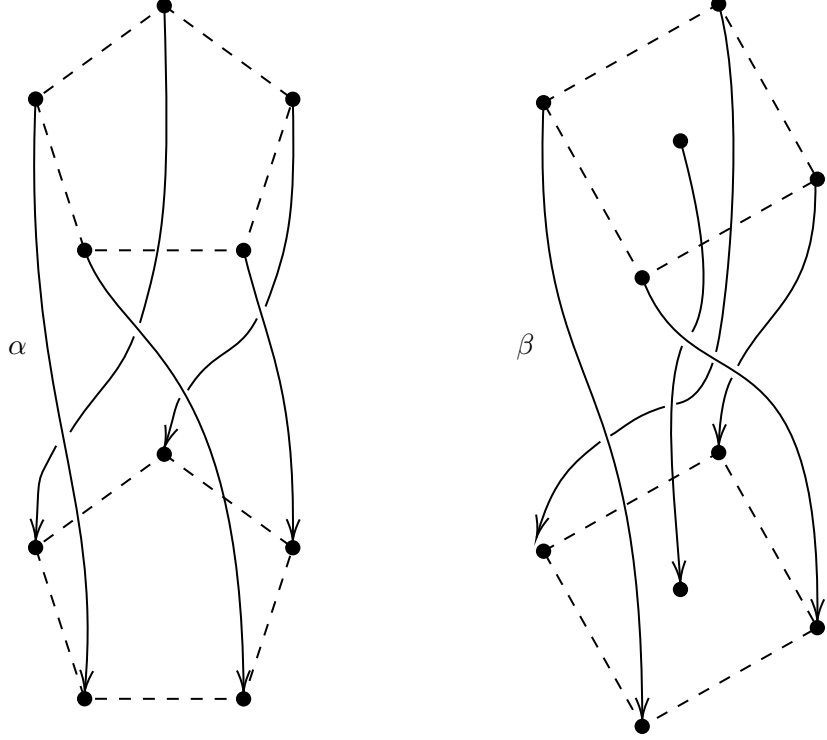

     \end{enumerate}
\end{proof}

\section{On the torsion of $H_*(\mathcal{M}_{0,n+1}/\Sigma_n;\Z)$} 
\label{sec:torsion of strict quotients}
It is well known that rationally $\mathcal{M}_{0,n+1}/\Sigma_n$ has the same (co)homology of a point, so every class in $H_*(\mathcal{M}_{0,n+1}/\Sigma_n;\mathbb{Z})$ is a torsion class. This fact can be deduced, for example, from the work of Getzler \cite{Getzler}: in \cite{Getzler} $H^*(\mathcal{M}_{0,n+1};\mathbb{Q})$ is computed as $\Sigma_{n+1}$-representation, so we can compute  $H^*(\mathcal{M}_{0,n+1}/\Sigma_n;\mathbb{Q})$  by taking the $\Sigma_n$-invariants. 

In this section we prove some results about the torsion of $H_*(\mathcal{M}_{0,n+1}/\Sigma_n;\mathbb{Z})$. In some cases (for small $n$) it is possible to fully compute the integral homology.
  
\subsection{An upper bound on the order of the elements}
We start by giving an upper bound on the torsion that may  appear in $H_*(\mathcal{M}_{0,n+1}/\Sigma_n;\mathbb{Z})$. 

\begin{thm}\label{thm: bound sulla torsione dei quozienti}
    Let $x\in H_*(\mathcal{M}_{0,n+1}/\Sigma_n;\Z)$. Then the order of $x$ divides $n!$.
\end{thm}
\begin{proof}
    The quotient map $p:\mathcal{M}_{0,n+1}\to\mathcal{M}_{0,n+1}/\Sigma_n$ is a $n!$ fold ramified covering in the sense of \cite{Smith}. Therefore we have a transfer map $\tau:H_*(\mathcal{M}_{0,n+1}/\Sigma_n;\Z)\to H_*(\mathcal{M}_{0,n+1};Z)$ such that the diagram 
    \[
    \begin{tikzcd}
        & H_*(\mathcal{M}_{0,n+1}/\Sigma_n;\Z) \arrow[rr,"\cdot n!"] \arrow[dr,"\tau"] &  & H_*(\mathcal{M}_{0,n+1}/\Sigma_n;\Z) \\
        &  & H_*(\mathcal{M}_{0,n+1};\Z) \arrow[ur,"p_*"] &
    \end{tikzcd}
    \]
    commutes. Given $x\in H_*(\mathcal{M}_{0,n+1}/\Sigma_n;\Z) $ we know that it is a torsion class; however $H_*(\mathcal{M}_{0,n+1};\Z)$ is torsion free, so $\tau$ is the zero map and we get the statement.
\end{proof}
In order to get more information on the torsion of $H_*(\mathcal{M}_{0,n+1}/\Sigma_n;\Z)$ we will follow this approach:
\begin{itemize}
    \item We will focus on the calculation of $H_*(\mathcal{M}_{0,n+1}/\Sigma_n;\F_p)$, with $p$ be any prime number.
    \item Note that $\M_{0,n+1}/\Sigma_n$ is homotopy equivalent to $C_n(\C)/S^1$, which is a bit easier to handle (Remark \ref{oss: quoziente stretto è omotopo al quoziente delle configurazioni non ordinate per il cerchio}).
    \item To compute $H_*(\mathcal{M}_{0,n+1}/\Sigma_n;\F_p)\cong H_*(C_n(\C)/S^1;\F_p)$ the main idea is that we can compare it to $H^{S^1}_*(C_n(\C);\F_p)$, which is known by the computation of \cite{Rossi1}.
\end{itemize}
	We will get a complete calculation of $H_*(\mathcal{M}_{0,n+1}/\Sigma_n;\F_p)$ when $n\neq 0,1$ mod $p$ and $n=p,p+1$. Moreover we will prove that $\mathcal{M}_{0,n+1}/\Sigma_n$ is contractible for any $n\leq 5$. Before going into the details of these computations let us discuss the relation between homotopy quotients and strict quotients.
\subsection{Homotopy quotients vs strict quotients} \label{subsec: homotopy quotients vs strict}
In this paragraph we discuss the relation between the homotopy quotient $X_{S^1}$ and the strict quotient $X/S^1$. Most of the results presented are easy consequences of the theory of transformation groups, which is  developed in \cite{BredonGroups} and \cite{TomDieck}.

\begin{prop}[\cite{BredonGroups}, p. 371]\label{prop: relazione tra quoziente omotopico e stretto}
    Let $X$ be a $S^1$-space, $A\subseteq X$ be a closed invariant subspace and $F$ be an abelian group of coefficients. Suppose that for any $x\in X-(A\cup X^{S^1})$ we have $H^i(BG_x;F)=0$ for all $i>0$. Then for any $i\in\N$ the map $f:X_{S^1}\to X/S^1$ induces an isomorphism
    \[
    f^*:H^i(X/S^1,A/S^1\cup X^{S^1};F)\to H^i_{S^1}(X,A\cup X^{S^1};F)
    \]
\end{prop}
If we fix $\F_p$ as coefficients and $A=X^{\Z/p}$ we get:
\begin{corollario}\label{cor:relazione tra coomologia quoziente omotopico e quoziente stretto}
    Let $X$ be an $S^1$-space. Then for any $i\in\N$ the map $f:X_{S^1}\to X/S^1$ induces an isomorphism
     \[
    f^*:H^i(X/S^1,X^{\Z/p}/S^1;\F_p)\to H^i_{S^1}(X,X^{\Z/p};\F_p)
    \]
   
\end{corollario}
\begin{oss}\label{oss: quoziente stretto e omotopico coincidono mod p}
    If there are not $\Z/p$-fixed points, then we have
    \[
    H^i(X/S^1;\F_p)\cong H^i_{S^1}(X;\F_p)
    \]
\end{oss}
We can use this Corollary to get a \emph{Mayer-Vietoris} sequence as follows: the map $f:X_{S^1}\to X/S^1$ sends $(X^{\Z/p})_{S^1}$ to $X^{\Z/p}/S^1$ therefore we get a map of long exact sequences 
\[
\begin{tikzcd}[column sep=small]
    &\cdots\arrow[r] & H^{i}(\frac{X}{S^1},\frac{X^{\Z/p}}{S^1})\arrow[r]\arrow[d,red]& H^i(\frac{X}{S^1})\arrow[r]\arrow[d] & H^i(\frac{X^{\Z/p}}{S^1})\arrow[r]\arrow[d]&H^{i+1}(\frac{X}{S^1},\frac{X^{\Z/p}}{S^1})\arrow[r]\arrow[d,red]&\cdots\\
     &\cdots\arrow[r] & H^{i}_{S^1}(X,X^{\Z/p})\arrow[r]& H^i_{S^1}(X)\arrow[r] & H^i_{S^1}(X^{\Z/p})\arrow[r]& H^{i}_{S^1}(X,X^{\Z/p})\arrow[r] &\cdots
\end{tikzcd}
\]
where the red vertical arrows are isomorphisms by the previous Corollary. Therefore we get:
\begin{prop}[Mayer-Vietoris]\label{prop: Mayer-Vietoris}
    Fix $p$ a prime and use $\F_p$ as field of coefficients for cohomology. Then we have a long exact sequence
    \[
    \begin{tikzcd}[column sep=small]
         &\cdots\arrow[r] & H^{i}(\frac{X}{S^1})\arrow[r]&  H^i(\frac{X^{\Z/p}}{S^1})\oplus H^i_{S^1}(X)\arrow[r]&H^{i}_{S^1}(X^{\Z/p})\arrow[r] & H^{i+1}(\frac{X}{S^1})\arrow[r]&\cdots
    \end{tikzcd}
    \]
\end{prop}
We are going to use Proposition \ref{prop: Mayer-Vietoris} for $X=C_n(\C)$ in order to do some computations of $H_*(C_n(\C)/S^1;\F_p)$. Before going on let us observe that the space of fixed points $C_n(\C)^{\Z/p}$ is well understood:
	\begin{lem}\label{lem: Z/p punti fissi}
		Let $n=pq$ or $n=pq+1$. Then the fixed points $C_n(\C)^{\Z/p}$ are homeomorphic to $C_q(\C^*)$.
	\end{lem}
	\begin{proof}
		Let us prove the statement when $n=pq$, the other case in similar. Let us denote by $\zeta\coloneqq e^{i2\pi/p}$ the generator of $\Z/p$. Consider the quotient space
		\[
		H\coloneqq\{z\in\C^*\mid arg(z)\in[0,2\pi/p]\}/\sim
		\]
		where $\sim$ identifies a point $z\in\{z\in\C^*\mid arg(z)=0\}$ with $\zeta z$. So $H$ is homeomorphic to $\C^*$. Now observe that any configuration in $C_{n}(\C)^{\Z/p}$ is of the form $\{z_1,\zeta z_1,\dots\zeta^{p-1}z_1,\dots, z_q,\zeta z_q,\dots\zeta^{p-1}z_q\}$, where $z_1,\dots,z_q$ are distinct points in $\{z\in\C^*\mid arg(z)\in[0,2\pi/p)\}$. The associ ation $\{z_1,\zeta z_1,\dots\zeta^{p-1}z_1,\dots, z_q,\zeta z_q,\dots\zeta^{p-1}z_q\}\mapsto\{z_1,\dots,z_q\}$ defines a continuous map 
		\[
		f: C_{n}(\C)^{\Z/p}\to C_q(H)
		\]
		Conversely, if we have a configuration $\{z_1,\dots,z_q\}\in  C_q(H)$, we can produce a configuration of $C_n(\C)^{\Z/p}$ by taking the $\Z/p$-orbits of every point. More precisely, the association $\{z_1,\dots,z_q\}\mapsto\{z_1,\zeta z_1,\dots\zeta^{p-1}z_1,\dots, z_q,\zeta z_q,\dots\zeta^{p-1}z_q\}$ defines a continuous function $C_q(H)\to C_n(\C)^{\Z/p}$, which is the inverse of $f$. See Figure \ref{fig:homeomorfismo dello spazio dei punti fissi} for a pictorial description of $f$.
	\end{proof}
	\begin{figure}
		\centering

		\tikzset{every picture/.style={line width=0.75pt}} 
		
		\begin{tikzpicture}[x=0.75pt,y=0.75pt,yscale=-1,xscale=1]
			
			\draw    (140.15,163.71) -- (234.79,163.71) ;
			\draw    (140.15,163.71) -- (169.4,73.71) ;
			\draw  [fill={rgb, 255:red, 0; green, 0; blue, 0 }  ,fill opacity=1 ] (174.27,144) .. controls (174.27,143.25) and (174.88,142.64) .. (175.63,142.64) .. controls (176.39,142.64) and (177,143.25) .. (177,144) .. controls (177,144.75) and (176.39,145.36) .. (175.63,145.36) .. controls (174.88,145.36) and (174.27,144.75) .. (174.27,144) -- cycle ;
			\draw  [fill={rgb, 255:red, 0; green, 0; blue, 0 }  ,fill opacity=1 ] (194.01,148.23) .. controls (194.01,147.48) and (194.62,146.87) .. (195.37,146.87) .. controls (196.12,146.87) and (196.73,147.48) .. (196.73,148.23) .. controls (196.73,148.98) and (196.12,149.59) .. (195.37,149.59) .. controls (194.62,149.59) and (194.01,148.98) .. (194.01,148.23) -- cycle ;
			\draw  [fill={rgb, 255:red, 0; green, 0; blue, 0 }  ,fill opacity=1 ] (177.09,111.58) .. controls (177.09,110.83) and (177.7,110.22) .. (178.45,110.22) .. controls (179.2,110.22) and (179.81,110.83) .. (179.81,111.58) .. controls (179.81,112.33) and (179.2,112.94) .. (178.45,112.94) .. controls (177.7,112.94) and (177.09,112.33) .. (177.09,111.58) -- cycle ;
			\draw    (139.99,163.6) -- (169.03,253.68) ;
			\draw  [fill={rgb, 255:red, 0; green, 0; blue, 0 }  ,fill opacity=1 ] (169.22,190.03) .. controls (169.93,189.8) and (170.7,190.19) .. (170.93,190.91) .. controls (171.16,191.62) and (170.77,192.39) .. (170.05,192.62) .. controls (169.34,192.85) and (168.57,192.46) .. (168.34,191.74) .. controls (168.11,191.03) and (168.5,190.26) .. (169.22,190.03) -- cycle ;
			\draw  [fill={rgb, 255:red, 0; green, 0; blue, 0 }  ,fill opacity=1 ] (171.25,210.11) .. controls (171.97,209.88) and (172.73,210.27) .. (172.96,210.99) .. controls (173.19,211.7) and (172.8,212.47) .. (172.09,212.7) .. controls (171.37,212.93) and (170.6,212.54) .. (170.37,211.82) .. controls (170.14,211.11) and (170.53,210.34) .. (171.25,210.11) -- cycle ;
			\draw  [fill={rgb, 255:red, 0; green, 0; blue, 0 }  ,fill opacity=1 ] (200.94,182.76) .. controls (201.66,182.53) and (202.42,182.92) .. (202.66,183.64) .. controls (202.89,184.36) and (202.49,185.12) .. (201.78,185.35) .. controls (201.06,185.59) and (200.29,185.19) .. (200.06,184.48) .. controls (199.83,183.76) and (200.23,182.99) .. (200.94,182.76) -- cycle ;
			\draw    (140.11,163.96) -- (63,218.83) ;
			\draw  [fill={rgb, 255:red, 0; green, 0; blue, 0 }  ,fill opacity=1 ] (123.74,199.8) .. controls (124.18,200.41) and (124.03,201.26) .. (123.42,201.7) .. controls (122.81,202.13) and (121.96,201.99) .. (121.52,201.38) .. controls (121.09,200.77) and (121.23,199.92) .. (121.84,199.48) .. controls (122.45,199.04) and (123.3,199.19) .. (123.74,199.8) -- cycle ;
			\draw  [fill={rgb, 255:red, 0; green, 0; blue, 0 }  ,fill opacity=1 ] (105.21,207.8) .. controls (105.65,208.41) and (105.5,209.26) .. (104.89,209.7) .. controls (104.28,210.13) and (103.43,209.99) .. (102.99,209.38) .. controls (102.55,208.76) and (102.7,207.91) .. (103.31,207.48) .. controls (103.92,207.04) and (104.77,207.19) .. (105.21,207.8) -- cycle ;
			\draw  [fill={rgb, 255:red, 0; green, 0; blue, 0 }  ,fill opacity=1 ] (140.24,227.85) .. controls (140.68,228.46) and (140.54,229.31) .. (139.92,229.75) .. controls (139.31,230.18) and (138.46,230.04) .. (138.02,229.43) .. controls (137.59,228.82) and (137.73,227.97) .. (138.34,227.53) .. controls (138.96,227.09) and (139.81,227.24) .. (140.24,227.85) -- cycle ;
			\draw    (139.99,163.6) -- (63.42,107.98) ;
			\draw  [fill={rgb, 255:red, 0; green, 0; blue, 0 }  ,fill opacity=1 ] (131.62,125.59) .. controls (130.91,125.36) and (130.51,124.59) .. (130.75,123.88) .. controls (130.98,123.16) and (131.75,122.77) .. (132.46,123) .. controls (133.18,123.24) and (133.57,124) .. (133.34,124.72) .. controls (133.1,125.43) and (132.34,125.83) .. (131.62,125.59) -- cycle ;
			\draw  [fill={rgb, 255:red, 0; green, 0; blue, 0 }  ,fill opacity=1 ] (141.74,108.13) .. controls (141.03,107.9) and (140.63,107.13) .. (140.87,106.42) .. controls (141.1,105.7) and (141.87,105.31) .. (142.58,105.54) .. controls (143.3,105.77) and (143.69,106.54) .. (143.46,107.26) .. controls (143.22,107.97) and (142.46,108.36) .. (141.74,108.13) -- cycle ;
			\draw  [fill={rgb, 255:red, 0; green, 0; blue, 0 }  ,fill opacity=1 ] (101.66,112.89) .. controls (100.94,112.66) and (100.55,111.89) .. (100.78,111.18) .. controls (101.01,110.46) and (101.78,110.07) .. (102.5,110.3) .. controls (103.21,110.54) and (103.6,111.3) .. (103.37,112.02) .. controls (103.14,112.73) and (102.37,113.13) .. (101.66,112.89) -- cycle ;
			\draw  [fill={rgb, 255:red, 0; green, 0; blue, 0 }  ,fill opacity=1 ] (100.82,159.88) .. controls (100.37,160.48) and (99.52,160.61) .. (98.91,160.17) .. controls (98.31,159.73) and (98.17,158.88) .. (98.62,158.27) .. controls (99.06,157.66) and (99.91,157.53) .. (100.52,157.97) .. controls (101.13,158.42) and (101.26,159.27) .. (100.82,159.88) -- cycle ;
			\draw  [fill={rgb, 255:red, 0; green, 0; blue, 0 }  ,fill opacity=1 ] (87.38,144.81) .. controls (86.94,145.42) and (86.09,145.55) .. (85.48,145.11) .. controls (84.87,144.67) and (84.74,143.81) .. (85.18,143.21) .. controls (85.63,142.6) and (86.48,142.47) .. (87.09,142.91) .. controls (87.69,143.35) and (87.83,144.21) .. (87.38,144.81) -- cycle ;
			\draw  [fill={rgb, 255:red, 0; green, 0; blue, 0 }  ,fill opacity=1 ] (79.41,184.38) .. controls (78.96,184.99) and (78.11,185.12) .. (77.5,184.68) .. controls (76.9,184.24) and (76.76,183.38) .. (77.21,182.78) .. controls (77.65,182.17) and (78.5,182.04) .. (79.11,182.48) .. controls (79.72,182.93) and (79.85,183.78) .. (79.41,184.38) -- cycle ;
			\draw    (264,166) -- (330.73,166) ;
			\draw [shift={(332.73,166)}, rotate = 180] [color={rgb, 255:red, 0; green, 0; blue, 0 }  ][line width=0.75]    (10.93,-3.29) .. controls (6.95,-1.4) and (3.31,-0.3) .. (0,0) .. controls (3.31,0.3) and (6.95,1.4) .. (10.93,3.29)   ;
			\draw  [dash pattern={on 0.84pt off 2.51pt}] (39.73,102.88) .. controls (39.73,80.68) and (57.73,62.68) .. (79.93,62.68) -- (210.53,62.68) .. controls (232.73,62.68) and (250.73,80.68) .. (250.73,102.88) -- (250.73,223.48) .. controls (250.73,245.68) and (232.73,263.68) .. (210.53,263.68) -- (79.93,263.68) .. controls (57.73,263.68) and (39.73,245.68) .. (39.73,223.48) -- cycle ;
			\draw    (360.46,206.41) -- (494.73,206.41) ;
			\draw    (360.46,206.41) -- (401.95,78.71) ;
			\draw  [fill={rgb, 255:red, 0; green, 0; blue, 0 }  ,fill opacity=1 ] (408.87,178.44) .. controls (408.87,177.37) and (409.73,176.51) .. (410.8,176.51) .. controls (411.87,176.51) and (412.73,177.37) .. (412.73,178.44) .. controls (412.73,179.5) and (411.87,180.37) .. (410.8,180.37) .. controls (409.73,180.37) and (408.87,179.5) .. (408.87,178.44) -- cycle ;
			\draw  [fill={rgb, 255:red, 0; green, 0; blue, 0 }  ,fill opacity=1 ] (436.87,184.44) .. controls (436.87,183.37) and (437.73,182.51) .. (438.8,182.51) .. controls (439.87,182.51) and (440.73,183.37) .. (440.73,184.44) .. controls (440.73,185.5) and (439.87,186.37) .. (438.8,186.37) .. controls (437.73,186.37) and (436.87,185.5) .. (436.87,184.44) -- cycle ;
			\draw  [fill={rgb, 255:red, 0; green, 0; blue, 0 }  ,fill opacity=1 ] (412.87,132.44) .. controls (412.87,131.37) and (413.73,130.51) .. (414.8,130.51) .. controls (415.87,130.51) and (416.73,131.37) .. (416.73,132.44) .. controls (416.73,133.5) and (415.87,134.37) .. (414.8,134.37) .. controls (413.73,134.37) and (412.87,133.5) .. (412.87,132.44) -- cycle ;
			\draw  [fill={rgb, 255:red, 255; green, 255; blue, 255 }  ,fill opacity=1 ] (357.37,206.41) .. controls (357.37,204.7) and (358.76,203.32) .. (360.46,203.32) .. controls (362.17,203.32) and (363.55,204.7) .. (363.55,206.41) .. controls (363.55,208.11) and (362.17,209.49) .. (360.46,209.49) .. controls (358.76,209.49) and (357.37,208.11) .. (357.37,206.41) -- cycle ;
			\draw    (494.73,206.41) -- (429.6,206.41) ;
			\draw [shift={(427.6,206.41)}, rotate = 360] [color={rgb, 255:red, 0; green, 0; blue, 0 }  ][line width=0.75]    (10.93,-3.29) .. controls (6.95,-1.4) and (3.31,-0.3) .. (0,0) .. controls (3.31,0.3) and (6.95,1.4) .. (10.93,3.29)   ;
			\draw    (401.95,78.71) -- (381.83,140.65) ;
			\draw [shift={(381.21,142.56)}, rotate = 288] [color={rgb, 255:red, 0; green, 0; blue, 0 }  ][line width=0.75]    (10.93,-3.29) .. controls (6.95,-1.4) and (3.31,-0.3) .. (0,0) .. controls (3.31,0.3) and (6.95,1.4) .. (10.93,3.29)   ;
			
			\draw (292,143) node [anchor=north west][inner sep=0.75pt]   [align=left] {$\displaystyle f$};

		\end{tikzpicture}
		
		\caption{This picture shows how the homeomorphism $f:C_{15}(\C)^{\Z/5}\to C_3(\C^*)$ works. }
		\label{fig:homeomorfismo dello spazio dei punti fissi}
	\end{figure}
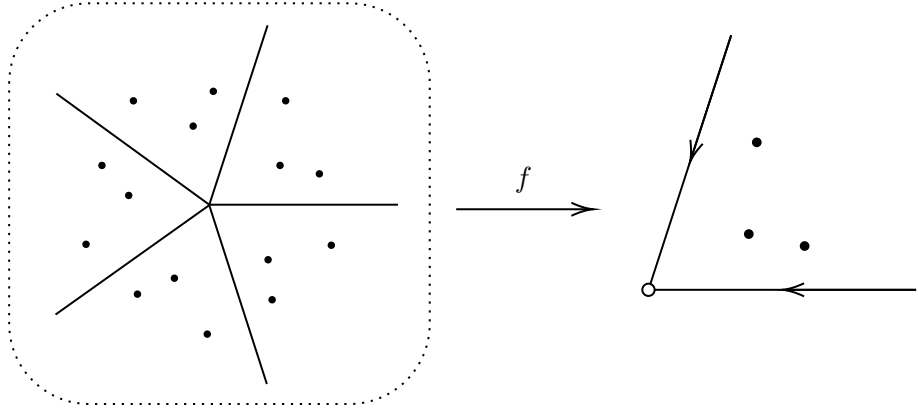
\subsection{Recollection on $H_*^{S^1}(C_n(\C);\F_p)$ }
In this paragraph we recall the computation of $H_*^{S^1}(C_n(\C);\F_p)$ for any $n\in\N$ and any prime $p$ (see \cite{Rossi1} for details).
Combining this with the Mayer-Vietoris sequence of Proposition \ref{prop: Mayer-Vietoris} we will be able to compute $H_*(C_n(\C)/S^1;\F_p)\cong H_*(\mathcal{M}_{0,n+1}/\Sigma_n;\F_p)$ is some cases. 

$H_*^{S^1}(C_n(\C);\F_p)$ can be computed using the Serre spectral sequence associated to the fibration
\begin{equation}\label{fib: fibrazione quoziente omotopico spazio configurazioni sul piano}
	C_n(\C)\hookrightarrow C_n(\C)_{S^1}\to BS^1
\end{equation}
The homology of the fiber is known, thanks to the work of F. Cohen (see \cite{Cohen}):
	\begin{thm}[\cite{Cohen}]\label{thm: omologia spazio di configurazione }
		Consider the disjoint union $C(\C)\coloneqq \bigsqcup_{n\in\N}C_n(\C)$ and let $p$ be any prime. Then $H_*(C(\C);\F_p)$ has the following form:
		\begin{description}
			\item[$p=2$:] $H_*(C(\C);\F_2)$ is the free graded commutative algebra on classes $\{Q^i(a)\}_{i\in\N}$, where $Q^i(a)$ is a class of degree $2^{i}-1$ in $H_*(C_{2^{i}}(\C);\F_2)$.
			\item[$p\neq 2$:] $H_*(C(\C);\F_p)$ is the free graded commutative algebra on classes 
   \[
   \{a,[a,a], Q^i[a,a], \beta Q^i[a,a]\}_{i\geq 1}
   \]
   where $a$ is the generator of $H_0(C_1(\C);\F_p)$, $[a,a]$ is the generator of $H_1(C_2(\C);\F_p)$  and $Q^i[a,a]$ (resp. $\beta Q^i[a,a]$) is a class of degree $2p^{i}-1$ (resp. $2p^i-2$) in $H_*(C_{2p^{i}}(\C);\F_p)$.
		\end{description}
	\end{thm}
The main results of \cite{Rossi1} are the following:
\begin{thm}\label{thm: calcolo dell'omologia equivariante degli spazi di configurazione quando p divide n}
		Let $p$ be a prime, $n\in\N$ such that $n=0,1$ mod $p$. Then the (co)homological spectral sequence (with $\F_p$-coefficients) associated to fibration \ref{fib: fibrazione quoziente omotopico spazio configurazioni sul piano}
  degenerates at the second page. In particular
		\[
		H_*^{S^1}(C_n(\C);\F_p)\cong H_*(C_n(\C);\F_p)\otimes H_*(BS^1;\F_p)
		\]
	\end{thm}
\begin{thm}\label{thm:calcolo omologia equivariante quando p non divide n o n-1}
		Let $p$ be an odd prime, $n\in\N$ such that $n\neq 0,1$ mod $p$. The homological spectral sequence (with $\F_p$-coefficients) associated to fibration  \ref{fib: fibrazione quoziente omotopico spazio configurazioni sul piano}
  degenerates at the third page. In particular
		\[
		H_*^{S^1}(C_n(\C);\F_p)\cong coker(\Delta)
		\]
		where $\Delta:H_*(C_n(\C);\F_p)\to H_{*+1}(C_n(\C);\F_p)$ is the $BV$-operator.
	\end{thm}
\begin{oss}
	 Consider a class $a^k[a,a]^{l}x\in H_*(C_n(\C);\F_p)$, where $k\in\N$, $l=0,1$  and $x$ is a monomial which contains only the letters $\{Q^i[a,a],\beta Q^i[a,a]\}_{i\geq 1}$. It turns out that the operator $\Delta$ acts as follows:
	\begin{equation}\label{eq: formule di delta sulla base dei monomi}
		\Delta(a^kx)=k(k-1)a^{k-2}[a,a]x \qquad \Delta(a^k[a,a]x)=0
	\end{equation}
	Therefore a basis of $coker(\Delta)$ is given by (the image of) classes in $H_*(C_n(\C);\F_p)$ which do not contain the bracket $[a,a]$
\end{oss}	
\subsection{Computations in some special cases}\label{sec: computations}

From now on we fix $\F_p$ as field of coefficients for homology, where $p$ is any prime number. We are interested in the quotient $\M_{0,n+1}/\Sigma_n$, which is homotopy equivalent to $C_n(\C)/S^1$. Corollary \ref{cor:relazione tra coomologia quoziente omotopico e quoziente stretto} allow us to do the following computation:

\begin{thm}\label{thm:omologia quoziente stretto e omotopico sono uguali}
	If $n\neq 0,1$ mod $p$ the map $f:C_n(\C)_{S^1}\to C_n(\C)/S^1$ induces an isomorphism
	\[
	f^*:H^i(C_n(\C)/S^1;\F_p)\to H^i_{S^1}(C_n(\C);\F_p)
	\]
\end{thm}
\begin{proof}
	Just observe that if $n\neq 0,1$ mod $p$ there are no $\Z/p$-fixed points.
\end{proof}
Now suppose that $p$ divides $n$ (or $n-1$). We would like to apply the Mayer-Vietoris sequence of Proposition \ref{prop: Mayer-Vietoris} 
\begin{equation}\label{eq: Mayer-Vietoris per gli spazi di configurazione}
	\begin{tikzcd}[column sep=small]
		&\cdots\arrow[r] & H^{i}(\frac{C_n(\C)}{S^1})\arrow[r]&  H^i(\frac{C_n(\C)^{\Z/p}}{S^1})\oplus H^i_{S^1}(C_n(\C))\arrow[r]&H^{i}_{S^1}(C_n(\C)^{\Z/p})\arrow[r] &\cdots
	\end{tikzcd}
\end{equation}
to compute of $H^i(C_n(\C)/S^1;\F_p)$, but in this generality the problem is hard. Indeed we do not know $H^*(C_n(\C)^{\Z/p}/S^1;\F_p)$, which could be as difficult to compute as $H^*(C_n(\C)/S^1;\F_p)$. The only advantage is that $C_n(\C)^{\Z/p}/S^1$ is a much smaller space with respect to $C_n(\C)/S^1$. Moreover, even if we would be able to do such computation it remains to understand the maps that fit into the Mayer-Vietoris sequence, which can be an even harder problem. 
\begin{oss}\label{oss: sequenza spettrale dei punti fissi degenera}
	In the Mayer-Vietoris sequence above the terms $H^i_{S^1}(C_n(\C)^{\Z/p};\F_p)$ can be computed easily: consider the map of fibrations
	\[
	\begin{tikzcd}
		& C_n(\C)^{\Z/p}\arrow[r] \arrow[d,hook] & C_n(\C)^{\Z/p}\arrow[d,hook]\\
		& \left(C_n(\C)^{\Z/p}\right)_{\Z/p}\arrow[r] \arrow[d] & \left(C_n(\C)^{\Z/p}\right)_{S^1}\arrow[d]\\
		& B\Z/p\arrow[r] &BS^1
	\end{tikzcd}
	\]
	and observe that the mod $p$ Serre spectral sequence of the left fibration degenerates at the second page. The map of cohomological spectral sequences induced by the map of fibrations above is surjective at the $E_2$ page so the Serre spectral sequence of the right fibration degenerates at the second page as well. Therefore $H^*_{S^1}(C_n(\C)^{\Z/p};\F_p)=H^*(BS^1;\F_p)\otimes H^*(C_n(\C)^{\Z/p};\F_p)$ as $\F_p$-vector space. Finally we note that $H^*_{S^1}(C_n(\C)^{\Z/p};\F_p)$ is known: $C_n(\C)^{\Z/p}$ is homeomorphic to the space of unordered configurations of points in $\C^*$ ( Lemma \ref{lem: Z/p punti fissi}) and the homology of this space is known thanks to the work of F. Cohen \cite{Cohen} (see also Proposition \ref{prop:omologia nelle configurazioni nel piano bucato} for a computation of these homology groups).
\end{oss}
We end this paragraph with an acyclicity result:
\begin{thm}\label{thm: quoziente stretto di p punti non ha omologia modulo p}
	Let $p$ be a prime. Then 
	\[
	H^*(C_p(\C)/S^1;\F_p)\cong H^*(C_{p+1}(\C)/S^1;\F_p)\cong \begin{cases}
		\F_p \text{ if } *=0\\
		0 \text{ otherwise}
	\end{cases} 
	\]
\end{thm}
In order to prove this Theorem we need a preliminary Lemma:
\begin{lem}\label{lem:mono in coomologia equivariante}
	Let  $n\in\N$ and $p$ be a prime. Suppose $n=0,1$ mod $p$. If $i:C_n(\C)^{\Z/p}\hookrightarrow C_n(\C)$ is the inclusion then \[
	i^*:H^k_{S^1}(C_n(\C);\F_p)\to H^k_{S^1}(C_n(\C)^{\Z/p};\F_p)
	\]
	is a monomorphism.
\end{lem}
\begin{proof}
	Consider the commutative diagram
	\[
	\begin{tikzcd}
		& H^k_{S^1}(C_n(\C))\arrow[r,"i^*"]\arrow[d] & H^k_{S^1}(C_n(\C)^{\Z/p})\arrow[d]\\
		&\left( S^{-1}H^*_{S^1}(C_n(\C))\right)^{k} \arrow[r,"i^*"]&\left( S^{-1}H^*_{S^1}(C_n(\C)^{\Z/p})\right)^{k} 
	\end{tikzcd}
	\]
	where $S$ is the multiplicatively closed subset $\{1,c,c^2,c^3,\dots\}\subseteq H^*(BS^1;\F_p)=\F_p[c]$, with $c$ be a variable of degree two. By the Localization Theorem (see \cite[Theorem 4.2 p.198]{TomDieck}) the bottom horizontal arrow is an isomorphism. The cohomological version of Theorem \ref{thm: calcolo dell'omologia equivariante degli spazi di configurazione quando p divide n} tells us that the mod $p$ spectral sequence associated to the fibration $C_n(\C)\hookrightarrow C_n(\C)_{S^1}\to BS^1$ degenerates at the $E_2$ page, so $H^*_{S^1}(C_n(\C);\F_p)$ is a free $H^*(BS^1)$-module. Therefore the left vertical arrow is a monomorphism, and this proves the statement.
\end{proof}
\begin{proof}(of Theorem \ref{thm: quoziente stretto di p punti non ha omologia modulo p})
	We do the case of $C_p(\C)$, the other one is completely analogous. First of all note that $C_p(\C)^{\Z/p}\cong C_1(\C^*)=\C^*$. Therefore
	\[
	C_p(\C)^{\Z/p}/S^1\cong\C^*/S^1
	\]
	and we conclude that $C_p(\C)^{\Z/p}/S^1$ is contractible. If $i\geq 1$ the Mayer-Vietoris sequence \ref{eq: Mayer-Vietoris per gli spazi di configurazione} becomes
	\[
	\begin{tikzcd}[column sep=small]
		&\cdots\arrow[r] & H^{i}(\frac{C_p(\C)}{S^1})\arrow[r]& H^i_{S^1}(C_p(\C))\arrow[r]&H^{i}_{S^1}(C_p(\C)^{\Z/p})\arrow[r] &\cdots
	\end{tikzcd}
	\]
	By Lemma \ref{lem:mono in coomologia equivariante} the map $
	H^k_{S^1}(C_p(\C);\F_p)\to H^k_{S^1}(C_p(\C)^{\Z/p};\F_p)$
	is a monomorphism, therefore the Mayer-Vietoris sequence splits (for $i\geq 1$) as:
	\[
	\begin{tikzcd}[column sep=small]
		& 0\arrow[r] & H^i_{S^1}(C_p(\C))\arrow[r]&H^{i}_{S^1}(C_p(\C)^{\Z/p})\arrow[r] & H^{i+1}(\frac{C_p(\C)}{S^1})\arrow[r] & 0
	\end{tikzcd}
	\]
	To conclude the proof we show that $ H^i_{S^1}(C_p(\C);\F_p)$ and $H^{i}_{S^1}(C_p(\C)^{\Z/p};\F_p)$ are vector spaces of the same dimension: by Theorem \ref{thm: calcolo dell'omologia equivariante degli spazi di configurazione quando p divide n} we have $H^*_{S^1}(C_p(\C);\F_p)\cong H^*(C_p(\C);\F_p)\otimes H^*(BS^1;\F_p)$ as $H^*(BS^1;\F_p)$-module. Similarly $H^*_{S^1}(C_p(\C)^{\Z/p};\F_p)\cong H^*(C_p(\C)^{\Z/p};\F_p)\otimes H^*(BS^1;\F_p)$ by Remark \ref{oss: sequenza spettrale dei punti fissi degenera}. Finally observe that $H_*(C_p(\C);\F_p)$ is generated by a class of degree zero and one of degree one (i.e. $a^p$ and $[a,a]a^{p-2}$) therefore $H^i(C_p(\C);\F_p)$ and $H^i(C_p(\C)^{\Z/p};\F_p)\cong H^i(S^1;\F_p)$ have the same dimension for each $i\in\N$. As a consequence we get that $ H^i_{S^1}(C_p(\C);\F_p)$ and $H^{i}_{S^1}(C_p(\C)^{\Z/p};\F_p)$ have the same dimension, as claimed.
\end{proof}

\subsection{Examples}\label{sec:examples}
In what follows we use the results of the previous paragraph to do some explicit computations of $H_*(C_n(\C)/S^1;\F_p)$. In particular, we will see that the quotients $C_{n}(\C)/S^1$ are contractible until $n=6$, which is the first non contractible space.
\paragraph{n=1:} $C_1(\C)/S^1$ is homeomorphic to a half line, so it is contractible.
\paragraph{n=2:} by the combinatorial model explained in Section \ref{sec: a combinatorial model using cacti} $C_2(\C)/S^1$ is homotopy equivalent to a CW-complex with only one cell of dimension zero, so it is contractible.
\paragraph{n=3:} by the combinatorial model explained in Section \ref{sec: a combinatorial model using cacti} $C_3(\C)/S^1$ is homotopy equivalent to a CW-complex of dimension one (see also Figure \ref{fig:esempio con tre lobi}), so there are no homology classes of degrees strictly greater than 1. But $C_3(\C)/S^1$ is simply connected, so $H_1(C_3(\C)/S^1;\Z)=0$ and this implies that $C_3(\C)/S^1$ is contractible. 
\paragraph{n=4:} We prove that $C_4(\C)/S^1$ is contractible by showing that $H_i(C_4(\C)/S^1;\Z)=0$ for any $i\geq 1$. This is enough to prove the claim since $C_4(\C)/S^1$ is homotopic to a CW complex of dimension two, and it is simply connected. By the cellular model we know that there are no homology classes of degrees strictly greater than 2.  Rationally $C_4(\C)/S^1\simeq\M_{0,5}/\Sigma_4$ has the same homology of a point, therefore its Euler characteristic must be one:
\[
1=b_0-b_1+b_2
\]
where $b_i$ is $i$-th Betti number. But $C_4(\C)/S^1$ is simply connected, so $b_1=0$ and combining this fact with the equation above we get that $b_2=0$. Therefore $C_4(\C)/S^1$ is contractible.

\paragraph{n=5:} We will prove that $C_5(\C)/S^1$ is contractible by showing that for any prime $p$ and any $i\geq 1$ $H_i(C_5(\C)/S^1;\F_p)=0$. The cellular model tells us that $C_5(\C)/S^1$ has no homology classes of degrees strictly greater than three. By Theorem \ref{thm: bound sulla torsione dei quozienti} the order of any class in $H_*(C_5(\C);\Z)$ divides $5!$. So the only interesting cases are when we take $\F_2$, $\F_3$ and $\F_5$ as coefficients for homology.
\begin{itemize}
	\item $\F_2$-coefficients: by Proposition \ref{prop: Mayer-Vietoris} we have a Mayer-Vietoris sequence
	\[
	\begin{tikzcd}[column sep=small]
		&\cdots\arrow[r]& H^{i}(\frac{C_5(\C)}{S^1})\arrow[r] &  H^i(\frac{C_5(\C)^{\Z/2}}{S^1})\oplus H^i_{S^1}(C_5(\C))\arrow[r]&H^{i}_{S^1}(C_5(\C)^{\Z/2})\arrow[r] & \cdots
	\end{tikzcd}
	\]
	Now observe that $C_5(\C)^{\Z/2}$ is homeomorphic $C_2(\C^*)$, whose homology with $\F_2$-coefficients is generated by the classes listed below:
	\[
	\begin{tabular}{cc}
		\toprule
		Homology class  & Degree\\
		\midrule
		$a^2b$ & $0$ \\
		$[b,a]a$ & $1$ \\
		$b\cdot Qa$ & $1$ \\
		$[[b,a],a]$ & $2$ \\
		
		\bottomrule
	\end{tabular}
	\]
	For the notation and more details about $H_*(C_n(\C^*);\F_p)$ see Proposition \ref{prop:omologia nelle configurazioni nel piano bucato}. Moreover, $C_2(\C^*)/S^1$ is homotopy equivalent to a circle (see Figure \ref{fig:equivalenza omotopica con il cerchio} and its caption for an explanation). The second pages of the Serre spectral sequences that compute $H^*_{S^1}(C_5(\C);\F_2)$ and $H^*_{S^1}(C_5(\C)^{\Z/2};\F_2)$ are displayed below:
	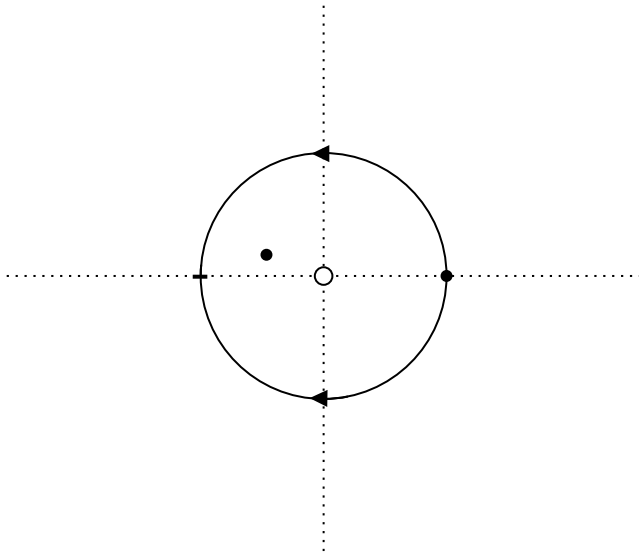
\begin{figure}
		\centering

		\tikzset{every picture/.style={line width=0.75pt}} 
		
		\begin{tikzpicture}[x=0.75pt,y=0.75pt,yscale=-1,xscale=1]
			
			\draw  [draw opacity=0] (309.6,150.63) .. controls (311.59,150.43) and (313.62,150.33) .. (315.67,150.33) .. controls (349.72,150.33) and (377.33,177.94) .. (377.33,212) .. controls (377.33,212.13) and (377.33,212.25) .. (377.33,212.38) -- (315.67,212) -- cycle ; \draw    (312.61,150.41) .. controls (313.62,150.36) and (314.64,150.33) .. (315.67,150.33) .. controls (349.72,150.33) and (377.33,177.94) .. (377.33,212) .. controls (377.33,212.13) and (377.33,212.25) .. (377.33,212.38) ;  \draw [shift={(309.6,150.63)}, rotate = 359.94] [fill={rgb, 255:red, 0; green, 0; blue, 0 }  ][line width=0.08]  [draw opacity=0] (8.93,-4.29) -- (0,0) -- (8.93,4.29) -- cycle    ;
			\draw  [fill={rgb, 255:red, 0; green, 0; blue, 0 }  ,fill opacity=1 ] (374.83,212) .. controls (374.83,210.62) and (375.95,209.5) .. (377.33,209.5) .. controls (378.71,209.5) and (379.83,210.62) .. (379.83,212) .. controls (379.83,213.38) and (378.71,214.5) .. (377.33,214.5) .. controls (375.95,214.5) and (374.83,213.38) .. (374.83,212) -- cycle ;
			\draw  [dash pattern={on 0.84pt off 2.51pt}]  (156.67,212) -- (474.67,212) ;
			\draw  [dash pattern={on 0.84pt off 2.51pt}]  (315.67,350) -- (315.67,74) ;
			\draw  [draw opacity=0] (377.33,212.34) .. controls (377.15,246.24) and (349.61,273.67) .. (315.67,273.67) .. controls (313.29,273.67) and (310.94,273.53) .. (308.63,273.27) -- (315.67,212) -- cycle ; \draw    (377.33,212.34) .. controls (377.15,246.24) and (349.61,273.67) .. (315.67,273.67) .. controls (314.3,273.67) and (312.94,273.62) .. (311.6,273.53) ; \draw [shift={(308.63,273.27)}, rotate = 0.97] [fill={rgb, 255:red, 0; green, 0; blue, 0 }  ][line width=0.08]  [draw opacity=0] (8.93,-4.29) -- (0,0) -- (8.93,4.29) -- cycle    ; 
			\draw  [draw opacity=0] (254,212.37) .. controls (254,212.24) and (254,212.12) .. (254,212) .. controls (254,177.94) and (281.61,150.33) .. (315.67,150.33) -- (315.67,212) -- cycle ; \draw    (254,212.37) .. controls (254,212.24) and (254,212.12) .. (254,212) .. controls (254,177.94) and (281.61,150.33) .. (315.67,150.33) ;  
			\draw  [draw opacity=0] (327.9,272.49) .. controls (323.95,273.28) and (319.85,273.7) .. (315.67,273.7) .. controls (281.61,273.7) and (254,246.09) .. (254,212.03) .. controls (254,210.13) and (254.09,208.25) .. (254.25,206.39) -- (315.67,212.03) -- cycle ; \draw    (327.9,272.49) .. controls (323.95,273.28) and (319.85,273.7) .. (315.67,273.7) .. controls (281.61,273.7) and (254,246.09) .. (254,212.03) .. controls (254,210.13) and (254.09,208.25) .. (254.25,206.39) ;  
			\draw  [fill={rgb, 255:red, 0; green, 0; blue, 0 }  ,fill opacity=1 ] (284.5,201.37) .. controls (284.5,199.99) and (285.62,198.87) .. (287,198.87) .. controls (288.38,198.87) and (289.5,199.99) .. (289.5,201.37) .. controls (289.5,202.75) and (288.38,203.87) .. (287,203.87) .. controls (285.62,203.87) and (284.5,202.75) .. (284.5,201.37) -- cycle ;
			\draw  [fill={rgb, 255:red, 255; green, 255; blue, 255 }  ,fill opacity=1 ] (311.25,212.03) .. controls (311.25,209.59) and (313.23,207.62) .. (315.67,207.62) .. controls (318.11,207.62) and (320.08,209.59) .. (320.08,212.03) .. controls (320.08,214.47) and (318.11,216.45) .. (315.67,216.45) .. controls (313.23,216.45) and (311.25,214.47) .. (311.25,212.03) -- cycle ;
			\draw [line width=1.5]    (250,212.37) -- (257.33,212.37) ;

		\end{tikzpicture}
		
		\caption{This picture explains the homotopy equivalence between $C_2(\C^*)/S^1$ and $S^1$: first of all note that $C_2(\C^*)/S^1$ is homotopy equivalent to $C_2(\C^*)/\C^*$. Now observe that any configuration $\{z_1,z_2\}\in C_2(\C^*)$ is equivalent (up to rotations and dilations) to a configuration of the form $\{1,z\}$, where $z\in D^2-\{0,1\}$. If $\abs{z_1}\neq \abs{z_2}$ there is a unique representative $\{1,z\}$ of the class $[z_1,z_2]\in C_2(\C^*)/\C^*$. In the case $\abs{z_1}=\abs{z_2}$ there are two representatives of $[z_1,z_2]$: $\{1,z_1z_2^{-1}\}$ and $\{1,z_2z_1^{-1}\}$. Therefore $C_2(\C^*)/\C^*$ is homeomorphic to the space obtained from $D^2-\{0,1\}$ by gluing the boundary of the disk according to the relation $z\sim z^{-1}$. This space is homeomorphic to $S^2$ minus two points, so we can conclude that $C_2(\C^*)/\C^*$ is homotopy equivalent to a circle. }
		\label{fig:equivalenza omotopica con il cerchio}
	\end{figure}
	
	\begin{sseqpage}[classes = {draw = none }, xscale = 0.6, yscale=0.8, no ticks]

		\class["1"](0,0)
		\class["1"](0,1)
		\class["1"](0,2)
		\class["1"](0,3)
		\class["0"](1,0)
		\class["0"](1,1)
		\class["0"](1,2)
		\class["0"](1,3)
		\class["1"](2,0)
		\class["1"](2,1)
		\class["1"](2,2)
		\class["1"](2,3)
		\class["0"](3,0)
		\class["0"](3,1)
		\class["0"](3,2)
		\class["0"](3,3)
		\class["1"](4,0)
		\class["1"](4,1)
		\class["1"](4,2)
		\class["1"](4,3)
		\class["0"](5,0)
		\class["0"](5,1)
		\class["0"](5,2)
		\class["0"](5,3)
		\class["1"](6,0)
		\class["1"](6,1)
		\class["1"](6,2)
		\class["1"](6,3)
		\class["\cdots"](7,0)
		\class["\cdots"](7,1)
		\class["\cdots"](7,2)
		\class["\cdots"](7,3)
		\node[background] at (3.5,-1.5) {H^*_{S^1}(C_5(\C);\F_2)};
		
	\end{sseqpage}
	\qquad
	\begin{sseqpage}[classes = {draw = none }, xscale = 0.6, yscale=0.8, no ticks]
		\class["1"](0,0)
		\class["2"](0,1)
		\class["1"](0,2)
		\class["0"](1,0)
		\class["0"](1,1)
		\class["0"](1,2)
		
		\class["1"](2,0)
		\class["2"](2,1)
		\class["1"](2,2)
		
		\class["0"](3,0)
		\class["0"](3,1)
		\class["0"](3,2)
		
		\class["1"](4,0)
		\class["2"](4,1)
		\class["1"](4,2)
		
		\class["0"](5,0)
		\class["0"](5,1)
		\class["0"](5,2)
		
		\class["1"](6,0)
		\class["2"](6,1)
		\class["1"](6,2)
		
		\class["\cdots"](7,0)
		\class["\cdots"](7,1)
		\class["\cdots"](7,2)
		\node[background] at (3.5,-1.5) {H^*_{S^1}(C_5(\C)^{\Z/2};\F_2)};
	\end{sseqpage}
	
	As we know, both the spectral sequences degenerate at the second page (Theorem \ref{thm: calcolo dell'omologia equivariante degli spazi di configurazione quando p divide n} and Remark \ref{oss: sequenza spettrale dei punti fissi degenera}) and summing over the diagonals we get the ranks of $H^i_{S^1}(C_5(\C))$ and $H^i_{S^1}(C_5(\C)^{\Z/2})$. By Lemma \ref{lem:mono in coomologia equivariante} the inclusion $i: C_5(\C)^{\Z/2}\hookrightarrow C_5(\C)$ induces a monomorphism $i^*:H^i_{S^1}(C_5(\C);\F_2)\to H^i_{S^1}(C_5(\C)^{\Z/2};\F_2)$ in each degree, so we can conclude that it is an isomorphism for any $i\geq 2$ by looking at the ranks. If we put all these information in the Mayer-Vietoris sequence we conclude immediately that
	\[
	H^i(C_5(\C)/S^1;\F_2)=0 \quad \text{ for all } i\geq 3
	\]
	Now $C_5(\C)/S^1$ is simply connected, so the first (co)homology group is zero. By the same argument as before using the Euler characteristic we can conclude that $H^2(C_5(\C)/S^1;\F_2)=0$ as well. So there is no $2$-torsion in $H_*(C_5(\C)/S^1;\Z)$.
	
	\item $\F_3$-coefficients: by Theorem \ref{thm:omologia quoziente stretto e omotopico sono uguali} $H_*(C_5(\C)/S^1;\F_3)\cong H_*^{S^1}(C_5(\C);\F_3)$. Theorem \ref{thm:calcolo omologia equivariante quando p non divide n o n-1} tells us that $H_*^{S^1}(C_5(\C);\F_3)$ is the subspace of $H_*(C_5(\C);\F_3)$ spanned by the classes which do not contain the bracket $[a,a]$. In this case $H_*(C_5(\C);\F_3)$ has only two classes: $a^5$ of degree zero and $a^3[a,a]$ of degree one, so we can conclude that $H_*^{S^1 }(C_5(\C);\F_3)$ is trivial. Hence there is no $3$-torsion in $H_*(C_5(\C)/S^1;\Z)$.
	
	\item $\F_5$-coefficients: By Theorem \ref{thm: quoziente stretto di p punti non ha omologia modulo p} $H^i(C_5(\C)/S^1;\F_5)=0$ for any $i\geq 1$. Thus there is no $5$-torsion in $H_*(C_5(\C)/S^1;\Z)$. 
\end{itemize}

\paragraph{n=6:} as we said at the beginning of this paragraph, $C_6(\C)/S^1$ is the first non contractible space. In particular its homology with $\F_3$-coefficients will be not trivial. By Theorem \ref{thm: bound sulla torsione dei quozienti} the order of any class in $H_*(C_6(\C)/S^1;\Z)$ divides $6!$. So the only interesting coefficients for homology are $\F_2$, $\F_3$ and $\F_5$. From the cellular model described in Section \ref{sec: a combinatorial model using cacti} we know that $C_6(\C)/S^1$ is homotopy equivalent to a CW-complex of dimension $4$, therefore $H^i(C_6(\C)/S^1;\Z)=0$ for any $i\geq 5$. Moreover $C_6(\C)/S^1$ is simply connected and its Euler characteristic is equal to one, so we get $0=b_2-b_3+b_4$. 
\begin{itemize}
	\item $\F_5$-coefficients: by Theorem \ref{thm: quoziente stretto di p punti non ha omologia modulo p} $H^i(C_6(\C)/S^1;\F_5)=0$ for any $i\geq 1$, so there is no $5$-torsion in $H_*(C_5(\C)/S^1;\Z)$. 
	\item $\F_3$-coefficients: consider the Mayer-Vietoris sequence
	\[
	\begin{tikzcd}[column sep=small]
		& \cdots\arrow[r] & H^{i}(\frac{C_6(\C)}{S^1})\arrow[r]&  H^i(\frac{C_6(\C)^{\Z/3}}{S^1})\oplus H^i_{S^1}(C_6(\C))\arrow[r]&H^{i}_{S^1}(C_6(\C)^{\Z/3})\arrow[r] & \cdots
	\end{tikzcd}
	\]
	Now observe that $C_6(\C)^{\Z/3}$ is homeomorphic $C_2(\C^*)$, whose homology with $\F_3$-coefficients is generated by the classes listed below:
	\[
	\begin{tabular}{cc}
		
		\toprule
		Homology class  & Degree\\
		\midrule
		$ba^2$ & $0$ \\
		$b[a,a]$ & $1$ \\
		$[b,a]a$ & $1$ \\
		$[[b,a],a]$ & $2$\\
		
		\bottomrule
	\end{tabular}
	\]
	As we already know $C_2(\C^*)/S^1$ is homotopy equivalent to a circle, so it has no classes of degrees greater that two. The second pages of the Serre spectral sequences that compute $H^*_{S^1}(C_6(\C);\F_3)$ and $H^*_{S^1}(C_6(\C)^{\Z/3};\F_3)$ are displayed below:

	\begin{sseqpage}[classes = {draw = none }, xscale = 0.6, yscale=0.8, no ticks]

		\class["1"](0,0)
		\class["1"](0,1)
		\class["0"](0,2)
		\class["0"](0,3)
		\class["1"](0,4)
		\class["1"](0,5)

		\class["0"](1,0)
		\class["0"](1,1)
		\class["0"](1,2)
		\class["0"](1,3)
		\class["0"](1,4)
		\class["0"](1,5)
		
		\class["1"](2,0)
		\class["1"](2,1)
		\class["0"](2,2)
		\class["0"](2,3)
		\class["1"](2,4)
		\class["1"](2,5)
		
		\class["0"](3,0)
		\class["0"](3,1)
		\class["0"](3,2)
		\class["0"](3,3)
		\class["0"](3,4)
		\class["0"](3,5)
		
		\class["1"](4,0)
		\class["1"](4,1)
		\class["0"](4,2)
		\class["0"](4,3)
		\class["1"](4,4)
		\class["1"](4,5)
		
		\class["0"](5,0)
		\class["0"](5,1)
		\class["0"](5,2)
		\class["0"](5,3)
		\class["0"](5,4)
		\class["0"](5,5)
		
		\class["1"](6,0)
		\class["1"](6,1)
		\class["0"](6,2)
		\class["0"](6,3)
		\class["1"](6,4)
		\class["1"](6,5)
		
		\class["\cdots"](7,0)
		\class["\cdots"](7,1)
		\class["\cdots"](7,2)
		\class["\cdots"](7,3)
		\class["\cdots"](7,4)
		\class["\dots"](7,5)
		
		\node[background] at (3.5,-1.5) {H^*_{S^1}(C_6(\C);\F_3)};
		
	\end{sseqpage}
	\qquad
	\begin{sseqpage}[classes = {draw = none }, xscale = 0.6, yscale=0.8, no ticks]
		\class["1"](0,0)
		\class["2"](0,1)
		\class["1"](0,2)

		\class["0"](1,0)
		\class["0"](1,1)
		\class["0"](1,2)
		
		\class["1"](2,0)
		\class["2"](2,1)
		\class["1"](2,2)
		
		\class["0"](3,0)
		\class["0"](3,1)
		\class["0"](3,2)
		
		\class["1"](4,0)
		\class["2"](4,1)
		\class["1"](4,2)
		
		\class["0"](5,0)
		\class["0"](5,1)
		\class["0"](5,2)
		
		\class["1"](6,0)
		\class["2"](6,1)
		\class["1"](6,2)

		\class["\cdots"](7,0)
		\class["\cdots"](7,1)
		\class["\cdots"](7,2)
		
		\node[background] at (3.5,-1.5) {H^*_{S^1}(C_6(\C)^{\Z/3};\F_3)};
	\end{sseqpage}
	
	As we know, both the spectral sequences degenerate at the second page and summing over the diagonals we get the ranks of $H^i_{S^1}(C_6(\C))$ and $H^i_{S^1}(C_6(\C)^{\Z/3})$. By Lemma \ref{lem:mono in coomologia equivariante} the inclusion $i: C_6(\C)^{\Z/3}\hookrightarrow C_6(\C)$ induces a monomorphism $i^*:H^i_{S^1}(C_6(\C);\F_3)\to H^i_{S^1}(C_6(\C)^{\Z/3};\F_3)$ in each degree, so we can conclude that it is an isomorphism for any $i\geq 4$ by looking at the ranks. If we put all these information in the Mayer-Vietoris sequence we conclude immediately that
	\[
	H^i(C_6(\C)/S^1;\F_3)=\begin{cases}
		0 \quad \text{ for all } i\geq 5 \text{ and } i=1,2 \\
		\F_3 \quad \text{ for } i=0,3,4\\
		
	\end{cases}
	\]
	Thus $C_6(\C)/S^1$ is not contractible, since  $H^*(C_6(\C)/S^1;\F_3)\neq 0$ for $*=3,4$.
\end{itemize}

\appendix
\section{The homology of $C_n(\C^*)$}\label{app: omologia spazi di configurazione etichettati }
In this appendix we briefly review some classical results about the homology of $C_n(\C^*)$ that were used in Section \ref{sec:examples}. 
	\begin{defn}[\cite{Bodigheimer}]
	Let $M$ be a manifold and $(X,\ast)$ be a based CW-complex, not necessarily connected. The space of configurations in $M$ with labels in $X$ is defined as 
	\[
	C(M;X)\coloneqq\bigsqcup_{n\in\N}F_n(M)\times_{\Sigma_n} X^n/\sim
	\]
	where $(p_1,\dots,p_n;x_1,\dots,x_n)\sim (p_1,\dots,\hat{p}_i,\dots,p_n;x_1,\dots,\hat{x}_i,\dots,x_n)$ if $x_i=\ast$.
\end{defn}
When $M=\R^n$ the homology of $C(\R^n;X)$ is known thanks to the work of Fred Cohen \cite{Cohen}, who computed $H_*(C(\R^n;X);\F_p)$ in terms of $H_*(X;\F_p)$. For the purpose of this paper we focus on the case $X=S^0\vee S^0=\{\ast,a,b\}$, where $\ast$ is the basepoint. In this case $C(\C;S^0\vee S^0)$ is a disjoint union of components $\{C_{n\circ+m\bullet}(\C)\}_{n,m\in\N}$ consisting of all configurations of $m$ black particles and $n$ white particles, where by \emph{black particle} (resp. \emph{white particle}) we mean a point in the plane with label $a$ (resp. $b$). A tipical element of $C(\C;S^0\vee S^0)$ is depicted in Figure \ref{fig:black and white particles}.
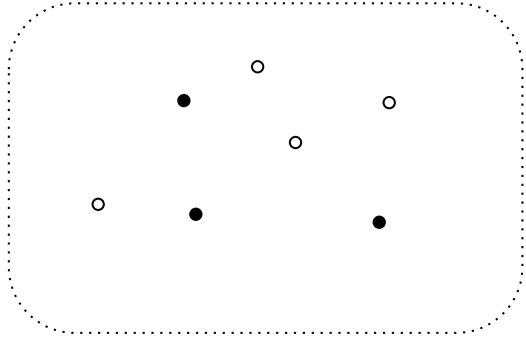
\begin{figure}
	\centering
	
	\tikzset{every picture/.style={line width=0.75pt}} 
	
	\begin{tikzpicture}[x=0.75pt,y=0.75pt,yscale=-1,xscale=1]
		
		\draw  [dash pattern={on 0.84pt off 2.51pt}] (98,132.07) .. controls (98,113.8) and (112.8,99) .. (131.07,99) -- (322.6,99) .. controls (340.86,99) and (355.67,113.8) .. (355.67,132.07) -- (355.67,231.27) .. controls (355.67,249.53) and (340.86,264.33) .. (322.6,264.33) -- (131.07,264.33) .. controls (112.8,264.33) and (98,249.53) .. (98,231.27) -- cycle ;
		\draw   (220,130.83) .. controls (220,129.27) and (221.27,128) .. (222.83,128) .. controls (224.4,128) and (225.67,129.27) .. (225.67,130.83) .. controls (225.67,132.4) and (224.4,133.67) .. (222.83,133.67) .. controls (221.27,133.67) and (220,132.4) .. (220,130.83) -- cycle ;
		\draw   (140,199.83) .. controls (140,198.27) and (141.27,197) .. (142.83,197) .. controls (144.4,197) and (145.67,198.27) .. (145.67,199.83) .. controls (145.67,201.4) and (144.4,202.67) .. (142.83,202.67) .. controls (141.27,202.67) and (140,201.4) .. (140,199.83) -- cycle ;
		\draw   (286,148.83) .. controls (286,147.27) and (287.27,146) .. (288.83,146) .. controls (290.4,146) and (291.67,147.27) .. (291.67,148.83) .. controls (291.67,150.4) and (290.4,151.67) .. (288.83,151.67) .. controls (287.27,151.67) and (286,150.4) .. (286,148.83) -- cycle ;
		\draw   (239,168.83) .. controls (239,167.27) and (240.27,166) .. (241.83,166) .. controls (243.4,166) and (244.67,167.27) .. (244.67,168.83) .. controls (244.67,170.4) and (243.4,171.67) .. (241.83,171.67) .. controls (240.27,171.67) and (239,170.4) .. (239,168.83) -- cycle ;
		\draw  [fill={rgb, 255:red, 0; green, 0; blue, 0 }  ,fill opacity=1 ] (183,147.83) .. controls (183,146.27) and (184.27,145) .. (185.83,145) .. controls (187.4,145) and (188.67,146.27) .. (188.67,147.83) .. controls (188.67,149.4) and (187.4,150.67) .. (185.83,150.67) .. controls (184.27,150.67) and (183,149.4) .. (183,147.83) -- cycle ;
		\draw  [fill={rgb, 255:red, 0; green, 0; blue, 0 }  ,fill opacity=1 ] (189,204.83) .. controls (189,203.27) and (190.27,202) .. (191.83,202) .. controls (193.4,202) and (194.67,203.27) .. (194.67,204.83) .. controls (194.67,206.4) and (193.4,207.67) .. (191.83,207.67) .. controls (190.27,207.67) and (189,206.4) .. (189,204.83) -- cycle ;
		\draw  [fill={rgb, 255:red, 0; green, 0; blue, 0 }  ,fill opacity=1 ] (281,208.83) .. controls (281,207.27) and (282.27,206) .. (283.83,206) .. controls (285.4,206) and (286.67,207.27) .. (286.67,208.83) .. controls (286.67,210.4) and (285.4,211.67) .. (283.83,211.67) .. controls (282.27,211.67) and (281,210.4) .. (281,208.83) -- cycle ;
		
	\end{tikzpicture}
	
	\caption{In this picture we see a point of $C_{4\circ+3\bullet}(\C)$.}
	\label{fig:black and white particles}
\end{figure}

\begin{oss}\label{oss: gli Z/p punti fissi sono omeomorfi a configurazioni sul piano bucato}
	$C_n(\C^*)$ is homotopy equivalent to the configuration space of $n$ black particles and one white particle in the plane. Therefore $H_*(C_n(\C^*);\F_p)$ will be the subspace of $H_*(C(\C,S^0\vee S^0);\F_p)$ spanned by those classes that involve only $n$ black particles and one white particle. 
\end{oss}
Let us denote by $a$ (resp. $b$) the class in $H_0(S^0\vee S^0)$ which represent a black particle (resp. a white particle). The following Proposition just identifies $H_*(C_n(\C^*);\F_p)$ as a subspace of $H_*(C(\C,S^0\vee S^0);\F_p)$ (for details see \cite{Rossi1}):
\begin{prop}\label{prop:omologia nelle configurazioni nel piano bucato} Let $n\in\N$ and $p$ be a prime. Then
	\[
	H_*(C_n(\C^*);\F_p)=\bigoplus_{k=0}^nb_k\cdot H_*(C_{n-k}(\C);\F_p)
	\]
	where 
	\begin{align*}
		& b_0\coloneqq b\\
		& b_1\coloneqq[b,a]\\
		&b_2\coloneqq[[b,a],a]\\
		&b_3\coloneqq[[[b,a],a],a]\\
		&\cdots
	\end{align*}
\end{prop}

\printbibliography
\end{document}